\documentclass[12pt]{amsart}

\usepackage{fullpage}
\usepackage{mathtools}
\usepackage{amssymb,amsmath,amsthm,amscd,mathrsfs,graphicx}
\usepackage[dvipsnames]{xcolor}
\usepackage{bm}
\usepackage{hyperref}
\usepackage{dsfont}
\usepackage{enumerate}
\usepackage{epsfig}
\usepackage{float, graphicx}
\usepackage{latexsym, amsxtra}
\usepackage{pdfpages}
\usepackage{multicol}
\usepackage[normalem]{ulem}
\usepackage{psfrag}
\usepackage{stmaryrd}
\usepackage{tikz}
\usepackage[T1]{fontenc}
\usepackage{url}
\usepackage{verbatim}
\usepackage{indentfirst}
\usepackage{tikz-cd}
\usepackage{url}
\usepackage[all,pdf]{xy}
\usepackage{color}
\usepackage{verbatim}

\usepackage{cleveref}

\makeatletter
\def\thm@space@setup{%
  \thm@preskip=2ex \thm@postskip=2ex
}
\makeatother

\hypersetup{hidelinks}

\numberwithin{equation}{subsection}
\theoremstyle{plain}

\newtheorem{thm}[subsection]{Theorem~}
\newtheorem{lem}[subsection]{Lemma~}
\newtheorem{prop}[subsection]{Proposition~}

\newtheorem{pro}[subsection]{Proposition~}

\newtheorem{cor}[subsection]{Corollary~}
\newtheorem{ass}[subsection]{Assumption~}
\newtheorem{defn}[subsection]{Definition~}

\theoremstyle{remark}
\newtheorem{rmk}[subsection]{Remark~}

\newcommand\s[1]{\mathscr{#1}}
\renewcommand\b[1]{\mathbb{#1}}
\newcommand\f[1]{\mathfrak{#1}}
\newcommand\m[1]{\mathrm{#1}}
\renewcommand\c[1]{\mathcal{#1}}
\renewcommand\d{\dagger}
\newcommand{\Q}{\mathbb{Q}}
\newcommand{\D}{\mathbb{D}}
\newcommand{\gr}{\mathrm{gr}\,}
\newcommand{\coh}{\mathrm{coh}}
\newcommand{\gen}{\mathrm{gen}}

\newcommand{\Spm}{\mathrm{Spm}\,}

\newcommand{\Spec}{\mathrm{Spec}\,}
\newcommand{\RANK}{d!\int_{\Delta_{\infty}\cap\mathfrak C} 
	\prod_{\alpha \in R^+} \frac{(\lambda, \alpha)^2}{(\rho, \alpha)^2}\mathrm d\lambda}
\newcommand{\unif}{\varpi}
\newcommand{\FsM}{\hat{M}}
\newcommand{\an}{\m{an}}
\newcommand{\BBmm}{R\langle x\rangle}
\newcommand{\nn}{\c{N}'}
\newcommand{\Ninf}{\c{N}}

\newcommand{\important}{}%\textcolor{blue}}
\newcommand{\actM}{\beta}
\newcommand{\actV}{\alpha}

\newcommand{\algclo}{\bar{k}}
\newcommand{\morZ}{\varphi}

\usepackage{todonotes}

\title{$p$-adic hypergeometric $\s{D}^\dagger(\infty)$-module and exponential sums on reductive groups}

\author[X. Li]{Xuanyou Li}
\address{Tsinghua University, Beijing 100084, P. R. China}
\email{lixuanyo21@mails.tsinghua.edu.cn}

\author[C. Liu]{Chenhan Liu}
\address{Tsinghua University, Beijing 100084, P. R. China}
\email{liu-ch22@mails.tsinghua.edu.cn}

\date{}
\begin{document}
\bibliographystyle{amsalpha}

\begin{abstract}
We study the $p$-adic analogue of the $\ell$-adic hypergeometric sheaves for reductive groups, called the hypergeometric
$\s{D}^\dagger(\infty)$-modules. They are overholonomic objects in the derived category of arithmetic $\s{D}$-modules with Frobenius
structures. Over the
non-degenerate locus, the hypergeometric $\s{D}^\dagger(\infty)$-modules define $F$-isocrystals overconvergent along the
complement of the
non-degenerate locus. As an application, we use the theory of $L$-functions of overholonomic arithmetic $\s{D}$-modules to
study hypergeometric exponential sums on reductive groups.

\noindent\textbf{Key words:} arithmetic $\s{D}$-module, spherical variety, Fourier transformation.
\end{abstract}

\maketitle

\section*{Introduction}

Let $p$ be a prime number, $q$ a power of $p$, $k$ the finite field with $q$ elements, $\psi:k\to\overline{\b{Q}}^*_p$ a nontrivial additive character, $G_k$ a reductive group over the finite field $k$, and \[\rho_{j,k}:G_k\to\m{GL}(V_{j,k})\quad
(j=1, \cdots, N)\] a family of representations of $G_k$, where $V_{j,k}$ are finite dimensional vector spaces over $k$. For any $k$-point $A=(A_1,\cdots,A_N)$ of $\prod_{j=1}^N\m{End}(V_{j,k})$, consider the morphism
$$f:G_k\to\b{A}_k^1:g\mapsto\sum_{j=1}^N\m{Tr}(A_j\rho_j(g)).$$ It defines a regular function on $G_k$. In this paper, we use the $p$-adic method to study the exponential sum \[\m{Hyp}(A):=\sum_{g\in G_k(k)}\psi(\sum_{j=1}^N\m{Tr}(A_j\rho_{j,k}(g))),\] where $G_k(k)$ is the set of $k$-points of $G_k$. We call it the \textit{hypergeometric exponential sum} associated with the representations $\rho_{j,k}$.

In \cite{l-adic}, the authors study this exponential sum using the theory of $\ell$-adic sheaves and algebraic $\s{D}$-modules.
In  \cite{padictorus}, the authors study the twisted GKZ hypergeometric exponential sum by the $p$-adic method, which corresponds to the
case where $G_k=\b{G}_{m,k}^n$ is a torus.

\subsection{Basic setting}\label{setting}
Let $R$ be a Dedekind domain so that its fraction field $K$
is of characteristic $0$ and its residue fields at maximal ideals are finite,
$G_R$ a split reductive group scheme over $R$, $B_R$ a
Borel subgroup scheme of $G_R$, and $T_R$ the maximal torus in $B_R$. Fix an algebraic closure $\bar K$ of $K$.
By \cite[Lemme 2, 3.2 b) and Lemme 5]{Serre}, irreducible representations of
$G_{\bar K}=G_R\otimes_R{\bar K}$ can be lifted to representations of $G_R$.
Let \[\rho_{j,R}: G_R\to\mathrm{GL}(V_{j,R})\quad(j=1, \ldots, N)\] be representations such that the representations
$\rho_{j, \bar K}: G_{\bar K}\to \mathrm{GL}(V_{j,\bar K})$ are irreducible, where
$V_{j,R}$ are projective $R$-modules of finite ranks, and $V_{j, \bar K}:=V_{j, R}\otimes_R \bar K$.

Let $\b{V}_R:=\prod_{j=1}^N\m{End}(V_{j,R})$, $n:=\mathrm{rank}_R(\b{V}_R)$. Consider the morphism
$$\iota_0': G_R\to \b{V}_R,\quad g\mapsto(\rho_{1,R}(g),\cdots,\rho_{N,R}(g)),$$ we assume $\iota_0'$ is quasi-finite. % Fix a character $\chi: G\to \b{G}_m$. Pullback the Kummer F-isocrystal $\c{K}=\c{K}_{1/(p-1)}$ on $\b{G}_m$ %Cohomologie rigide et cohomologie rigide à supports propres 2.3.8
%along $\chi$, we get a overconvergent F-isocrystal $\c{K}_{\chi}$ on $G_{}$.
Let $\b{A}_R^1$ be the affine line over $R$, $\b{P}_R^1$ the projective line over $R$,
$\b{P}_R=\b{P}(\b{V}_R\times_R \b{A}^1_R)$ the projective space that contains $\b{V}_R$ as an open subset,
$Y_R'$ the scheme theoretic image of the composite $$G_R\to \b{V}_R\to\b{P}_R,$$
$Y_R$ the integral closure of $Y_R'$ in the fraction field of $G_R$.
Let $\iota'$ be the canonical morphism $$\iota': Y_R\to \b{P}_R.$$ We assume it is finite. We denote $Y^{\m{aff}}_R=\iota'^{-1}(\b V_R)$. It is an open affine subscheme of $Y_R$.
Let $H_R=G_R\times_R G_R$. By definition, $Y_{K}$ is a proper spherical variety for the reductive
group $H_K$, and $G_K$ is the open $H_K$-orbit in $Y_K$. If we replace the subscript $R$ of a scheme by $K$, it means taking tensor product of the scheme with $K$ over $R$.

Associated with the split reductive group scheme $G_R$ and
the representations \[\rho_{j,R}: G_R\to\mathrm{GL}(V_{j,R})\quad(j=1, \ldots, N),\] we have the group scheme
$G_R\times_R\mathbb{G}_{m,R}$ and
the representations $$\rho'_{j,R}: G_R\times_R \mathbb{G}_{m,R}\to\mathrm{GL}(V_{j,R}), \quad (g,t)\mapsto t\rho_j(g)$$
together with the representation $$\rho'_{0,R}: G_R\times_R\mathbb{G}_{m,R}\to \mathrm{GL}(V_{0,R}),\quad (g,t)\mapsto t$$
with $V_{0,R}=R$, where $\b{G}_{m,R}=\m{Spec}(R[t^{\pm}])$.
As above, we can define a quasi-finite morphism $$G_R\times_R\b{G}_{m,R}\to \b{V}_R\times_R \b{A}_R^1.$$
Let $Z_R'$ be the scheme theoretic image of the morphism $$\morZ'_0:G_R\times_R\b{G}_{m,R}\to \b{P}_R\times_R \b{P}^1_R,$$ and
let $Z_R$ be the integral closure of $Z_R'$ in the fraction field of $G_R\times_R\b{G}_{m,R}$. Let $\morZ'$ be the canonical morphism $$\morZ': Z_R\to \b{P}_R\times_R \b{P}^1_R.$$ It is finite. We denote $Z^{\m{aff}}_R=\morZ'^{-1}(\b V_R\times_R\b A_R^1)$. It is an open affine subscheme of $Z_R$.
Then
$Z_{K}$ is a proper spherical variety for the reductive group $H_K\times_K \b{G}_{m,K}$, and $G_K\times_K\b{G}_{m,K}$ is
the open $H_K$-orbit in $Z_K$.

\begin{pro}\label{YfiberofZ} The $R$-algebra homomorphism $$R[t]\to R,\quad t\mapsto 1$$ induces a morphism $\mathrm{Spec}\,R\to \mathbb A_R^1$. Composing this morphism with the open immersion $\mathbb A_R^1\hookrightarrow \mathbb P_R^1$, we obtain a morphism $1: \mathrm{Spec}\,R\to \mathbb P_R^1$. Let $\pi$ be the composite
$$Z_R\to \b{P}_R\times_R \b{P}^1_R\to \b{P}^1_R.$$ Consider the following diagram \[\begin{tikzcd} G_R\ar[r]\ar[d] & Z_{R,1}\ar[r]\ar[d] & \b{P}_R\ar[r]\ar[d] & \m{Spec}(R)\ar[d,"1"] \\ G_R\times_R\b{G}_{m,R}\ar[r] & Z_R\ar[r] & \b{P}_R\times_R\b{P}_R^1\ar[r] & \b{P}_R^1 \end{tikzcd}\] with all squares Cartesian. We have an isomorphism between $Y_R$ and $Z_{R,1}$.
\end{pro}

\begin{proof}
    Note that $Z_{R,1}\to \b{P}_R$ is finite and the composition $G_R\to Z_{R,1}\to \b{P}_R$ coincides with the composition $G_R\xrightarrow{\iota_0'}\b{V}_R\hookrightarrow\b{P}_R$. To prove the proposition, it suffices to show that $G_R$ is dense in $Z_{R,1}$ and $Z_{R,1}$ is normal.
    The $\b{G}_{m,R}$-action on $Z_R$ induces an isomorphism
    \[Z_{R,1}\times_R \b{G}_{m,R}\to \pi^{-1}(\b{G}_{m,R}).\]
    The open immersion $G_R\times_R \b{G}_{m,R}\to \pi^{-1}(\b{G}_{m,R})$ is $\b{G}_{m,R}$-equivariant, hence the fact that $G_R\times_R \b{G}_{m,R}$ is dense in $\pi^{-1}(\b{G}_{m,R})$ implies that $G_R$ is dense in $Z_{R,1}$. Note that the projection $Z_{R,1}\times_R \b{G}_{m,R}\to Z_{R,1}$ is faithfully flat. Since $\pi^{-1}(\b{G}_{m,R})$ is normal, $Z_{R,1}$ is also normal by \cite[\href{https://stacks.math.columbia.edu/tag/033G}{Tag 033G}]{stacksproject}.
\end{proof}

To find a reasonable resolution of $Z_R\otimes_R(R/\f{p})$ for some maximal ideal $\f{p}$ of $R$, we need to delete finitely many primes so that the following assumption holds.

\begin{ass}\label{ass}
We assume that there exists a smooth projective %toroidal
$R$-scheme $\tilde{Z}_R$ with $(H_R\times_R \b{G}_{m,R})$-action and an equivariant birational morphism $p: \tilde{Z}_R\to Z_R$ such that:
\begin{enumerate}
    \item\label{ass-0}$p|_{p^{-1}(G_R\times_R\b{G}_{m,R})}$ is an isomorphism,
    and $\tilde{Z}_R\to \Spec R$ has  geometrically irreducible fibers.

    %\item $Z\to \Spec R$ has normal fibers. $G\subset Y$ is fiberwise dense.
    %\item \label{ass-1} $R^ip_*\omega_{\tilde{Z}}=0$ for $i>0$.

    %\item \label{ass-2} $R^0p_*\omega_{\tilde{Z}}$ has Cohen-Macaulay fibers over $R$.
\end{enumerate}
Let $\tilde{\pi}$ be the composition $\tilde{Z}_R\to \b{P}_R\times_R\b{P}^1_R\to \b{P}^1_R$. We define $\tilde{Y}_R$ by the Cartesian diagram \[\begin{tikzcd} \tilde{Y}_R\ar[r]\ar[d] & \m{Spec}(R)\ar[d,"1"] \\ \tilde{Z}_R\ar[r,"\tilde{\pi}"] & \b{P}_R^1 \end{tikzcd}\] Then $\tilde{Y}_R$
    is a proper scheme over $R$ with $H_R$-action. The $\b{G}_{m,R}$-action on $\tilde{Z}_R$ induces an isomorphism
    \[\tilde{Y}_R\times_R \b{G}_{m,R}\to \tilde{\pi}^{-1}(\b{G}_{m,R}).\]
    Note that $\tilde{Y}_R\times_R \b{G}_{m,R}\cong \tilde{\pi}^{-1}(\b{G}_{m,R})\to \tilde{Y}_R$ is a smooth cover.  %这一段依赖（1）
    By \cite[\href{https://stacks.math.columbia.edu/tag/036U}{Tag 036U}]{stacksproject}, $\tilde{\pi}^{-1}(\b{G}_{m,R})$ being smooth over $R$ implies $\tilde{Y}_R$ is smooth over $R$.
    Let $P_R=B_R\times_R B_R^-\subset H_R$ which is a smooth parabolic subgroup.

\begin{enumerate}
    \item[(2)] \label{ass-3}

    We assume that there exists an open subset $\tilde{Y}_{0,R}\subset \tilde{Y}_R$ admitting a $P_R$-equivariant isomorphism between $\tilde{Y}_{0,R}$ and $R_u(P_R)\times_R S_R$,
    where $S_R$ is a smooth toric variety for %a quotient torus of
    the split torus $T_R$, and $R_u(P_R)$ is the unipotent radical of $P_R$. %Let $\Delta_{T_R}$ be the diagonal of $T_R$ in $T_R\times_R T_R$, here we identify $T_R\times_R T_R/\Delta_{T_R}$ with $T_R$.
    $P_R$ acts on $R_u(P_R)\times_R S_R$ via
    \[(p_1,p_2).(r,s)=((p_1,p_2)r(p_1,p_2)^{-1}(u_1,u_2),t_1t_2^{-1}s)\]
    for $p_1 = u_1 t_1 \in B_R$, $p_2 = u_2 t_2 \in B_R^-$,  $r \in R_u(P_R)$ and $s \in S_R$, where $t_1, t_2 \in T_R$, $u_1 \in R_u(B_R)$,  $u_2 \in R_u(B_R^-)$.
    We assume  $H_R\tilde{Y}_{0,R}=\tilde{Y}_R$, $G_R\cap \tilde{Y}_{0,R}=R_u(P_R)\times_R T_R$.
\end{enumerate}
\end{ass}

\begin{pro}\label{ass-true-on-dense}
    After replacing $\Spec R$ by an open dense subset, \Cref{ass} is satisfied.
\end{pro}

We delay the proof of Proposition \ref{ass-true-on-dense} to \Cref{Appendix}.
In the following, we always make this assumption.

\subsection{Geometric Fourier transformation}
We recall some basic properties of the geometric Fourier transformation \cite{Huyghe2004Trans} here, only in the affine space case.

Fix a maximal ideal $\f{p}$ of $R$. Then $R_{\f{p}}:=\varprojlim_nR/\f{p}^n$ is a complete discrete valuation ring with fraction field $K$. Assume that there exists $\pi\in K$ such that $\pi^{p-1}=-p$. Let $\hat{\b{P}}$ be the formal projective space over $R_{\f{p}}$ of relative dimension $n$, $t_0,\cdots,t_n$ the canonical homogeneous coordinates of $\hat{\b{P}}$, $\infty$ the hyperplane defined by $t_n=0$, and $\hat{\b{V}}=\hat{\b{P}}-\infty$. Let $x_1,\cdots,x_n$ be the canonical coordinates of $\hat{\b{V}}$, and $\partial_1,\cdots,\partial_n$ the corresponding derivations. It is known that \[\Gamma(\hat{\b{P}},\s{D}_{\hat{\b{P}},\b{Q}}^\dagger(\infty))\cong\left\{\sum_{\underline{i},\underline{j}}a_{\underline{i},\underline{j}}\underline{x}^{\underline{i}}\underline{\partial}^{[\underline{j}]}\,\Big|a_{\underline{i},\underline{j}}\in K,\exists(c>0,\eta<1)\text{ such that }\forall \underline{i},\underline{j},\, |a_{\underline{i},\underline{j}}|\leq c\eta^{|\underline{i}|+|\underline{j}|}\right\}.\] We denote it by $D_{\hat{\b{P}},\b{Q}}^\dagger(\infty)$. The following result can be found \cite[5.3.3]{Ddaggeraffine}.

\begin{thm}
    The functor $R\Gamma(\hat{\b{P}},-)\cong\Gamma(\hat{\b{P}},-)$ from the category of coherent $\s{D}_{\hat{\b{P}},\b{Q}}^\dagger(\infty)$-modules to the category of coherent $D_{\hat{\b{P}},\b{Q}}^\dagger(\infty)$-modules is an equivalence of categories.
\end{thm}

The automorphism \[F_\pi:D_{\hat{\b{P}},\b{Q}}^\dagger(\infty)\to D_{\hat{\b{P}},\b{Q}}^\dagger(\infty),
\quad x_i\mapsto -\partial_i/\pi, \quad \partial_i\mapsto\pi x_i,\] is
called the \emph{naive Fourier transformation}. For any left coherent $\s{D}_{\hat{\b{P}},\b{Q}}^\dagger(\infty)$-module $\c{E}$, let
$E=\Gamma(\hat{\b{P}},\c{E})$, and let $$F_\pi(E):=D_{\hat{\b{P}},\b{Q}}^\dagger(\infty)\otimes_{F_\pi, D_{\hat{\b{P}},\b{Q}}^\dagger(\infty)}E.$$

We define the \emph{geometric Fourier transformation} by six functors.
For any left coherent $\s{D}_{\hat{\b{P}},\b{Q}}^\dagger(\infty)$-module $\c{E}$, we define \[\f{F}_\pi(\c{E})=p_{2,+}(\c{K}_\pi\otimes^L p_1^!(\c{E})[2-3n]),\] where $p_1,p_2:\hat{\b{P}}\times \hat{\b{P}}\to\hat{\b{P}}$ are the canonical projections, and the Fourier kernel
$\c{K}_\pi$ is defined in \cite{Huyghe2004Trans}. (Our definition of $\mathfrak F_\pi$ differs from that of Huyghe by a translation.) 
Huyghe proves the following.

\begin{thm}[Huyghe (\cite{Huyghe2004Trans})]
For any left coherent $\s{D}_{\hat{\b{P}},\b{Q}}^\dagger(\infty)$-module $\c{E}$, $\f{F}_\pi(\c{E})$ is a left coherent $\s{D}_{\hat{\b{P}},\b{Q}}^\dagger(\infty)$-module, and there exists a canonical $D_{\hat{\b{P}},\b{Q}}^\dagger(\infty)$-isomorphism $$\Gamma(\hat{\b{P}},\f{F}_\pi(\c{E}))\cong F_\pi(E).$$
\end{thm}

\subsection{Hypergeometric $\s{D}^\dagger(\infty)$-module}\label{Hyper-D}
Fix a maximal ideal $\f{p}$ of $R$.
Let $p$ be the characteristic of the residue field $k(\f{p})=R/\f{p}$,
$R_\f{p}=\varprojlim_n R/\f{p}^n$, $K_\f{p}=\m{Frac}(R_\f{p})$, and $\bar{K}_\f{p}$ an algebraic closure of $K_\f{p}$.
Choose $\pi\in \bar{K}_\f{p}$ such that $\pi^{p-1}=-p$. Consider the integral closure $R'_\f{p}$ of $R_\f{p}$ in $K_\f{p}(\pi)$.
By base change from $R$ to $R'$, we can always assume $\pi\in R_\f{p}$. Since $R_\f{p}$ and $R_\f{p}'$ have the same residue field $k(\f{p})$,  this base change does not have any effect on special fibers.

Let $\iota_0:\f{G}_\f{p}\to\hat{\b{V}}_\f{p}$ be the $\f{p}$-adic completion of the morphism $\iota_0':G_R\to\b{V}_R$ given by the representations $\rho_1,\cdots,\rho_N$ in \Cref{setting}. Similarly, let $\iota:\f{Y}_\f{p}\to \hat{\b{P}}_\f{p}$ be the $\f{p}$-adic completion of $\iota':Y_R\to\b{P}_R$. By \Cref{ass}, we have a resolution $\tilde{\f{Y}}_\f{p}$ of $\f{Y}_\f{p}$ such that the following diagram

\[\xymatrix{ & \f{G}_\f{p}\ar[r]^{\iota_0} \ar[d]^f \ar[ld]_{\tilde{f}} & \hat{\b{V}}_\f{p}\ar[d]^j \\ \tilde{\f{Y}}_\f{p}\ar[r]^p & \f{Y}_\f{p}\ar[r]^{\iota} & \hat{\b{P}}_\f{p}}\] is a commutative diagram, where $p$ is a proper birational morphism.

Let $d$ be the relative dimension of $G_R$, and let
$\pi_2:\tilde{\f{Y}}_\f{p}\times \hat{\b{P}}_\f{p}\to\hat{\b{P}}_\f{p}$ be the projection. We call
$$\m{Hyp}_+:=\pi_{2,+}((\iota p)\times\m{id})^!\c{K}_\pi$$ the \emph{$p$-adic hypergeometric $\s{D}_{\hat{\b{P}}_\f{p},\b{Q}}^\dagger(\infty_\f{p})$-module}, where $\infty_\f{p}:=\hat{\b{P}}_\f{p}-\hat{\b{V}}_\f{p}$ is an effective Cartier divisor. It can be regarded as the reduction of $\infty_R:=\b{P}_R-\b{V}_R$ in $k(\f{p})$.

Consider another effective Cartier divisor on $\b{P}_R$ denoted by $D_R=(\prod_{j=1}^N \m{det}_{j,R})_0+\infty_R$, where $\det_j$ is the polynomial given by the determinant of $V_{j,R} $, $(\prod_{j=1}^N \m{det}_{j,R})_0$ is the divisor defined by $\prod_{j=1}^N \m{det}_{j,R}=0$. Let $\tilde{D}_R$ be the preimage of $D_R$ in $\tilde{Y}_R$, which is an effective Cartier divisor on $\tilde{Y}_R$.

\begin{lem}\label{divisorD}
    $Y_R- D_R=G_R$. In other words, $(\iota')^{-1}(\prod_{j=1}^N\m{GL}(V_{j,R}))=G_R$.
\end{lem}

\begin{proof}
    Clearly $(\iota')^{-1}(\prod_j\m{GL}(V_{j,R}))\supset G_R$. We need to show $(\iota')^{-1}(\prod_j\m{GL}(V_{j,R}))-G_R=\emptyset$.
    It suffices to check this on each fiber over $\Spec R$. Let $x\in \Spec R$, $\bar{k}(x)$ an algebraic closure of the residue field of $x$, and $U_R=(\iota')^{-1}(\prod_j\m{GL}(V_{j,R}))$. By \Cref{ass}(\ref{ass-0}), $\tilde{Y}_{\bar{k}(x)}$ is irreducible. Then $G_{\bar{k}(x)}$ is dense in $\tilde{Y}_{\bar{k}(x)}$ since it is open. Therefore, we have $G_{\bar{k}(x)}$ is dense in $Y_{\bar{k}(x)}$. Hence \[\dim (U_{\bar{k}(x)}-G_{\bar{k}(x)})< \dim G_{\bar{k}(x)}.\] We conclude $\iota'(U_{\bar{k}(x)}-G_{\bar{k}(x)})\subset \prod_j\m{GL}(V_j)$ is a $\iota'(G_{\bar{k}(x)})$-invariant subset of dimension less than $\dim G_{\bar{k}(x)}$. If $U_{\bar{k}(x)}-G_{\bar{k}(x)}\ne \emptyset$, choose $x_0\in \iota'(U_{\bar{k}(x)}-G_{\bar{k}(x)})$. Then $\iota'(U_{\bar{k}(x)}-G_{\bar{k}(x)})$ contains $\iota'(G_{\bar{k}(x)})x_0\cong \iota'(G_{\bar{k}(x)})$ since $x_0$
  is invertible. So we have \[\dim \iota'(U_{\bar{k}(x)}-G_{\bar{k}(x)})\geq \dim \iota'(G_{\bar{k}(x)})=\dim G_{\bar{k}(x)},\] which is a contradiction.
\end{proof}

We can identify $\mathbb V_R$ with its dual vector bundle $\mathbb V_R^*$ via the pairing
$$\langle\;,\;\rangle: \prod_{j=1}^N\mathrm{End}(V_{j,R})\times \prod_{j=1}^N\mathrm{End}(V_{j,R}) \to \mathbb A^1_R,
\quad ((A_1, \ldots, A_N), (B_1, \ldots, B_N))\mapsto \sum_{j=1}^N \mathrm{Tr}(A_jB_j).$$

In addition, we define $\m{Hyp}_!=\b{D}_{\infty_\f{p}}(\m{Hyp}_+)$ as the \textit{$p$-adic hypergeometric $\s{D}_{\hat{\b{P}}_\f{p},\b{Q}}^\dagger(\infty_\f{p})$-module with proper support}, here $\b{D}_{\infty_\f{p}}$ is the Verdier duality functor on $D_\m{coh}^b(\s{D}^\d_{\hat{\b{P}}_\f{p},\b{Q}}(\infty_\f{p}))$.

Define $\c{O}_{G_{k(\f{p})}}$ to be the $\s{D}^\d_{\tilde{\f Y},\Q}$-module $\c{O}_{G_{k(\f{p})}}:=(^\d \tilde{D}_{k(\f{p})})\c{O}_{\tilde{\f{Y}}, \Q}$.
In \Cref{Fourier}, we will prove \[\f{F}_\pi((\iota p)_+\c{O}_{G_{k(\f{p})}})\cong\m{Hyp}_+[2-n-d].\]
For any closed point $i_A:A\to \b{V}_{\f{p}}$, we will prove \[\m{Tr}_K(\m{Fr},i_A^+(\m{Hyp}_!))= \m{Hyp}(A)\] in \Cref{exponential}, where $\m{Fr}$ is the Frobenius action on $i_A^+(\m{Hyp}_!)$.

To study the hypergeometric exponential sum $\m{Hyp}(A)$, we need to study $(\iota p)_+\c{O}_{G_{k(\f{p})}}$ and its geometric Fourier transformation.

\subsection{Nondegeneracy}\label{nondegeneracy}
Fix an algebraic closure $\bar{K}$ of $K$. The subscript $\bar{K}$ means taking tensor product of the scheme with $\bar{K}$ over $K$. Let $\Lambda=\mathrm{Hom}(T_{\bar K}, \mathbb G_{m,\bar K})$ be the weight lattice,
$W=N_{G_{\bar K}}(T_{\bar K})/T_{\bar K}$ the Weyl group, and
$\lambda_j$ $(j=1, \ldots, N)$ the maximal weight of $\rho_{j,\bar K}$. We define the \emph{Newton
polytope at infinity} $\Delta_\infty$ for the family $\rho_{1,\bar K}, \ldots, \rho_{N,\bar K}$ to be the convex hull in
$\Lambda_{\mathbb R}:=\Lambda\otimes \mathbb R$ of $W(0, \lambda_1, \ldots, \lambda_N)$.
Occasionally we also use the \textit{Newton polytope} $\Delta$ which is defined to be the convex hull of $W(\lambda_1,\cdots,\lambda_N)$.
For any face $\tau\prec \Delta_\infty$, let $e(\tau)=(e(\tau)_j)\in\prod_{j=1}^N \mathrm{End}(V_{j,R})$
be defined as follows: Let
$$V_{j,R}=\bigoplus_{\lambda\in \Lambda} V_{j,R}(\lambda)$$ be the decomposition so that
$V_{j,R}(\lambda)$ is the component of weight $\lambda$ under the action of $T_R$. We define $e(\tau)_j$ to be the block-diagonal
linear transformation so that
$$e(\tau)_j\Big|_{ V_{j,R}(\lambda)}=\left\{\begin{array}{cl}
\mathrm{id}_{V_{j,R}(\lambda)}&\hbox{if }\lambda\in \tau, \\
0&\hbox{otherwise}.
\end{array}\right.$$
Let $\bar k(\f{p})$ be an algebraic closure of the residue field $k(\f{p})$. For any $\bar k(\f{p})$-point $A=(A_1, \ldots, A_N)$ in $\mathbb V_R=\prod_{j=1}^N\mathrm{End}(V_{j,R})$, let $f_A$ be the morphism
$$f_A:\mathbb V_{\bar{k}(\f{p})}\to
\mathbb A^1_{\bar{k}(\f{p})}\,\quad f_A(B)=\sum_{j=1}^N\mathrm{Tr}(A_jB_j).$$
We have an action of
\begin{align*}(G_{\bar{k}(\f{p})}\times_{\bar{k}(\f{p})} G_{\bar{k}(\f{p})})\times_{k(\f{p})} \mathbb V_{\bar{k}(\f{p})}\to & \mathbb V_{\bar{k}(\f{p})}, \\
\Big((g, h), (A_1, \ldots, A_N)\Big)\mapsto & \Big(\rho_1(g) A_1 \rho_1(h^{-1}), \ldots, \rho_N(g) A_N\rho_N(h^{-1})\Big).\end{align*}
Consider the Laurent polynomial 
\begin{equation}\label{generalLaurent}
f: G_{\bar{k}(\f{p})}\to\mathbb A^1_{\bar{k}(\f{p})}, \quad f(g)=\sum_{j=1}^N \mathrm{Tr}(A_j \rho_j(g)).
\end{equation}
We say that the Laurent polynomial (\ref{generalLaurent})
is \emph{nondegenerate} if for any face $\tau$ of $\Delta_\infty$ not containing the origin, %the restriction of $f_A$ to the orbit $G_{\bar K}e(\tau) G_{\bar K}$ has no critical points, that is,
the function
\begin{equation}\label{ftauA}
f_{\tau, A}:G_{\bar{k}(\f{p})}\times_{\bar{k}(\f{p})} G_{\bar{k}(\f{p})}
\to\mathbb A^1_{\bar{k}(\f{p})}, \quad f_{\tau, A}(g, h)=\sum_{j=1}^N \mathrm{Tr}(A_j \rho_j(g)e(\tau)_j\rho_j(h))
\end{equation}
has no critical point, which means that $\mathrm df_{\tau, A}=0$ has no solution
on $G_{\bar{k}(\f{p})}\times_{\bar{k}(\f{p})} G_{\bar{k}(\f{p})}$. Let $\mathbb V_{k(\f{p})}^{\mathrm{gen}}$ be the complement of the Zariski closure of the set
consisting of those $A$ so that the Laurent polynomial (\ref{generalLaurent}) is nondegenerate.
Then $\mathbb V_{k(\f{p})}^{\mathrm{gen}}$ is a Zariski open subset of $\mathbb V_{k(\f{p})}$ parametrizing nondegenerate Laurent polynomials.

The main theorem of this paper is the following.
\begin{thm}\label{mainthm}
    For almost all prime $\f{p}$,
    the Fourier transform $\f{F}_{\pi}((\iota p)_{+}\c{O}_{G_{k(\f{p})}})$ restricted to $\mathbb V_{k(\f{p})}^{\mathrm{gen}}$ is $\c{O}_{\hat{\mathbb V}_\f{p}, \Q}$-coherent. %, hence locally free. It
    Moreover, let $\m{sp}$ be the canonical specialization map. Then for any divisor $T$ of $\b{V}_{k(\f{p})}$ containing the complement of $\mathbb V_{k(\f{p})}^{\mathrm{gen}}$,
    the associated isocrystal $\m{sp}^*(\f{F}_{\pi}((\iota p)_{+}\c{O}_{G_{k(\f{p})}})|_{\b{V}_{k(\f{p})}-T})$ on $\mathbb V_{k(\f{p})}-T$ overconvergent along $T$
    has rank less or equal to
    \[d!\int_{\Delta_{\infty}\cap\mathfrak C}
	\prod_{\alpha \in R^+} \frac{(\lambda, \alpha)^2}{(\rho, \alpha)^2}\mathrm d\lambda.\]
    Here $d = \dim G_{k(\f{p})}$, $\f{C}$ is the dominant Weyl chamber in $\Lambda_\b{R}$, $\m{d}\lambda$ is the Lebesgue measure on $\Lambda_\b{R}$ normalized by $\m{vol}(\Lambda_\b{R}/\Lambda) = 1$, $(-,-)$ is the pairing defined by the Killing form, $R_+$ is the set of positive roots, and $\rho=\frac12\sum_{\alpha\in R_+}\alpha$.
\end{thm}

This paper is organized as follows. In \Cref{case0}, we introduce a new method to study the (modified) hypergeometric $\s{D}$-module
introduced in \cite{K,l-adic}.
Unlike \cite{K,l-adic}, we do not use Brylinski's theorem, but rather construct a good filtration more directly. In \Cref{Dmodule}, we introduce the $p$-adic modified hypergeometric $\s{D}^\dagger(\infty)$-module. 
We study its relation with the $\s{D}^\dagger(\infty)$-module $\mathrm{Hyp}_+$.
In \Cref{sec-right-inv}, we define some right invariant differential operators on $H$ in the arithmetic $\s{D}$-module setting, and prove some related technical results. %which is used to define the 
In \Cref{pcase}, by imitating the methods in section 1, we prove the main theorem \ref{mainthm}. We choose finite level integral models of the modified hypergeometric  $\s{D}^\dagger(\infty)$-module explicitly, and use it to bound the generic rank of the modified hypergeometric $\s{D}^\dagger(\infty)$-module.
In \Cref{exponential}, we use our results to estimate the exponential sum $\m{Hyp}(\bar{A})$.
In \Cref{Appendix}, we study compactifiactions of $G_R$, or more precisely, good resolutions of $Y_R$ and $Z_R$ defined over the Dedekind domain $R$.

\textbf{Acknowledgements:}
The authors would like to thank their Ph.D advisor professor Lei Fu for suggesting the problem in this paper, and for helpful discussions and the suggestion of details.

\section{Characteristic $0$ case}\label{case0}

\subsection{}
In this section we only work over the generic point $\Spec K$ of $\Spec R$.  %All of the schemes in this section are $K$-schemes.
We omit all subscripts of $K$-schemes when we use the notation in the introduction. %Also, we just need to consider in affine space in this section. 

Notation as in \Cref{setting}. Let $\mathfrak g$ be the Lie algebra of $G$, and let 
$\f{L}(H)=\mathfrak g\times \mathfrak g$ be the Lie algebra of $H=G\times G$. 
For any $\xi=(\xi_1,\xi_2)\in \f{L}(H)$, let $L_{\xi}$ be the vector field on $\b{V}$ defined by $$L_\xi(A)= \frac{\mathrm d}{\mathrm dt}\Big|_{t=0}\exp(t \xi_1)A\exp(-t\xi_2).$$
In \cite[5.2]{K}, Kapronov defined the 
\emph{$A$-hypergeometric left $\s{D}_{\b{V}}$-module} $\c{M}'$ to be the Fourier transform of the left $\s{D}_{\b{V}}$-module
\[\mathcal N'=(\s{D}_{\mathbb V}\otimes_{\mathcal{O}_{\mathbb V}}\mathcal O_{Y^{\m{aff}}})\big
/\sum_{\xi\in\mathfrak L(H)}(\s{D}_{\mathbb V}\otimes_{\mathcal{O}_{\mathbb V}}
\mathcal O_{Y^{\m{aff}}})L_\xi.\]
Here $\mathcal O_{Y^{\m{aff}}}$ is regarded as a coherent $\mathcal{O}_{\mathbb V}$-module via the functor $\iota'_*$.
Actually, Kapronov only defined the $A$-hypergeometric left $\s{D}_{\b{V}}$-module $\mathcal M'$ when $\iota_0'$ is an immersion. In this case, $\mathcal M'$ can be identified with
\begin{eqnarray}\label{hypDmod}
	\mathcal M'=\s{D}_{\mathbb V}\Big/\Big(\sum_{f\in I} \s{D}_{\mathbb V}P_f+
	\sum_{\xi\in\f{L}(H)}\s{D}_{\mathbb V} (L_\xi-\chi(\xi))\Big),
\end{eqnarray}
where $I$ is the ideal of $Y'^{\m{aff}}:=Y'\cap\b{V}$ in $\b{V}$, $P_f$ is the corresponding differential operator for any $f\in I$, $\chi:H\to\b{G}_m$ is a character. Also, he only studied this module under the homogeneity condition. With our method, this condition can be relaxed.

In \cite{l-adic}, it is proved that $R\iota_{!}\omega_{G}$ is a direct factor of the right $\s{D}_{\b{V}}$-module
$$\mathcal N=(\omega_{Y^{\m{aff}}}\otimes_{\mathcal{O}_{\mathbb V}}\s{D}_{\mathbb V}
)\big/\sum_{\xi\in\mathfrak L(H)}L_\xi (\omega_{Y^{\m{aff}}}\otimes_{\mathcal{O}_{\mathbb V}}
\s{D}_{\mathbb V}).$$
The \emph{modified hypergeometric right $\s{D}_{\b{V}}$-module}  $\mathcal M$ is defined to be the Fourier transform of $\mathcal N$: 
$$\mathcal M=\mathcal F_{\mathcal L}(\mathcal N).$$

The main theorems of this section are the following. 
\begin{thm}\label{cha0mainthm} 
	Notation as above. The modified hypergeometric right  $\s{D}_{\b{V}}$-module $\mathcal M$ is holonomic. 
	The restriction to $\mathbb V^{\mathrm{gen}}_K$ of 	$\mathcal M$ is a connection. We have
	$$\mathrm{rank}\,(\mathcal M\big|_{\mathbb V_K^{\mathrm{gen}}})\leq d!\int_{\Delta_{\infty}\cap\mathfrak C} 
	\prod_{\alpha \in R^+} \frac{(\lambda, \alpha)^2}{(\rho, \alpha)^2}\mathrm d\lambda.$$
	The same conclusion holds for the Fourier transform of $R\iota_{!}\omega_{G}$.
\end{thm}

\begin{thm}\label{cha02ndmainthm} 
	Notation as above. The A-hypergeometric left $\s{D}_{\b{V}}$-module $\mathcal M'$ is holonomic. 
	The restriction to $\mathbb V^{\mathrm{gen}}_K$ of 
	$\mathcal M'$ is a connection. We have
	$$\mathrm{rank}\,(\mathcal M'\big|_{\mathbb V_K^{\mathrm{gen}}})\leq d!\int_{\Delta_{\infty}\cap\mathfrak C} 
	\prod_{\alpha \in R^+} \frac{(\lambda, \alpha)^2}{(\rho, \alpha)^2}\mathrm d\lambda.$$
\end{thm}

\subsection{}Fix an algebraic closure $\bar{K}$ of $K$. The subscript $\bar{K}$ means taking tensor product of the scheme with $\bar{K}$ over $K$. Let's recall the constructions in \cite[\S 5]{Popov}.
Let $A$ be a commutative $\bar{K}$-algebra, equipped with an action by the algebraic group $G_{\bar{K}}$. Let
\[A_0\subset A_1\subset\cdots\subset A_n\subset\cdots\]
be a $G$-stable filtration on $A$ such that $$\bar{K}\subset A_0,\quad A=\cup_n A_n,\quad  A_nA_m\subset A_{n+m}.$$ Consider the algebra
\[S(A)=\bigoplus_{n=0}^{\infty}A_nt^{n}\subseteq A[t].\]
We extend the action of $G_{\bar{K}}$ on $A$ to an action of $G_{\bar{K}}$ on $A[t]$ by requiring $\bar{K}[t]\subset A[t]^{G_{\bar{K}}}$. $\mathbb{G}_{m,\bar{K}}$ acts on $A[t]$ via
$$A[t]\to A[t]\otimes \bar{K}[t,t^{-1}],\quad at^n\mapsto at^n\otimes t^n.$$
Clearly these two actions commute and preserve $S(A)$. The monomorphism $\bar{K}[t]\hookrightarrow S(A)\hookrightarrow A[t]$ gives rise to a commutative diagram of schemes with $(G_{\bar{K}}\times \mathbb{G}_{m,\bar{K}})$-actions
\[\begin{tikzcd}
	\mathrm{Spec}\;A\times \mathbb{A}^1_{\bar{K}} \arrow[rd, "\mathrm{pr}_2"]\arrow[rr, "\tau"]  &              & \mathrm{Spec}\;S(A) \arrow[ld, "\pi"] \\
	& \mathbb{A}^1_{\bar{K}}.&                                                         
\end{tikzcd}\]

\begin{pro}[{\cite[proposition 9]{Popov}}]\label{popovmain}
	The morphism $\pi$ is surjective and flat, and $\tau$ induces an isomorphism
	$\tau: \mathrm{pr}_2^{-1}(\mathbb{G}_{m,\bar{K}})\stackrel\cong\to \pi^{-1}(\mathbb{G}_{m,\bar{K}})$. In particular, for any $a\in \mathbb{G}_{m,\bar{K}}(\bar{K})$, we have $\pi^{-1}(\{a\})\cong \mathrm{Spec}\;A$. Moreover, we have $\pi^{-1}(\{0\})\cong \mathrm{Spec}\;\mathrm{gr}\,A$.
\end{pro}

%\begin{rmk}This proposition also holds on characteristic $p$ case. We can use \cite[Theorem 13]{Popovcharp}\end{rmk}

\subsection{}\label{functorS}

Let $M$ be a $G_{\bar{K}}$-equivariant $A$-module provided with a $G_{\bar{K}}$-stable filtration
$$\cdots \subset M_n\subset M_{n+1}\subset \cdots$$ such that
$$M=\cup_n M_n, \quad A_nM_m\subset M_{m+n}.$$
%We assume the $G$-action on $M$ is compatible with the $G$-action on $A$. 
We define
\[S(M)=\bigoplus_{n=-\infty}^{\infty}M_nt^{n}\subseteq M[t].\]
$S(M)$ is an $S(A)$-module. It has no torsion as a $\bar{K}[t]$-module. So $S(M)$ is flat over $\bar{K}[t]$. As usual we let $G_{\bar{K}}$ act trivially on $t$ and let $\mathbb{G}_{m,\bar{K}}$ act on $S(M)$ via
$$S(M)\to S(M)\otimes \bar{K}[t,t^{-1}],\quad mt^n\mapsto mt^n\otimes t^n.$$
The $(G_{\bar{K}}\times \mathbb{G}_{m,\bar{K}})$-action on $S(M)$ is compatible with the $(G_{\bar{K}}\times \mathbb{G}_{m,\bar{K}})$-action on $S(A)$.

On can check that the restriction of the sheaf $S(M)^\sim$ to $\pi^{-1}(\mathbb{G}_{m,\bar{K}})$ can be identified with the 
restriction of the sheaf $M[t]^\sim$ to $\mathrm{pr}_2^{-1}(\mathbb{G}_{m,\bar{K}})$. For any  $a\in \mathbb{G}_{m,\bar{K}}(\bar{K})$, we have
$$S(M)\otimes_{\bar{K}[t]}\bar{K}[t]/(t-a)\cong M, \quad   S(M)\otimes_{\bar{K}[t]}\bar{K}[t]/(t)\cong \mathrm{gr}\,M.$$ Note that when $\mathrm{gr}\,A$ is noetherian, the filtration on $M$ is a good filtration if and only if $S(M)$ is a finite $S(A)$-module.

\begin{lem}\label{filmod}
	The functor $S(\cdot)$ from the category of $G_{\bar{K}}$-equivariant filtered $A$-module to the category of $(G_{\bar{K}}\times \mathbb{G}_{m,\bar{K}})$-equivariant $S(A)$-modules flat over $\bar{K}[t]$ is an equivalence of categories. The quasi-inverse of $S(\cdot)$ can be chosen as the functor $-\otimes_{\bar{K}[t]}\bar{K}[t]/(t-1)$.
\end{lem}

\begin{proof}
	Let $N$ be a $(G_{\bar{K}}\times \mathbb{G}_{m,\bar{K}})$-equivariant $S(A)$-module flat over $\bar{K}[t]$. Assume $N=\oplus_{n\in \mathbb{Z}}N_n$, where $\mathbb{G}_{m,\bar{K}}$ acts on $N_n$ via the character $\lambda\to \lambda^n$. Since $N$ is flat over $\bar{K}[t]$, multiplication by $t$ defines injections $N_n\to N_{n+1}$. Let $M=\varinjlim_{n}N_n$, and let $M_n$ be the image of $N_n$ in $M$. % We have $M\cong N\otimes_{\bar{K}[t]}\bar{K}[t]/(t-1)$.
    For any $a\in A_m$, multiplication by $at^m$ defines a map $N_n\to N_{n+m}$, and hence a map $M\to M$. This map only depends on $a\in A$ and is independent of $m\in \mathbb{Z}_{\geq 0}$. So we get an $A$-module structure on $M$. $M_n$ defines a filtration on $M$ and is compatible with the filtration on $A$. The $G_{\bar{K}}$-action on $N$ defines an action on each $N_n$, and hence an action on $M$. %We thus get a functor of inverse direction, one can check that this functor is the quasi-inverse of $S$.
\end{proof}

\subsection{}\label{filonY}

Consider the $\bar{K}$-algebra $\bar{K}[\b V_{\bar{K}}]$ with the $H_{\bar{K}}$-action defined in \Cref{setting}. Let $F_i\bar{K}[\b V_{\bar{K}}]$ 
be the space of polynomials with degree less than or equal to $i$. It is an $H_{\bar{K}}$-stable filtration on $\b V_{\bar{K}}$. %linear action
We have $S(\bar{K}[\b V_{\bar{K}}])=\bar{K}[\b V_{\bar{K}}\times\b{A}^1_{\bar{K}}]$.

Regard $\bar{K}[Z^{\m{aff}}_{\bar{K}}]$ as a $\bar{K}[\b V_{\bar{K}}\times\b{A}^1_{\bar{K}}]$-module via the morphism $\morZ'$. It is flat over $\bar{K}[t]$. % It is an integral domain. We have an injective ring homomorphism $\bar{K}[t]\to \bar{K}[Z^{\m{aff}}_{\bar{K}}]$ since the morphism $Z^{\m{aff}}_{\bar{K}}\to\b{A}^1_{\bar{K}}$ is surjective. Then $\bar{K}[Z^{\m{aff}}_{\bar{K}}]$ is torsion-free as $\bar{K}[t]$-module, hence flat over the Dedekind domain $\bar{K}[t]$.
Recall that $Y^{\m{aff}}_{\bar{K}}$ is isomorphic to the fiber of $Z^{\m{aff}}_{\bar{K}}$ at $1$ by \Cref{YfiberofZ}. Then by \Cref{filmod}, $\bar{K}[Z^{\m{aff}}_{\bar{K}}]$ defines an $H_{\bar{K}}$-equivariant good filtration on $\bar{K}[Y^{\m{aff}}_{\bar{K}}]$, % it may not a filtration as ring. I'm not sure.
where $\bar{K}[Y^{\m{aff}}_{\bar{K}}]$ is regarded as a $\bar{K}[\b V_{\bar{K}}]$-module via the morphism $\iota'$, and good filtration means that $\m{gr}(\bar{K}[Y^{\m{aff}}_{\bar{K}}])$ is a coherent $\m{gr}^F(\bar{K}[\b{V}])$-module.
%and $\c{O}_{Z^{\m{aff}}_{\bar{K}}}$ defines an $H$-equivariant good filtration on $\c{O}_{Y^{\m{aff}}_{\bar{K}}}$. 
We have $S(\bar{K}[Y^{\m{aff}}_{\bar{K}}])=\bar{K}[Z^{\m{aff}}_{\bar{K}}]$.

\subsection{}
Recall in \Cref{ass} we have an $(H\times\mathbb{G}_m)$-equivariant birational morphism $p: \tilde{Z}\to Z$. Let $\tilde{Z}^{\m{aff}}:=p^{-1}(Z^{\m{aff}})$. As usual we define $\omega_{Z^{\m{aff}}}=p_*\omega_{\tilde{Z}^{\m{aff}}}$. It is endowed with an $(H\times\mathbb{G}_m)$-action. %Recall that $Y^{\m{aff}}$ is isomorphic to the fiber of $Z^{\m{aff}}$ at $1$. 
Let $\tilde{Y}^{\m{aff}}:=\tilde{Z}^{\m{aff}}\times_{ Z^{\m{aff}}}Y^{\m{aff}}$, let $p_1:\tilde{Y}^{\m{aff}}\to Y^{\m{aff}}$ be the canonical morphism,
and let $\omega_{Y^{\m{aff}}}=p_{1*}\omega_{\tilde{Y}^{\m{aff}}}$. The projection $p:\tilde{Z}^{\m{aff}}\to Z^{\m{aff}}$ is $\mathbb{G}_m$-equivariant. We may identify $p|_{\tilde\pi^{-1}(\mathbb{G}_m)}$ with the morphism $$p_1\times \mathrm{id}_{\mathbb{G}_m}: \tilde{Y}^{\m{aff}}\times \mathbb{G}_m\to Y^{\m{aff}}\times \mathbb{G}_m.$$ %, where $\tilde{\pi}$ is the composition $\tilde{Z}\to \b{P}\times\b{P}^1\to \b{P}^1$
	Since $\tilde{Z}^{\m{aff}}$ is smooth and $p$ is birational, $\tilde{Y}^{\m{aff}}$ is smooth and $p_1$ is birational. So the restriction of  $\omega_{Z^{\m{aff}}}=R^0p_*\omega_{\tilde{Z}^{\m{aff}}}$  to $\pi^{-1}(\mathbb{G}_m)$ is isomorphic to $$R^0(p_1\times \mathrm{id}_{\mathbb{G}_m})_*\omega_{\tilde{Y}^{\m{aff}}\times \mathbb{G}_m}\cong R^0(p_1\times \mathrm{id}_{\mathbb{G}_m})_*(\omega_{\tilde{Y}^{\m{aff}}}\boxtimes \mathcal{O}_{\mathbb{G}_m})\cong R^0p_{1*}\omega_{\tilde{Y}^{\m{aff}}}\boxtimes \mathcal{O}_{\mathbb{G}_m}.$$

\begin{pro}\label{filomega}
	$\omega_{Z^{\m{aff}}}$ is a Cohen-Macaulay $\mathcal{O}_{Z^{\m{aff}}}$-module flat over $\mathbb{A}^1$. In particular, if we regard $\omega_{Y^{\m{aff}}_{\bar{K}}}$ and $\omega_{Z^{\m{aff}}_{\bar{K}}}$ as $\c{O}_{\b V_{\bar{K}}}$ and $\c{O}_{\b V_{\bar{K}}\times\b{A}^1_{\bar{K}}}$ modules via the functors $\iota'_*$ and $\morZ_*'$, respectively, $\omega_{Z^{\m{aff}}_{\bar{K}}}$ defines an $H_{\bar{K}}$-equivariant good filtration on $\omega_{Y^{\m{aff}}_{\bar{K}}}$ by \Cref{filmod}.
    %For any $a\in \mathbb{G}_m(K)$, the fiber of $\omega_{Z^{\m{aff}}}$ at $a$ is isomorphic to $\omega_{Y^{\m{aff}}}$. 
\end{pro}

\begin{proof}
	By \Cref{ass-1}, $\omega_{Z^{\m{aff}}}=Rp_*\omega_{\tilde{Z}^{\m{aff}}}=R^0p_*\omega_{\tilde{Z}^{\m{aff}}}$ is Cohen-Macaulay. %since the normalization of $Z^{\m{aff}}$ has rational singularity.
    $\omega_{\tilde{Z}^{\m{aff}}}$ is an invertible sheaf on the smooth variety $\tilde{Z}^{\m{aff}}$. % in particular, $\tilde{Z}^{\m{aff}}$ is integral.
    The morphism $\tilde{Z}^{\m{aff}}\to \mathbb A^1$ is surjective. %implies injective ring homomorphism by integral and dominant
	So $\omega_{\tilde{Z}^{\m{aff}}}$ is flat over $\mathbb{A}^1$. Hence $\omega_{Z^{\m{aff}}}=R^0p_*\omega_{\tilde{Z}^{\m{aff}}}$ is flat over $\mathbb{A}^1$.
\end{proof}

\begin{prop}\label{orbitstr}
    Consider the projection $$Z'^{\m{aff}}_{\bar{K}}:=Z'_{\bar{K}}\cap(\b V_{\bar{K}}\times\b A^1_{\bar{K}})\to \b{A}^1_{\bar{K}}$$ Let $Z'^{\m{aff}}_{0,\bar{K}}$ be the fiber of $Z'^{\m{aff}}_{\bar{K}}$ at $0\in\b{A}^1_{\bar{K}}$. Every $(H_{\bar{K}}\times \mathbb{G}_{m,\bar{K}})$-orbit of $Z'^{\m{aff}}_{0,\bar{K}}$ has the form $(H_{\bar{K}}\times \mathbb{G}_{m,\bar{K}})\cdot e(\tau\times\{1\})$, where $e(\tau\times\{1\})$ is defined in \Cref{nondegeneracy}, $\tau\prec \Delta_{\infty}$ is a face not containing $0$. 
\end{prop}

\begin{proof}
    Consider the representations $\rho_1',\cdots,\rho_N'$ and $\rho_0'$ in \Cref{setting}. The corresponding Newton polytope $\Delta'$ is the convex hull of $\{(x,1):x\in \Delta\}$, and $(0,1)$. So $\Delta'=\Delta_{\infty}\times\{1\}$.  Let $\tau\times\{1\}\prec \Delta_{\infty}\times\{1\}$ be a face. Note that $e(\tau\times\{1\})$ lies in the fiber of $\mathbb{V}_{\bar{K}}\times\mathbb{A}^1_{\bar{K}}\to \mathbb{A}^1_{\bar{K}}$ at $0$ if and only if $0$ doesn't lie in $\tau$. It suffices to show that every $(H_{\bar{K}}\times\mathbb{G}_{m,\bar{K}})$-orbit of $Z'^{\m{aff}}_{\bar{K}}$ has the form $(H_{\bar{K}}\times \mathbb{G}_{m,\bar{K}})\cdot e(\tau\times\{1\})$. 
    
    By \cite[Lemma 3]{charporbit}, every $(H_{\bar{K}}\times \mathbb{G}_{m,\bar{K}})$-orbit of $Z'^{\m{aff}}_{\bar{K}}$ contains an idempotent element which belongs to the closure of $\morZ_0'(T_{\bar{K}}\times\mathbb{G}_{m,\bar{K}})$ in $Z'^{\m{aff}}_{\bar{K}}$. %Note that the $(H\times \mathbb{G}_m)$-action on $Z'$ is same as the $(G\times \mathbb{G}_m)\times(G\times \mathbb{G}_m)$-action.
    So we reduce to the toric case. %Note that $T\times\mathbb{G}_m\subseteq \morZ_0'^{-1}((\b{V}\times\b{A}^1)-\{0\})\subseteq Z'^{\m{aff}}$.
    Let $T'_{\bar{K}}$ be the image of the composite \[T_{\bar{K}}\times\mathbb{G}_{m,\bar{K}} \xrightarrow{\morZ'_0} Z'^{\m{aff}}_{\bar{K}}-\{0\}\to \b{P}(Z'^{\m{aff}}_{\bar{K}}).\] Then the closure $\bar{T}'_{\bar{K}}$ can be regarded as a projective toric variety with torus $T'_{\bar{K}}$. The closure of $\morZ'_0(T_{\bar{K}}\times\mathbb{G}_{m,\bar{K}})$ in $Z'^{\m{aff}}_{\bar{K}}$ is the union of $\{0\}$ and the preimage of $T'_{\bar{K}}$ in $Z'^{\m{aff}}_{\bar{K}}-\{0\}$. % this case satieties the homogeneous condition
    So it suffices to show for any $x\in \bar{T}'_{\bar{K}}$, there exists a face $\Gamma\prec \Delta_{\infty}\times\{1\}$ such that $x\in T'_{\bar{K}}\cdot e(\Gamma)$.

    We repeat the argument in \cite[chap5, 1.9]{toricorbit} by formal germs (see \cite[24]{Timashev}). For any $x\in \bar{T}'_{\bar{K}}$, there exists a formal germ $\eta\in \m{Hom}(\m{Spec} \,\bar{K}((t)),T'_{\bar{K}})$ such that $\lim_{t\to 0}\eta(t)\cdot e(\Delta_{\infty}\times\{1\})=x$, %just [1,1,\cdots,1]
    where $\bar{K}((t))$ denotes the ring of Laurent series in one variable with coefficients in $\bar{K}$. 
    Suppose that \[\eta(t)=(c_1t^{a_1}+\cdots,\ldots,c_lt^{a_l}+\cdots),\quad c_i\in \bar{K}^\times,a_i\in\b{Z},\quad i=1,\cdots,l,\]
    where $l=\dim T'_{\bar{K}}$, the dots mean terms of higher orders of $t$. Consider a linear function 
    \[\psi_a:\b{R}^{l}\to\b{R},\quad(b_1,\cdots,b_l) \mapsto\sum_{i=1}^la_ib_i.\]
    Then there exists a face $\Gamma_a\prec \Delta_{\infty}\times\{1\}$ such that at every point of $\Gamma_a$, the value of $\psi_a$ is equal to its minimum over the polytope $\Delta_{\infty}\times\{1\}$. Then
    \[\lim_{t\to 0}\eta(t)\cdot e(\Delta_{\infty}\times\{1\})=(c_1,\cdots,c_l)\cdot e(\Gamma_a).\qedhere\] % In fact, suppose minimum is $s$. We have $\eta(t)\cdot[1,1,\cdots,1]=\eta(t)\cdot[t^{-m},t^{-m},\cdots,t^{-m}]$. $\lambda\in\Gamma_a$ iff it matches the minimum iff the limit is $c_\lambda$, not zero.
\end{proof}

\begin{rmk}
    This proof works for any algebraic closed field of arbitrary characteristic. Moreover, by repeating the argument in \cite[2.5]{K}, the $(H_{\bar{K}}\times \mathbb{G}_{m,\bar{K}})$-orbits of $Z'^{\m{aff}}_{0,\bar{K}}$ are in one-to-one correspondence with the $W$-orbits on faces of $\Delta_{\infty}$ not containing $0$, where a face $\tau\prec \Delta_{\infty}$ corresponds to the orbit containing $e(\tau\times\{1\})$. But we don't need it in this paper.
\end{rmk}

\subsection{}\label{weakequiv} Suppose $X$ is a smooth $K$-variety provided with an action 
$$A:H\times X\to X.$$ 
Let $p_2: H\times X\to X$ be the projection. A \emph{weakly $H$-equivariant left $\s{D}_X$-module} is a pair $(K, \phi)$ consisting of a 
left $\s{D}_X$-module $\mathcal K$ and an $(\mathcal{O}_H\boxtimes \s{D}_X)$-module isomorphism
\[\phi: A^*\c{K}\to p_2^*\c{K}\]
satisfying the cocycle condition in \cite[1.14]{l-adic}.
A weakly $H$-equivariant $\s{D}_X$-module $(\c{K}, \phi)$ is $H$-equivariant if $\phi$ is a 
$(\s{D}_H\boxtimes \s{D}_X)$-module isomorphism. For any $\xi\in\mathfrak L(H)$, denote by $L_\xi^H$ and $L_\xi$ the vector fields 
$$L_\xi^H(h)=\frac{\mathrm d}{\mathrm dt}\Big|_{t=0}(\exp(t\xi)h), \quad 
L_\xi(x)=\frac{\mathrm d}{\mathrm dt}\Big|_{t=0}(\exp(t\xi)x)$$ on $H$ and $X$, respectively. For any section $s$ of 
a weakly $H$-equivariant left $\s{D}_X$-module $\mathcal K$, we define $D_{\xi}(s)$ to be the restriction of 
$(L^H_\xi\otimes1)\phi(A^*(s))\in p_2^*(\mathcal K)$ to $\{1\}\times X$, that is, 
\[D_{\xi}(s)=\lim_{t\to 0}\frac{\exp(-t\xi)s-s}{t}.\] %We can also define $D_\xi(s)$ for the section of an $H$-equivariant $\c{O}_X$-module
By the argument in \cite[2.11]{l-adic}, 
$\mathcal K$ is $H$-equivariant if and only if for any $s\in\mathcal K$ and $\xi\in \mathfrak L(H)$, we have
$D_{\xi}(s)=L_{\xi}(s)$.
Consider the morphism of left $\s{D}_X$-modules
$$\alpha: \mathfrak L(H)\otimes \mathcal K\to \mathcal K, \quad \xi\otimes s\mapsto L_{\xi}(s)-D_{\xi}(s).$$
By \cite[2.12]{l-adic}, there exists a weakly $H$-equivariant $\s{D}_X$-module structure on 
$\mathfrak L(H)\otimes \mathcal K$ so that $\alpha$ is equivariant, and the induced weakly $H$-equivariant $\s{D}_X$-module 
structure on $\m{coker}(\alpha)$ is $H$-equivariant. 
Clearly $\m{coker}(\alpha)$ is the \emph{maximal $H$-equivariant quotient} of $\mathcal{K}$.

\begin{lem}\label{passtoHequi}
	Let $\mathcal{K}$ be a weakly $H$-equivariant coherent left $\s{D}_X$-module and let $\overline{\mathcal{K}}$ be its maximal $H$-equivariant quotient. Assume $F_i\mathcal{K}$ is an $H$-equivariant good filtration on $\c{K}$. % $H$-equivariant object as $\c{O}_X$-module
    It induces a good filtration on $\overline{\mathcal{K}}$. Then the natural morphism
	\[\mathrm{gr}^F\;\mathcal{K}\twoheadrightarrow \mathrm{gr}^F\;\overline{\mathcal{K}}\]
	factors through 
	\[\mathrm{gr}^F\;\mathcal{K}\to \mathrm{gr}^F\;\mathcal{K}/\sum_{\xi\in\mathfrak L(H)}L_{\xi}(\mathrm{gr}^F\;\mathcal{K}).\]
	Here $L_{\xi}$ is regarded as a homogeneous element of degree $1$ in $\mathrm{gr}(\s{D}_X)$, which is just the %linear
	function on the cotangent bundle $T^*X$ defined by $(x,v)\mapsto \langle L_{\xi}(x),\,v\rangle$.
\end{lem}

\begin{proof}
	We need to show that for any $f\in F_i\mathcal{K}$, the image of $L_{\xi}f\in F_{i+1}\mathcal{K}$ in $$F_{i+1}\overline{\mathcal{K}}/F_{i}\overline{\mathcal{K}}=\dfrac{F_{i+1}\mathcal{K}}{F_{i}\mathcal{K}+F_{i+1}\mathcal{K}\cap\sum_{\xi\in\mathfrak L(H)}(L_{\xi}-D_{\xi})\mathcal{K}}$$ vanishes. Since $F_i\mathcal{K}$ is $H$-equivariant, we have $D_{\xi}f\in F_i\mathcal{K}$. Hence
	\[L_{\xi}f=D_{\xi}f+(L_{\xi}-D_{\xi})f\in F_{i}\mathcal{K}+F_{i+1}\mathcal{K}\cap\sum_{\xi\in\mathfrak L(H)}(L_{\xi}-D_{\xi})\mathcal{K}.\qedhere\]
\end{proof}

%\subsection{} 
Let $V$ be a $K$-vector space with a left $H$-action, regarded as a $K$-scheme. 

\begin{lem}\label{actFourier}
	Let $\mathcal{K}$ be a weakly $H$-equivariant coherent 
	left $\s{D}_V$-module. Then $\mathcal F_{\mathcal L}(\mathcal{K})$ has a natural weakly $H$-equivariant left $\s{D}_V$-module structure. The Fourier transform of the maximal $H$-equivariant quotient of $\mathcal{K}$ is the maximal $H$-equivariant quotient of $\mathcal F_{\mathcal L}(\mathcal{K})$. As a corollary, $\mathcal{K}$ is $H$-equivariant if and only if $\mathcal F_{\mathcal L}(\mathcal{K})$ is $H$-equivariant.
\end{lem}

\begin{proof}

    Let $V^*$ be the dual space of $V$. Equip $V^*$  with the dual $H$-action. %module structure.
    $\Gamma(V,\c{K})$ has an $H$-module structure given by its weakly $H$-equivariant structure. Note that $\Gamma(V^*, \mathcal F_{\mathcal L}(\mathcal{K}))$ has the same underlying $K$-vector space as $\Gamma(V, \mathcal{K})$. We define its $H$-module structure by 
    \[\Gamma(V^*, \mathcal F_{\mathcal L}(\mathcal{K}))=\Gamma(V, \mathcal{K})\otimes\det(V)^{-1}.\] % regard $V$ as a representation of $H$
    We thus get a weakly $H$-equivariant left $\s{D}_{V^*}$-module structure on $\mathcal F_{\mathcal L}(\mathcal{K})$. % i.e. we choose a weakly $H$-equivariant structure on $\c{F}_{\c{L}}(\c{K})$ satieties our condition, we means this is the natural structure.
    For any $\xi\in \mathfrak L(H)$, let $D_\xi', L_\xi'$ be the operators defined in \Cref{weakequiv} on $\Gamma(V^*, \mathcal F_{\mathcal L}(\mathcal{K}))$. 
    Then we have $D_\xi'=D_\xi-\m{tr}(\xi| V)$ and $L_\xi'=L_\xi-\m{tr}(\xi| V)$ under the identification
     of $\Gamma(V^*, \mathcal F_{\mathcal L}(\mathcal{K}))$ and $\Gamma(V,\c{K})$. Hence $L_\xi-D_\xi=L_\xi'-D_\xi'$. The claim follows.  
\end{proof}

Let $\c{C}$ be an $H$-equivariant coherent sheaf on $Y^{\m{aff}}$. Regard $\c{C}$ as a coherent $\mathcal{O}_{\mathbb V}$-module via the functor $\iota'_*$. We assume $F_i\c{C}$ is an $H$-equivariant good
filtration on the $\c{O}_{\b{V}}$-module $\c{C}$, where the filtration on $\c{O}_{\b{V}}$ is defined by the degree of polynomials, and good
filtration means that $\m{gr}^F(\c{C})$ is a coherent $\m{gr}^F(\c{O}_{\b V})$-module as in \Cref{filonY}.

The right multiplication on $\s{D}_{\b{V}}$ gives a right $\s{D}_{\b{V}}$-module structure on $\c{C}\otimes_{\mathcal{O}_{\mathbb{V}}}\s{D}_{\mathbb V}$. Note that $\c{C}\otimes_{\mathcal{O}_{\mathbb{V}}}\s{D}_{\mathbb V}\cong \c{C}\otimes_K K[\partial]$ as $K$-vector spaces, where $\partial=\{\partial_j\}_{j=1}^N$ is a basis of the derivations on $\b{V}$. This isomorphism induces a right $\s{D}_{\b{V}}$-module structure on $\c{C}\otimes_K K[\partial]$. More precisely, for any $g\in\c{O}_{\b{V}}$, $c\in\c{C}$ and % two monomials
$\underline{\partial}^{\underline{l}},\underline{\partial}^{\underline{r}}\in K[\partial]$, the action is defined by 
$$(c\otimes\underline{\partial}^{\underline{l}})g\underline{\partial}^{\underline{r}}=\sum_{\underline{l'}+\underline{l''}=\underline{l}}\binom{\underline{l}}{\underline{l'}}\underline{\partial}^{\underline{l'}}(g)c\otimes\underline{\partial}^{\underline{l''}}\underline{\partial}^{\underline{r}},$$ where $\underline{r},\underline{l},\underline{l'},\underline{l''}\in\b{Z}_{\geq0}^n$ are multi-indices.

\begin{lem}\label{filonFourier}
    We define a filtration on $\mathcal F_{\mathcal L}(\c{C}\otimes_{\mathcal{O}_{\mathbb{V}}}\s{D}_{\mathbb V})$ by letting $F_i\mathcal F_{\mathcal L}(\c{C}\otimes_{\mathcal{O}_{\mathbb V}}\s{D}_{\mathbb V})$ be the subsheaf corresponding to %the image of
    $F_i\c{C}\otimes_K K[\partial]$. It is an $H$-equivariant good filtration and
	\[\mathrm{gr}^F(\mathcal F_{\mathcal L}(\c{C}\otimes_{\mathcal{O}_{\mathbb V}}\s{D}_{\mathbb V}))\cong  \mathrm{gr}^F(\c{C})\otimes_KK[\mathbb{V}].\]
\end{lem}

\begin{proof}
	First let's check that the filtration $F_i\mathcal F_{\mathcal L}(\c{C}\otimes_{\mathcal{O}_{\mathbb V}}\s{D}_{\mathbb V})$ %is $\c{O}_{\b{V}}$-submodule of $\c{C}\otimes_{\mathcal{O}_{\mathbb V}}\s{D}_{\mathbb V}$ and it 
    is compatible with the order filtration of $D_{\mathbb V}$. Unwinding the definitions, we need to show
	\begin{enumerate}
		\item For any $g\in  K[\partial]$, $(F_i\c{C}\otimes_K K[\partial])g\subseteq F_i\c{C}\otimes_K K[\partial]$,
		
		\item For any $g\in K[\mathbb V]$ of degree $\leq j$, $(F_i\c{C}\otimes_K K[\partial])g\subseteq F_{i+j}\c{C}\otimes_K K[\partial]$.
	\end{enumerate}
	The first condition is clear. For the second, note that $K[\partial]g\subseteq (F_jK[\mathbb V])K[\partial]$ in $D_{\mathbb V}$. So by our definition of the $D_{\b{V}}$-action on $\c{C}\otimes_KK[\partial]$, we have
	\[(F_i\c{C}\otimes_K K[\partial])g\subseteq ((F_i\c{C})F_jK[\mathbb V])\otimes_K K[\partial]\subseteq F_{i+j}\c{C}\otimes_K K[\partial].\]
	Since $\c{F}_\c{L}$ is an exact functor, it is clear that
	\[\mathrm{gr}^F(\mathcal F_{\mathcal L}(\c{C}\otimes_{\mathcal{O}_{\mathbb V}}\s{D}_{\mathbb V}))\cong  \mathrm{gr}^F(\c{C})\otimes_KK[\mathbb{V}].\]
	In particular, $\mathrm{gr}^F(\mathcal F_{\mathcal L}(\c{C}\otimes_{\mathcal{O}_{\mathbb V}}\s{D}_{\mathbb V}))$ is a coherent $\mathrm{gr}(\s{D}_{\mathbb V})$-module. %since under this action, $\mathrm{gr}^F(\c{C})$ is a coherent $K[\bar{\partial}]$-module
    So we get a good filtration on $\mathcal F_{\mathcal L}(\c{C}\otimes_{\mathcal{O}_{\mathbb{V}}}\s{D}_{\mathbb V})$.
\end{proof}

Note that $\c{C}\otimes_{\mathcal{O}_{\mathbb V}}\s{D}_{\mathbb V}$ is a weakly $H$-equivariant right $\s{D}_{\b V}$-module. We can define $\c{N}(\c{C})$ as the maximal $H$-equivariant quotient of $\c{C}\otimes_{\mathcal{O}_{\mathbb V}}\s{D}_{\mathbb V}$. Then $\c{M}(\c{C}):=\c{F}_{\c L}(\c{N}(\c{C}))$ is the maximal $H$-equivariant quotient of $\mathcal F_{\mathcal L}(\c{C}\otimes_{\mathcal{O}_{\mathbb V}}\s{D}_{\mathbb V})$ by \Cref{actFourier}.

\begin{cor}\label{maincor}
	The $H$-equivariant good filtration on $\mathcal F_{\mathcal L}(\c{C}\otimes_{\mathcal{O}_{\mathbb V}}\s{D}_{\mathbb V})$ defined in \Cref{filonFourier} induces a good filtration on $\mathcal{M}(\c{C})$, denoted by $F_i\mathcal{M}(C)$. 
    Then $\mathrm{gr}^F(\mathcal M(\c{C}))$ is a quotient of \[\mathrm{gr}^F(\c{C})\otimes_KK[\mathbb{V}]\Big/\sum_{\xi\in\mathfrak L(H)}L_{\xi}(\mathrm{gr}^F(\c{C})\otimes_KK[\mathbb{V}]).\]
	Moreover, $\mathrm{gr}^F(\mathcal M(\c{C})_{\bar{K}})$ is supported in the zero section if we restrict to $ \mathbb{V}_{\bar{K}}^{\mathrm{gen}}$.
\end{cor}

\begin{proof}
	The first claim follows from \Cref{filonFourier} and \Cref{passtoHequi}. Let's prove the second claim by contradiction. Assume $(A,v)$ lies in 
	$\mathrm{supp}\, \mathrm{gr}^F(\mathcal M(\c{C})_{\bar{K}})\subseteq T^*\b V_{\bar{K}}$, where $A\in\mathbb{V}_{\bar{K}}^{\mathrm{gen}}$, $v\ne 0$. Then $v\in \mathrm{Supp}\;\mathrm{gr}^F(\c{C})\subseteq Z'^{\m{aff}}_{0,\bar{K}}$ and $\langle L_{\xi}(A),\,v\rangle=0$ for any $\xi\in\f{L}(H)$. By \Cref{orbitstr}, replacing $v$ by $\lambda v$ for some $\lambda\in\bar{K}$, %note that in \Cref{orbitstr} we consider $H\times G_m$-action
    we may assume $v=ge(\tau)h$ for a face $\tau\prec \Delta_{\infty}$ not containing $0$ and $g,h\in G_{\bar{K}}$. Then the condition $-\langle L_{\xi}(A),\,v\rangle=\langle A,\,L_{\xi}(v)\rangle=0$ implies $(g,h)$ is a critical point for the function $f_{\tau, A}$, which contradicts $A\in \mathbb{V}_{\bar{K}}^{\mathrm{gen}}$.
\end{proof}

\begin{proof}[Proof of \Cref{cha0mainthm} and \Cref{cha02ndmainthm}]
	We can work over the algebraic closure $\bar{K}$ of $K$. By the proof of \cite[2.13]{l-adic}, we have $\c{N}_{\bar{K}}=\c{N}(\omega_{Y^{\m{aff}}_{\bar{K}}})$ and $\c{N}'_{\bar{K}}=\c{N}(\c{O}_{Y^{\m{aff}}_{\bar{K}}})$. Hence $\mathcal N_{\bar{K}}, \mathcal N'_{\bar{K}}$ are $H_{\bar{K}}$-equivariant $\s{D}_{\b V_{\bar{K}}}$-modules with support $Y'^{\m{aff}}_{\bar{K}}$. Since $Y'^{\m{aff}}_{\bar{K}}$ only has finitely many $H_{\bar{K}}$-orbits, by \cite[II.5 Theorem]{Hotta}, $\mathcal N_{\bar{K}}, \mathcal N'_{\bar{K}}$ are holonomic. So their Fourier transforms are holonomic. 
	Endow $\mathcal O_{Y^{\m{aff}}_{\bar{K}}}$ %with the usual filtration
    and $\omega_{Y^{\m{aff}}_{\bar{K}}}$ with the filtrations defined in \Cref{filonY} and \Cref{filomega}, respectively. By \Cref{maincor}, $\c{M}_{\bar{K}},\c{M}'_{\bar{K}}$ are connections on $\mathbb{V}_{\bar{K}}^{\mathrm{gen}}$. It remains to calculate their ranks.
	
	For any $A\in \mathbb{V}_{\bar{K}}^{\mathrm{gen}}$, let $F_A$ be the linear subspace of $\mathbb V_{\bar{K}}$ defined by 
	$$F_A=\bigcap_{(\xi_1,\xi_2)\in\f{L}(H)}\Big\{B\in\mathbb V_{\bar{K}}%\prod_{j=1}^N\mathrm{End}(\mathbb{V}_{j,\bar{K}})
    :\frac{d}{dt}\Big|_{t=0}f_A(\exp(t\xi_1) B\exp(-t\xi_2))=0\Big\}.$$
	In \Cref{maincor}, we have proved
	$$Z'^{\m{aff}}_{0,\bar{K}}\cap F_A=\{0\}.$$ 
	
	Choose a linear subspace $L$ of $\b{V}_{\bar{K}}$ with codimension equal to $\dim(Z^{\m{aff}}_0)=\dim(Y^{\m{aff}})=d$ containing $F_A$ 
	such that $Y'^{\m{aff}}_{\bar{K}}\cap L=\{0\}.$ 
	Recall we may identify $\mathrm{gr}(\omega_{Y^{\m{aff}}_{\bar{K}}})$ with the fiber of $\omega_{Z^{\m{aff}}_{\bar{K}}}$ at $0\in \mathbb{A}_{\bar{K}}^1$.
	Regard $L$ as a linear subspace of $\mathbb V_{\bar{K}}\times \mathbb{A}_{\bar{K}}^1$ contained in $\mathbb{V}_{\bar{K}}\times\{0\}$.
	By \Cref{maincor} and the proof of \cite[3.3]{l-adic},
    the ranks of $\mathcal M_{\bar{K}}\big|_{\mathbb V_{\bar{K}}^{\mathrm{gen}}}$ and $\mathcal M'_{\bar{K}}\big|_{\mathbb V_{\bar{K}}^{\mathrm{gen}}}$ are bounded above by
	\[\dim_{\bar{K}}(\mathrm{gr}(\omega_{Y^{\m{aff}}_{\bar{K}}})
	\otimes_{\mathcal{O}_{\mathbb V_{\bar{K}}}} \mathcal{O}_L)=\dim_{\bar{K}}(\omega_{Z_{\bar{K}}^{\m{aff}}}
	\otimes_{\mathcal{O}_{\mathbb V_{\bar{K}}\times \mathbb{A}^1_{\bar{K}}}} \mathcal{O}_L)\]
	and
	\[\dim_{\bar{K}}(\mathrm{gr}(\c{O}_{Y^{\m{aff}}_{\bar{K}}})
	\otimes_{\mathcal{O}_{\mathbb V_{\bar{K}}}} \mathcal{O}_L)=\dim_{\bar{K}}(\c{O}_{Z_{\bar{K}}^{\m{aff}}}
	\otimes_{\mathcal{O}_{\mathbb V_{\bar{K}}\times \mathbb{A}^1_{\bar{K}}}} \mathcal{O}_L),\]
	respectively. $\omega_{Z^{\m{aff}}}$ and $\mathcal{O}_{Z^{\m{aff}}}$ are Cohen-Macaulay over $\mathcal{O}_{Z^{\m{aff}}}$ by \Cref{filomega} and \cite[6.2.9]{Frob-split}, respectively. Then by the proof of \cite[3.3]{l-adic}, these dimensions are identified with the degree of the closure of $Z'^{\m{aff}}_{\bar{K}}$ in the projective space containing $\mathbb{V}_{\bar{K}}\times \mathbb{A}^1_{\bar{K}}$. We denote the closure by $\overline{Z'^{\m{aff}}_{\bar{K}}}$. % relation with $Z'_{\bar{K}}$
	
    Consider the representations $\rho_1',\cdots,\rho_N'$ and $\rho_0'$ in \Cref{setting}. The corresponding Newton polytope at infinity $\Delta_{\infty}'$ is the convex hull of $\{(x,1):x\in \Delta_{\infty}\}$ and $(0,0)$. So  
	\[\Delta_{\infty}'=\{(x,t):x\in t\Delta_{\infty}, 0\leq t\leq 1\}.\]
	Note in this case the homogeneity condition holds. By \cite[3.4]{l-adic},
	\[\deg(\overline{Z'^{\m{aff}}_{\bar{K}}})
	=(d+1)!\int_{\Delta'_{\infty}\cap\mathfrak C} 
	\prod_{\alpha \in R^+} \frac{(\lambda, \alpha)^2}{(\rho, \alpha)^2}\mathrm d\lambda=d!\int_{\Delta_{\infty}\cap\mathfrak C} 
		\prod_{\alpha \in R^+} \frac{(\lambda, \alpha)^2}{(\rho, \alpha)^2}\mathrm d\lambda.\qedhere\]
\end{proof}

\section{Calculations on arithmetic $\s{D}$-modules}\label{Dmodule}

From now on, for convenience, we assume that $R$ is a complete discrete valuation ring with fraction field $K$, maximal ideal $\f{p}$ and residue field $k$, such that there exists a resolution $\tilde{Z}\to Z$ satisfying the \Cref{ass}. Without loss of generality, we may assume that there exists $\pi\in K$ such that $\pi^{p-1}=-p$. Moreover, when we use the notation in the introduction, we omit the subscript $R$ of $R$-schemes and the subscript $\f{p}$ of $R$-formal schemes.

Notation as \Cref{setting}, let $\f{L}(H)$ be the Lie algebra of $H=G\times G$. Suppose $X$ is a smooth variety with left (resp. right) $H$-action. For any $\xi\in \f{L}(H)$, let $L^X_{\xi}$ (resp. $R^X_\xi$) be the vector field on $X$ given by $\xi$ and the action, that is,
$$L^X_{\xi}(x)=\frac{d}{dt}\Big|_{t=0}(\exp(t\xi)x)\quad (\hbox{resp. } R^X_{\xi}(x)=\frac{d}{dt}\Big|_{t=0}(x\exp(t\xi)))$$ for $x\in X$. We omit the superscript $X$ from $L^X_\xi$ and $R^X_\xi$ if there is no confusion.
%In particular, on $H$ we have the right (resp. left) invariant vector field $L_\xi$ (resp. $R_\xi$) whose value at the unit is $\xi$.

By \Cref{ass}(2), the open subset $\tilde{Y}_0\subset \tilde{Y}$ is $P$-equivariantly isomorphic to $R_u(P)\times S$, where $S$ is a smooth toric variety for the split torus $T$. In particular $S-T$ is a divisor with normal crossing.
Let $E_1', \cdots, E_r'$ be the divisors of $\tilde Y_0$ corresponding to the irreducible components of $R_u(P)\times (S-T)$.
\begin{lem}
    We have $HE_i'\cap \tilde{Y}_0=E_i'$.
\end{lem}

\begin{proof}
    The inlusion $HE_i'\cap \tilde{Y}_0\supseteq E_i'$ is trivial. It suffices to show $HE_i'\cap \tilde{Y}_0\subseteq E_i'$ over each point $x$ of $\Spec R$.
    The irreducible components of $(\tilde{Y}-G)_x$ are $H_x$-stable since $H_x$ is geometrically irreducible. Note that each $E_i'$ is geometrically irreducible since it can be regarded as the product of an algebraic group and a toric scheme. Hence each $E_{i,x}'$ corresponds to an irreducible component of $R_u(P)_x\times (S-T)_x$. 
    So $\overline{E_{i,x}'}$ is an irreducible component of $(\tilde{Y}-G)_x$, where $\overline{E_{i,x}'}$ is the closure of $E_{i,x}'$ in $\tilde{Y}$. Then we conclude $H_x E_{i,x}'\subseteq \overline{E_{i,x}'}$. So $H_x E_{i,x}'\cap \tilde{Y}_{0,x}\subseteq E_{i,x}'$.
\end{proof}

Let $E_i=HE_i'$. Then $\sum_i E_i$ is a Weil divisor on $\tilde Y$. Clearly the underlying closed set of $\sum_i E_i$ is $\tilde{Y}-G$. By \Cref{divisorD}, it coincides with the underlying closed set of the effective Cartier $\tilde{D}$ defined in \Cref{Hyper-D}. 
Therefore, in the following, we regard the divisor $\tilde{D}$ as $\bigcup_i E_i$ by ignoring the multiplicities. Then $\tilde{D}|_{\tilde Y_0}$ has normal crossing on $\tilde{Y}_0$. Since $H\tilde{Y}_0=\tilde{Y}$, this implies $\tilde{D}$ is a divisor with normal crossing.

Let $\c{T}_{\tilde{Y}}(-\log \tilde{D})\subset \c{T}_{\tilde{Y}}$ be the logarithmic tangent bundle along $\tilde{D}$.
Let $\mathcal T_{\tilde{Y}}$, $\mathcal T_S$ and $\mathcal T_{R_u(P)}$ be the sheaves of tangent vectors of ${\tilde{Y}}$, $S$ and $R_u(P)$, respectively.
Let $\mathcal T_0$ be the $\mathcal{O}_{{\tilde{Y}}}$-submodule of $\mathcal T_{\tilde{Y}}$ generated by $L_{\xi}$ for $\xi\in\mathfrak L(H)$, and let $\mathcal T_1$ be the $\mathcal{O}_{S}$-submodule of $\mathcal T_S$ generated by $L_{\xi}$ for $\xi$ lying in the Lie algebra $\mathfrak L(T)$ of $T$.

\begin{lem}\label{ass-4}
    The subsheaves $\c{T}_{\tilde{Y}_0}(-\log \tilde{D}),  \c{T}_{R_{u}(P)}\boxplus \c{T}_1, \c{T}_0|_{\tilde{Y}_0}$ of $\c{T}_{\tilde{Y}_0}$ coincide with each other.
\end{lem}

\begin{proof}
    First let us check $\c{T}_{\tilde{Y}_0}(-\log \tilde{D})= \c{T}_{R_{u}(P)}\boxplus \c{T}_1$. This is a local problem. So we may assume $S\cong \b{G}_m^r\times\b{A}^s$ is a simple toric variety. Let $x_1,\cdots, x_s$ be the coordinates of $\b{A}^s$. Then both $\c{T}_{\tilde{Y}_0}(-\log \tilde{D})$ and $\c{T}_{R_{u}(P)}\boxplus \c{T}_1$ are generated by $\c{T}_{R_{u}(P)}, \c{T}_{\b{G}_m^r}$ and the elements $x_i\frac{\partial}{\partial x_i}$. The claim follows.
   By definition, we have $\c{T}_{R_{u}(P)}\boxplus \mathcal T_1 \subseteq \c{T}_0|_{\tilde{Y}_0}$.
   The irreducible components of $E$ are $H$-stable. Hence each $L_{\xi}$ is tangent to the irreducible components. So $L_\xi$ lies in $\c{T}_{\tilde{Y}}(-\log \tilde{D})$ by \cite[II.3.9]{logtangent}. It follows that $\c{T}_0 \subseteq \c{T}_{\tilde{Y}}(-\log \tilde{D})$.
\end{proof}

\begin{thm}\label{toroidalDmodiso}
    Notation as in \Cref{Hyper-D}. We have a natural isomorphism of left $\s{D}^\d_{\tilde{\f Y},\Q}$-modules
    \[\gamma: \D(^\d \tilde{D}_k)\c{O}_{\tilde{\f Y},\Q}\to \s{D}^\d_{\tilde{\f Y},\Q}/\sum_{\xi\in \f{L}(H)}\s{D}^\d_{\tilde{\f Y},\Q}L_{\xi},\] where $\b{D}$ is the Verdier dual of category $D_\m{coh}^b(\s{D}^\d_{\tilde{\f Y},\b{Q}})$, $(^\d \tilde{D}_k):=\s{D}^\d_{\tilde{\f Y},\b{Q}}(^\d \tilde{D}_k)\otimes_{\s{D}^\d_{\tilde{\f Y},\b{Q}}}-$. We regard $(^\d \tilde{D}_k)\c{O}_{\tilde{\f Y},\Q}$ as a coherent $\s{D}^\d_{\tilde{\f Y},\b{Q}}$-module.
\end{thm}
\begin{proof}
    We repeat the arguments in \cite[2.3]{l-adic}.
    The canonical epimophism \[\s{D}^\d_{\tilde{\f Y}}/\sum_{\xi\in \f{L}(H)}\s{D}^\d_{\tilde{\f Y}}L_{\xi}\to \c{O}_{\tilde{\f Y}}\] is an isomorphism when restricted to $G_{k}$. Taking tensor product with $\b{Q}$ and applying $(^\d \tilde{D}_k)\D$, we get a morphism \[(^\d \tilde{D}_k)\D(\s{D}^\d_{\tilde{\f Y},\mathbb Q}/\sum_{\xi\in \f{L}(H)}\s{D}^\d_{\tilde{\f Y},\mathbb Q}L_{\xi})\to (^\d \tilde{D}_k)\D\c{O}_{\tilde{\f Y},\b{Q}}\cong(^\d \tilde{D}_k)\c{O}_{\tilde{\f Y},\b{Q}}.\] It is also an isomorphism when restricted to $G_{k}$. Note that $\tilde{D}_{k}\subset \tilde{Y}_{k}$ is the complement of $G_k$. We conclude an isomorphism
    \[(^\d \tilde{D}_k)\D (\s{D}^\d_{\tilde{\f Y},\Q}/\sum_{\xi\in \f{L}(H)}\s{D}^\d_{\tilde{\f Y},\Q}L_{\xi})\cong (^\d \tilde{D}_k)\c{O}_{\tilde{\f Y},\Q}\] by \cite[4.3.12]{Berthelot1}.
    We thus get a morphism
    \[\D(\s{D}^\d_{\tilde{\f Y},\Q}/\sum_{\xi\in \f{L}(H)}\s{D}^\d_{\tilde{\f Y},\Q}L_{\xi})\to (^\d \tilde{D}_k)\c{O}_{\tilde{\f Y},\Q}.\]
    Applying $\D$ we get the morphism $\gamma$. We need to show it is an isomorphism.

    This morphism is $H_{k}$-equivariant in the naive sense, that is, for any $h\in H(\bar k)$ the pullback of this morphism by $h: \tilde{Y}_{\bar k}\to \tilde{Y}_{\bar k}$ can be identified with itself.
    By \Cref{ass}(2), we have $H\tilde{Y}_0=\tilde{Y}$. So it suffices to show $\gamma|_{\tilde{Y}_0}$ is an isomorphism. %Note that $E,\tilde{D}$ have the same underlying closed subset.
Since $\b{D}$ commutes with the pull back of smooth morphism, we have \[\D(^\d \tilde{D}_k)\c{O}_{\tilde{\f Y},\Q}\cong \c{O}_{\f{R}_u(P),\b{Q}}\boxtimes\b{D}(^\d (\tilde{D}\cap S)_k) \c{O}_{\f{S},\b{Q}},\] where $\f{R}_u(P)$ and $\f{S}$ are the $\f{p}$-adic completion of $R_u(P)$ and $S$, respectively. By \Cref{ass-4}, we have

\begin{align*} & \Big(\s{D}^\d_{\tilde{\f Y},\Q} / \sum_{\xi \in \mathfrak{L}(H)} \s{D}^\d_{\tilde{\f Y},\Q} L_{\xi}\Big)\Big|_{\tilde{Y}_0}=\left.\left(\s{D}^\d_{\tilde{\f Y},\Q} / \s{D}^\d_{\tilde{\f Y},\Q} \mathcal{T}_0\right)\right|_{\tilde{Y}_0} \\ \cong & \left(\s{D}^\d_{\f{R}_u(P),\b{Q}} / \s{D}^\d_{\f{R}_u(P),\b{Q}} \mathcal{T}_{R_u(P)}\right) \boxtimes\left(\s{D}^\d_{\f{S},\b{Q}} / \s{D}^\d_{\f{S},\b{Q}} \mathcal{T}_1\right) .\end{align*}
So we can treat $R_u(P)$ and $S$ respectively. For $R_u(P)$, we have\[\s{D}^\d_{\f{R}_u(P),\b{Q}} / \s{D}^\d_{\f{R}_u(P),\b{Q}} \mathcal{T}_{R_u(P)}\cong\c{O}_{\f{R}_u(P),\b{Q}}\] by \cite[3.2.2]{Berthelot1990}. For $S$, we may assume $S\cong \b{G}_m^r\times\b{A}^s$ is a simple toric variety. %since $S$ is locally covered by $\b{A}^n$,
We are reduced to the cases of $\b{A}^1$ and $\b{G}_m$. $\b{G}_m$ can be treated similarly to $R_u(P)$. %hence to the case $\b{A}^1$.
For the case $\mathbb A^1$, by \cite[4.3.2]{Berthelot1990}, we have an isomorphism
    \[\delta_0: \s{D}^\d_{\hat{\b{A}}^1,\Q}/\s{D}^\d_{\hat{\b{A}}^1,\Q}\partial x\stackrel\cong\to (^\d {D}_0)\c{O}_{\hat{\b{A}}^1,\Q},\]
where $D_0=\{0\}\subseteq\b{A}^1_k$. Taking the Verdier dual, we get an isomorphism
    \[\gamma_0: \D(^\d {D}_0)\c{O}_{\hat{\b{A}}^1,\Q}\stackrel\cong\to \s{D}^\d_{\hat{\b{A}}^1,\Q}/\s{D}^\d_{\hat{\b{A}}^1,\Q}x\partial.\qedhere\]
\end{proof}

On $\omega_{\tilde{\f Y},\Q} \otimes_{\mathcal{O}_{\tilde{\f Y},\Q}} \s{D}^\dagger_{{\tilde{\f Y},\Q}}$, we have the following two right $L_\xi$-actions: Let $b\otimes Q$ be a section of $\omega_{\tilde{\f Y},\Q} \otimes_{\mathcal{O}_{\tilde{\f Y},\Q}} \s{D}^\dagger_{{\tilde{\f Y},\Q}}$,

(1) The naive right action is given by
$$(b \otimes Q) \cdot L_\xi=b \otimes Q L_\xi .$$

(2) The right action is given by %Leibniz rule
$$
(b \otimes Q) \cdot L_\xi=-L_\xi b \otimes Q-b \otimes L_\xi Q ,$$ where $L_{\xi}b$ is the Lie derivation of $b$ along $L_{\xi}$.
This structure is obtained from the left $\s{D}^\dagger_{{\tilde{\f Y},\Q}}$-module structure on $\s{D}^\dagger_{{\tilde{\f Y},\Q}}$
by side change.

By \cite[1.3.3]{Berthelot2}, there exists an isomorphism $\omega_{\tilde{\f Y},\Q} \otimes_{\mathcal{O}_{\tilde{\f Y},\Q}} \s{D}^\dagger_{{\tilde{\f Y},\Q}}$
which exchanges the two right $\s{D}^\dagger_{{\tilde{\f Y},\Q}}$-module structures, hence exchanging the two $L_\xi$-actions. Then one can verify the following.

\begin{cor}\label{toroidalDmodiso'}
    We have a natural isomorphism of right $\s{D}^\d_{\tilde{\f Y},\Q}$-modules
    \[\gamma: \D(^\d \tilde{D}_k)\omega_{\tilde{\f Y},\Q}\to (\omega_{\tilde{\f Y},\Q}\otimes_{\c{O}_{\tilde{\f Y},\Q}}\s{D}^\d_{\tilde{\f Y},\Q})/\sum_{\xi\in \f{L}(H)} L_{\xi}(\omega_{\tilde{\f Y},\Q}\otimes_{\c{O}_{\tilde{\f Y},\Q}}\s{D}^\d_{\tilde{\f Y},\Q})\]
    where the right $\s{D}^\d_{\tilde{\f Y},\Q}$-module structure is induced by the canonical right $\s{D}^\d_{\tilde{\f Y},\Q}$-module structure on $(\omega_{\tilde{\f Y},\Q}\otimes_{\c{O}_{\tilde{\f Y},\Q}}\s{D}^\d_{\tilde{\f Y},\Q})$, and $L_{\xi}$ acts on $(\omega_{\tilde{\f Y},\Q}\otimes_{\c{O}_{\tilde{\f Y},\Q}}\s{D}^\d_{\tilde{\f Y},\Q})$ by  %by the Leibniz rule with a symbol. 
    \[L_{\xi}(b\otimes Q)=-L_{\xi}b\otimes Q-b\otimes L_{\xi}Q\]for any $b\in \omega_{\tilde{\f Y},\Q}, Q\in \s{D}^\d_{\tilde{\f Y},\Q}, \xi\in \f{L}(H)$.
\end{cor}

\begin{pro}\label{liesdeg0}
    Notation as in \Cref{Hyper-D}. Let $\omega_{G_k}:=(^\d \tilde D_k)\omega_{\tilde{\f{Y}},\b{Q}}$. Then
    $\c{H}^j(\D(\iota p)_{+}\omega_{G_k})=0$ for $j\neq 0$.
\end{pro}

\begin{proof}
    Since $\omega_{G_k}$ is overholonomic, $\D(\iota p)_{+}\omega_{G_k}$ is also overholonomic. %In particular it is holonomic. Note that $\b{D}$ is exact on holonomic complex. 
    We have $\c{H}^j((\iota p)_{+}\omega_{G_k})\cong \c{H}^j((^\d {D}_k)(\iota p)_{+}\omega_{\tilde{\f{Y}}, \Q})$. 
    It suffices to show $\c{H}^j((^\d {D}_k)(\iota p)_{+}\omega_{\tilde{\f{Y}}, \Q})=0$ for $j\neq 0$. By \cite[4.3.12]{Berthelot1}, we just need to show that the complex $(\iota p)_{+}\omega_{\tilde{\f{Y}}, \Q}|_{\hat{\b{P}}-D_k}$ is a $\s{D}^\d_{\hat{\b{P}}-D_k, \Q}$-module.

    The morphism $\iota p: \f{G}=\tilde{\f{Y}}-\tilde{D}_k\to \hat{\b{P}}-D_k$ is a finite morphism since $\f{Y}\to\hat{\b{P}}$ is finite. So
    \[(\iota p)_{+}\omega_{\tilde{\f{Y}}, \Q}|_{\hat{\b{P}}-D_k}\in D_{\m{ovhol}}^{\leq0}(\s{D}^\d_{\hat{\b{P}}-D_k, \Q}).\]
    Taking the Verdier dual, we conclude
    \[(\iota p)_{+}\omega_{\tilde{\f{Y}}, \Q}|_{\hat{\b{P}}-D_k}=\D(\iota p)_{+}\omega_{\tilde{\f{Y}}, \Q}|_{\hat{\b{P}}-D_k}\in D_{\m{ovhol}}^{\geq0}(\s{D}^\d_{\hat{\b{P}}-D_k, \Q}).\qedhere\]
\end{proof}

Let $\unif$ be the uniformizer of $R$. For any $\c{O}_{Y}$-module $\c{M}$, let $\c{M}^\wedge:=\varprojlim_n\c{M}/\unif^n\c{M}$ be the $\unif$-adic completion of $\c{M}$.

\begin{pro}
    $\D(\iota p)_{+}\omega_{G_k}$ is a direct summand of \[\nn:=
    (\omega_{\f{Y},\b Q}\otimes_{\c{O}_{\hat{\b{P}}, \Q}}\s{D}^\d_{\b{\hat P}, \Q})/\sum_{\xi\in \f{L}(H)}L_{\xi}(\omega_{\f{Y}}\otimes_{\c{O}_{\hat{\b{P}}, \Q}}\s{D}^\d_{\b{\hat P}, \Q}).\]
    Here $\omega_{\f{Y},\b Q}:=(Rp_*\omega_{\tilde{Y}})^\wedge\otimes_{\b Z}\b{Q}$ is regarded as a coherent $\c{O}_{\hat{\b{P}},\b Q}$-module via the functor $\iota_*$.
    The action of $L_{\xi}$ on $(\omega_{\tilde{\f Y},\Q}\otimes_{\c{O}_{\hat{\b{P}},\Q}}\s{D}^\d_{\hat{\b{P}},\Q})$ is similar to that in \Cref{toroidalDmodiso'}, that is, for any 
    $b\in \omega_{\f Y,\Q}, Q\in \s{D}^\d_{\hat{\b{P}},\Q}, \xi\in \f{L}(H)$, we define \[L_{\xi}(b\otimes Q)=-L_{\xi}b\otimes Q-b\otimes QL_{\xi} .\]%L_\xi^{\f{Y}}
    %We call $\c{N}$ the modified hypergeometric $\s{D}^\dagger$-module.
\end{pro}

\begin{proof}
    By \Cref{toroidalDmodiso'}, we have an exact sequence \begin{equation}\label{resolution}
    \f{L}(H)\otimes_{R}\omega_{\tilde{\f Y},\Q}\otimes_{\c{O}_{\tilde{\f Y},\Q}}\s{D}^\d_{\tilde{\f Y},\Q}\to \omega_{\tilde{\f Y},\Q}\otimes_{\c{O}_{\tilde{\f Y},\Q}}\s{D}^\d_{\tilde{\f Y},\Q}\to \D(^\d \tilde{D}_k)\omega_{\tilde{\f Y},\Q}\to 0 \end{equation}
    of right $\s{D}^\d_{\tilde{\f Y},\Q}$-modules. 
    By the theorem of formal functions \cite[III 11.1]{Hartshorne},  $\varprojlim_n Rp_*\omega_{\tilde{Y}_{R/\unif^n}}$ is isomorphic to $(Rp_*\omega_{\tilde{Y}})^\wedge=\omega_{\f{Y}}$. So we have
    \[(\iota p)_{+}(\omega_{\tilde{\f Y},\Q}\otimes_{\c{O}_{\tilde{\f Y},\Q}}\s{D}^\d_{\tilde{\f Y},\Q})\cong \omega_{\f{Y},\b Q}\otimes_{\c{O}_{\hat{\b{P}}, \Q}}\s{D}^\d_{\b{\hat P}, \Q}\] by direct calculation.
    Applying $(\iota p)_{+}$ to the sequence (\ref{resolution}), we get a complex \[\f{L}(H)\otimes_{R}\omega_{\f{Y},\b Q}\otimes_{\c{O}_{\hat{\b{P}}, \Q}}\s{D}^\d_{\b{\hat P}, \Q}\to \omega_{\f{Y},\b Q}\otimes_{\c{O}_{\hat{\b{P}}, \Q}}\s{D}^\d_{\b{\hat P}, \Q}\to \D(\iota p)_{+}\omega_{G_k}\]
    in the derived category of right $\s{D}_{\hat{\b{P}},\b{Q}}^\dagger$-modules. It induces a morphism
    \[\psi: \nn\to \D(\iota p)_{+}\omega_{G_k}.\]
    Restricted to the complement of $\tilde{D}_k$, the morphism $\iota p$ becomes finite. Hence $(\iota p)_{+}|_{\tilde{\f Y}-\tilde{D}_k}$ is the derived functor of a right exact functor. 
    So the right exact sequence
    \[\f{L}(H)\otimes_{R}\omega_{\tilde{\f Y},\Q}\otimes_{\c{O}_{\tilde{\f Y},\Q}}\s{D}^\d_{\tilde{\f Y},\Q}|_{\tilde{\f Y}-\tilde{D}_k}\to \omega_{\tilde{\f Y},\Q}\otimes_{\c{O}_{\tilde{\f Y},\Q}}\s{D}^\d_{\tilde{\f Y},\Q}|_{\tilde{\f Y}-\tilde{D}_k}\to \D(^\d \tilde{D}_k)\omega_{\tilde{\f Y},\Q}|_{\tilde{\f Y}-\tilde{D}_k}\to 0.\]
    gives a right exact sequence \[\f{L}(H)\otimes_{R}\omega_{\f{Y},\b Q}\otimes_{\c{O}_{\hat{\b{P}}, \Q}}\s{D}^\d_{\b{\hat P}, \Q}|_{\hat{\b{P}}-D_k}\to \omega_{\f{Y},\b Q}\otimes_{\c{O}_{\hat{\b{P}}, \Q}}\s{D}^\d_{\b{\hat P}, \Q}|_{\hat{\b{P}}-D_k}\to \D(\iota p)_{+}\omega_{G_k}|_{\hat{\b{P}}-D_k}\to 0.\]     
    Hence $\psi|_{\hat{\b{P}}-D_k}$ is an isomorphism. Since $\D\nn\in D^{\m b}_{\coh}(\s{D}^\d_{\b{\hat P}, \Q})$, %applying $(^\d {D}_k)$, 
    we have $(^\d {D}_k)\D\nn\in D^{\m b}_{\coh}(\s{D}^\d_{\b{\hat P}, \Q}(^\d {D}_k))$. By \cite[4.3.12]{Berthelot1}, this implies that the Verdier dual of $\psi$
    \[\D \psi: (\iota p)_{+}\omega_{G_k}\to \D\nn\]
    induces an isomorphism
    \[(^\d {D}_k)\D \psi: (\iota p)_{+}\omega_{G_k}\xrightarrow{\sim} (^\d {D}_k)\D\nn.\]
     Let $\phi$ be the composite
     \[\D\nn\to (^\d {D}_k)\D\nn\xrightarrow{((^\d {D}_k)\D \psi)^{-1}}(\iota p)_{+}\omega_{G_k}.\]
     Then $\phi\circ \D \psi$ is the identity. So $\psi\circ \D\phi$ is identity. Our assertion follows.
\end{proof}

Recall $\infty=\hat{\b{P}}-\hat{\b{V}}$ as in \Cref{Hyper-D}. In order to distinguish the functor $\b{D}$, we denote by $\b{D}_\infty$ the Verdier dual of category $D_\m{coh}^b(\s{D}^\d_{\hat{\b P},\b{Q}}(\infty))$. Applying $(^\d\infty)$, we have the following result in the overconvergent case.

\begin{cor}\label{directsummand}
    The $j$-th cohomological sheaf of \[\D_{\infty}(\iota p)_{+}\omega_{G_k}=\D_{\infty}(\iota p)_{+}(^\d \tilde{D}_k)\omega_{\tilde{\f{Y}}, \Q}\cong \D_{\infty}(^\d {D}_k)(\iota p)_{+}\omega_{\tilde{\f{Y}}, \Q}\] equals to $0$ for $j\neq 0$,
    and $\c{H}^0(\D_{\infty}(\iota p)_{+}\omega_{G_k})$ is a direct summand of \[\Ninf
    :=(\omega_{\f{Y},\b Q}\otimes_{\c{O}_{\hat{\b{P}}, \Q}}\s{D}^\d_{\b{\hat P}, \Q}(\infty))/\sum_{\xi\in \f{L}(H)}L_{\xi}(\omega_{\f{Y},\b Q}\otimes_{\c{O}_{\hat{\b{P}}, \Q}}\s{D}^\d_{\b{\hat P}, \Q}(\infty)).\]
    \end{cor}
    
    We call $\f{F}_\pi(\c{N})$ the \emph{modified hypergeometric $\s{D}^\d_{\b{\hat P}, \Q}(\infty)$-module}.

\begin{rmk}
    We do not know whether $\Ninf$ is holonomic and whether there exists an $F$-structure on $\Ninf$.
\end{rmk}

\section{Right invariant differential operators on $H$.}\label{sec-right-inv}
In this section we define some right invariant differential operators on $H$. In the classical case, $L_\xi$ is sufficient. But in the case of arithmetic $\s{D}$-modules, especially when the level $m$ is positive, the operators $L_\xi$ become nilpotent modulo $\unif$. So they have no contribution to the characteristic cycle. Thus we need to find higher order integral right invariant differential operators on $H$ to control the characteristic cycle.

\subsection{Lie derivation given by a group action}\label{lie-alg-act} Let $R[H]$ be the coordinate ring of the group 
$R$-scheme $H$. 
Let $M$ be a left $R[H]$-comodule. 
For any section $P$ of $\Gamma(H, {\s D}_{H}^{(m)})$, the $H$-action on $M$ defines a map 
\[\beta(P):M\xrightarrow{\m{act}}  R[H] \otimes_{R} M  \xrightarrow{P \otimes\m{id}}  R[H] \otimes_{R} M  \xrightarrow{\m{ev_1}\otimes\m{id}}M,\] where $\m{ev_1}:R[H]\to R$ corresponds to the unit section $1_H:\mathrm{Spec}\,R\to H$.
%We define $\beta:=\actM\circ \m{inv}_*$, where $\m{inv}_*$ is the automorphism ${\s D}_{H}^{(m)}$ induced by the inversion $\m{inv}:H\to H$. 
Note that $\beta$ is only an $R$-linear map from $\Gamma(H, {\s D}_{H}^{(m)})$ to $\m{End}_R (M)$. In the case where $L$ is a vector field, we may identify $\beta(L)\in \m{End}_R (M)$ with the value of the induced Lie algebra action 
$\mathfrak{L}(H)\to \m{End}_R (M)$ 
at $L_1$, where $L_1$ is the value of $L$ at $1_H$. 
    
\begin{lem}\label{beta}
        %$\actM$ vanishes on $\Ker(\m{ev}_1)\Gamma(H, {\s D}_{H}^{(m)})$. 
    $\actM(PL_{\xi})=\actM(P)\circ \actM( L_{\xi})$ for any $P\in\Gamma(H, {\s D}_{H}^{(m)})$. As a corollary, for any $\xi_1,\cdots,\xi_s\in\f{L}(H)$, we have $\beta(L_{\xi_1}\cdots L_{\xi_s})=\beta(L_{\xi_1})\cdots \beta(L_{\xi_s})$. 
\end{lem}

\begin{proof}
    %The map $\m{ev}_1$ is given by modulo $I:=\Ker(\m{ev}_1)$. So for any $f\in I$, the map  $(\m{id}\otimes\m{ev_1})\circ(\m{id}\otimes f)$ vanishes, hence $\actM$ vanishes on $I\Gamma(H, {\s D}_{H}^{(m)})$. 

    Since $L_\xi$ is right $H$-invariant, we have a commutative diagram
    \[\begin{tikzcd}
    {R[H]} \arrow[d, "m"] \arrow[r, "L_\xi"]       & {R[H]} \arrow[d, "m"] \\
    {R[H]\otimes R[H]} \arrow[r, " L_\xi \otimes \m{id}"] & {R[H]\otimes R[H].}         
    \end{tikzcd}\]
    Hence the following diagram commutes:
    \[\begin{tikzcd}
M \arrow[r, "\m{act}"] \arrow[d, "\m{act}"]                                     & {R[H] \otimes_{R} M} \arrow[d, "\m{id}\otimes \m{act} "] \arrow[r, "L_\xi \otimes \m{id}"] & {R[H] \otimes_{R} M} \arrow[d, "\m{id}\otimes \m{act} "] \arrow[r, "\m{ev_1}\otimes\m{id}"] & M \arrow[d, "\m{act}"] \\
{R[H] \otimes_{R} M} \arrow[r, "\m{id}\otimes m"] \arrow[rd, "L_\xi \otimes \m{id}"'] & {R[H] \otimes_{R} R[H] \otimes_{R} M} \arrow[r, "L_\xi \otimes \m{id}"]           & {R[H] \otimes_{R} R[H] \otimes_{R} M} \arrow[r, "\m{ev_1}\otimes\m{id}"]                       & {R[H] \otimes_{R} M.}        \\
     & {R[H] \otimes_{R} M} \arrow[ru, "\m{id}\otimes m"] \arrow[rru, "\m{id}"']            &                                                                            &                       
\end{tikzcd}\]
We conclude that $(L_\xi \otimes \m{id})\m{act}( m)=\m{act}(\beta(L_{\xi}) m)$. Hence
$$\actM(P L_{\xi})(m)=(\m{ev_1}\otimes\m{id})(\m{id}\otimes PL_{\xi})\m{act}(m)=(\m{ev_1}\otimes\m{id})(\m{id}\otimes P)\m{act}(\gamma(L_{\xi}) m)=\actM(P)\circ \actM( L_{\xi})(m).$$
\end{proof}
    Let $V$ be a smooth $R$-scheme with left $H$-action. The composition
    \[V\xrightarrow{(1_H,\m{id})}H\times V\xrightarrow{\m{act}}V\]
    is the identity. For any $P\in\Gamma(H, {\s D}_{H}^{(m)})$, we define $\actV(P)\in\Gamma(V, \s D_{ V}^{(m)})$ to be the element given by taking $(1_H,\m{id})^*$ to the image of $P\boxtimes 1$ under the left $\s D_{H\times V}^{(m)}$-module morphism $ \s D_{H\times V}^{(m)}\to \m{act}^* \s D_{ V}^{(m)}$.
    %\WARN{Note $\actV: \Gamma(U, {\s D}_{H}^{(m)})\to \Gamma(V, \s D_{ V}^{(m)})$ is NOT a ring morphism.} 

\begin{lem}\label{discription-alpha}
   % The map $\actV: \Gamma(U, {\s D}_{H}^{(m)})\to \Gamma(V, \s D_{ V}^{(m)})$ vanishes on $\Ker(\m{ev}_1)\Gamma(U, {\s D}_{H}^{(m)})$. 
   $\actV(L_{\xi_1}\cdots L_{\xi_s})=L^V_{\xi_1}\cdots L^V_{\xi_s}$ for any $\xi_1,\cdots,\xi_s\in\f{L}(H)$.
\end{lem}

\begin{proof}
    Note that $\m{act}_*: \c{T}_{H\times V} \to \m{act}^*\c{T}_V$ maps $L_{\xi}\boxtimes 1$ to $\m{act}^*L^V_{\xi}$. We conclude that \[(L_{\xi}\boxtimes 1)\cdot \m{act}^*(Q)=\m{act}^*(L_{\xi}^VQ)\]for any $Q\in \s D_{ V}^{(m)}$. %, where $L_{\xi}Q$ is defined by the left action of $H$ on $V$. 
    By induction we conclude that the image of $(L_{\xi_1}\cdots L_{\xi_s})\boxtimes 1$ under $ \s D_{H\times V}^{(m)}\to \m{act}^* \s D_{ V}^{(m)}$ is $\m{act}^*(L^V_{\xi_1}\cdots L^V_{\xi_s})$. Hence $\actV(L_{\xi_1}\cdots L_{\xi_s})=L^V_{\xi_1}\cdots L^V_{\xi_s}$. 
\end{proof}

\begin{lem}\label{generated-by-L}
    Let $\zeta_1,\cdots, \zeta_{2d}$ be a basis of $\f{L}(H)$ as an $R$-module. 
    For any integer $i$, 
    the elements $L_{\zeta_{i_1}}\cdots L_{\zeta_{i_s}}$ ($1\leq i_1\leq \cdots \leq i_s\leq 2d, s\leq j$) form a basis for $F^{j}{\s D}_{H,\Q}^{(m)}\cong F^{j}{\s D}_{H,\Q}^{(0)}$ as a left $\c{O}_{ H,\Q}$-module, where $F^{j}{\s D}_{H,\Q}^{(m)}$ denotes the subsheaf of differential operators in ${\s D}_{H,\Q}^{(m)}$ of order $\leq j$, and $d=\dim G_k$.
\end{lem}

\begin{proof}
    It suffices to show that the elements $L_{\zeta_{i_1}}\cdots L_{\zeta_{i_s}}$, $1\leq i_1\leq \cdots \leq i_s\leq 2d, s\leq j$ form a basis for $\gr {\s D}_{H,\Q}^{(0)}\cong \Gamma(H_K, \m{Sym}_{\c{O}_H}\, \c{T}_H)$ as an %left
    $\c{O}_{ H,\Q}$-module. This follows from the fact that the vector fields $L_{\zeta_1}, \cdots, L_{\zeta_{2d}}$ form a basis for $\Gamma(H,\c{T}_H)$. 
\end{proof}

Let $I$ be the global sections of the ideal sheaf of $1_H$ in $H$. It is an ideal of $R[H]$. By the definition in \cite[1.4]{Berthelot1}, we have the $m$-PD envelope $P_{(m)}(I)$ of $(R[H],I)$ and its $m$-PD quotients $P^n_{(m)}(I)$. Let $\c{P}_{(m)}(I)$ and $\c{P}^n_{(m)}(I)$ be the associative sheaves of $P_{(m)}(I)$ and $P^n_{(m)}(I)$, respectively.

\begin{pro}\label{bij-of-inv} 
    Let \[F^{n}U^{(m)}(\f{L}(H))=\m{Hom}_R (P^n_{(m)}(I), R),\quad U^{(m)}(\f{L}(H))=\varinjlim_n\, F^{n}U^{(m)}(\f{L}(H)).\]
    We may identify $U^{(m)}(\f{L}(H))$ with $\Gamma(H,{\s D}_{H}^{(m)}/I{\s D}_{H}^{(m)})$. 
    Let $\Gamma_{\m{inv}}(H, \s D^{(m)}_H)\subset \Gamma(H, \s D^{(m)}_H)$ be the subspace of right invariant differential operators. 
    The composite 
    $$\Gamma_{\m{inv}}(H, \s D^{(m)}_H)\to \Gamma(H, \s D^{(m)}_H/I \s D^{(m)}_H)
    \cong U^{(m)}(\f{L}(H))$$
    is a bijection which maps  $\Gamma_{\m{inv}}(H, F^{n}\s D^{(m)}_H)$ to $F^{n}U^{(m)}(\f{L}(H))$. %We remark that this bijection in general does not preserve the multiplication. 
\end{pro}

\begin{proof}
Let $A$ be the isomorphism $$H\times H \to H\times H, (g,h)\mapsto (g, hg^{-1}).$$ We have a commutative diagram.
\[\begin{tikzcd}
H \arrow[rd, "{(\m{id}, 1_H)}"'] \arrow[r, "\Delta"] & H\times H \arrow[d, "{A}"] \\
                                                  & H\times H   
\end{tikzcd}
\]    
We can thus identify $\c P^n_{(m)}(I_{\Delta})$ with $\c P^n_{(m)}(I_{(\m{id}, 1_H)})\cong \c{O}_{H}\otimes_R P^n_{(m)}(I)$, where $I_{\Delta}$ and $I_{(\m{id}, 1_H)}$ are the global sections of the ideal sheaves of $\Delta$ and $(\m{id}, 1_H)$, respectively.
Since $A(gb,hb)=(gb,hg^{-1})$, under this identification, 
the right multiplication by $H$ on $\c P^n_{(m)}(I_{\Delta})$ is identified with the right multiplication  of $H$
on $\c{O}_{H}\otimes_R P^n_{(m)}(I)$ which is trivial on the factor $P^n_{(m)}(I)$ and is the usual action on $\c{O}_{H}$. 
As a corollary, we have $$F^n\s D^{(m)}_{H}\cong \c{O}_{H}\otimes_R \m{Hom}_R (P^n_{(m)}(I), R)
=\c{O}_{H}\otimes_R F^{n}U^{(m)}(\f{L}(H)).$$ Moreover, a global section $P$ of $\c{O}_{H}\otimes_R  F^{n}U^{(m)}(\f{L}(H))$ 
is right invariant if and only if $P$ is of the form $1\otimes F^{n}U^{(m)}(\f{L}(H))$.  
\end{proof}

\begin{lem}\label{inv-generated-by-L}
    The elements $L_{\zeta_{i_1}}\cdots L_{\zeta_{i_s}}$ ($1\leq i_1\leq \cdots \leq i_s\leq 2d$, $s\geq 0$) form a basis for $\Gamma_{\m{inv}}(H, \s D^{(m)}_{H,\Q})$ as a $K$-module.
\end{lem}
\begin{proof}
    By \Cref{generated-by-L}, any element in $\Gamma(H, \s D^{(m)}_{H,\Q})$ can be written uniquely as $$\sum_{1\leq i_1\leq \cdots \leq i_s\leq 2d}f_{i_1\cdots i_s}L_{\zeta_{i_1}}\cdots L_{\zeta_{i_s}}$$ where each $f_{i_1\cdots i_s}\in K[H]$. Since $L_{\zeta_{i_1}}\cdots L_{\zeta_{i_s}}$ are right invariant, this sum is right invariant if and only if each $f_{i_1\cdots i_s}$ is right invariant, which is equivalent to say $f_{i_1\cdots i_s}\in K$. 
\end{proof}

\subsection{}
Let $D_V^{(m)}=\Gamma(V,\s{D}_V^{(m)})$. Consider the filtration $F^jD_V^{(m)}$ given by the order of differential operators. Let $\m{gr}D_V^{(m)}:=\bigoplus_{r\in\b{Z}}\m{gr}^rD_V^{(m)}$ be its graded ring and let $D_{V,k}^{(m)}:=D_V^{(m)}/\unif D_V^{(m)}$. Then $D_{V,k}^{(m)}$ is isomorphic to $D_{V_k}^{(m)}:=\Gamma(V_k,\s{D}_{V_k}^{(m)})$. For any ring $S$, we denote by $S_\m{red}$ the maximal reduced quotient ring of $S$.

\subsection{}
Let $X$ be a smooth scheme over $k$. 
Berthelot has announced in \cite[5.2.2]{Berthelot} that there is a canonical isomorhism 
\[\Spec (\gr \c{D}_{X}^{(m)})_{\m{red}}\cong T^*X^{(m)}\times_{X^{(m)}} X.\]
Here $X^{(m)}=X\times_{\Spec k,\m{Frob}^m}\Spec k$, and $\m{Frob}^m$ is the Frobenius morphism $\Spec k\to \Spec k,x\mapsto x^{p^m}$. 
This is straightforward, but to the best of the author’s knowledge, no proof or discription of this isomorphism has been written down. 
For completeness, we give a proof here. 

By covering $X$ by charts, we may assume there is an \'{e}tale morphism $(t_1,\cdots, t_n): X\to \b{A}^n_k$. By \cite[2.2.3]{Berthelot1}, $\gr \c{D}_{X}^{(m)}$ is the free $\c{O}_X$-module with basis  $\underline{\partial}^{\langle\underline{k}\rangle_{(m)}}$. By \cite[2.2.4(iii)]{Berthelot1}, for $1\leq i\leq n$, $0\leq j<m$ we obtain 
\[(\partial_i^{\langle p^j\rangle_{(m)}})^p= \frac{p^{j+1}!}{(p^j!)^p}\partial^{\langle p^{j+1}\rangle_{(m)}}=0.\]
Note that 
$$\gr \c{D}_{X}^{(m)}/(\partial_i^{\langle p^j\rangle_{(m)}} \text{ for } 1\leq i\leq n, 0\leq j<m)\cong \c{O}_X[\partial_i^{\langle p^m\rangle_{(m)}}\text{ for } 1\leq i\leq n],$$
which is clearly reduced. Hence we may identify $(\gr \c{D}_{X}^{(m)})_{\m{red}}$ with $\c{O}_X[\partial_i^{\langle p^m\rangle_{(m)}}]$. 
Let $\partial_i^{(m)}\in \Gamma(X^{(m)},T^*X^{(m)})$ be the base change of $\partial_i\in \Gamma(X,T^*X)$ along the Frobenius morphism $\m{Frob}^m$. 
The assignment $$\partial_i^{\langle p^m\rangle_{(m)}}\mapsto\partial_i^{(m)}$$
defines an isomorphism
\[\Spec (\gr \c{D}_{X}^{(m)})_{\m{red}}\cong T^*X^{(m)}\times_{X^{(m)}} X.\]
One can check that this does not depend on the choice of the coordinates $(t_i)$.  

\begin{pro}\label{def-L^<k>} 
    For any vector field $L$ on a smooth $R$-scheme 
    $V$ and any integer $p^m\geq r\geq 0$, there exists a unique element $L^{\langle r\rangle_{(m)}}\in \gr^r D_{V}^{(m)}$ such that $r!L^{\langle r\rangle_{(m)}}=L^r$. Moreover, the function $L^{\left<{r}\right>_{(m)}}$ vanishes in $\Spec (\gr D_{V,k}^{(m)})_{\m{red}}$ for $0<r<p^m$. The function $L^{\left<{p^m}\right>_{(m)}}$ can be identified with the composite 
    \begin{equation}\label{composite}\Spec (\gr D_{V,k}^{(m)})_{\m{red}}\cong T^*V_{k}^{(m)}\times_{V_{k}^{(m)}} V_{k}
    %\xrightarrow{(\m{id}, \m{Frob}^m)} V_{k}^{*(m)}\times V_{k}^{(m)}\cong 
    \to T^*V_{k}^{(m)}\xrightarrow{L^{(m)}}\b{A}^1_k.    \end{equation}
    Here $V_k^{(m)}$ and $L^{(m)}$ are the base change of $V_k$ and $L: T^*V_{k}\to \b{A}_k^1$ by the Frobenius morphism $\m{Frob}^m: \Spec k\to \Spec k,x\mapsto x^{p^m}$, respectively. %The isom
\end{pro}

\begin{proof}
  The uniqueness of $L^{\langle r\rangle_{(m)}}$ follows from the fact that
  $\gr D_{V}^{(m)}$ is a locally free $\c{O}_{V}$-module and $\c{O}_{V}$ is flat over $\b{Z}$. 
    The existence and the desired property of $L^{\langle r\rangle_{(m)}}$ are local problems. So we may assume $V$ is affine and we may find global coordinates $y_i$. 
    Consider the subset 
    \[\c{S}:=\left\{L\in \Gamma(V,\c{T}_V) 
    \left| \begin{array}{l}\forall 0\leq r \leq p^m, \exists L^{\langle r\rangle_{(m)}}\in \gr^r D_{V}^{(m)}\text{ such that}: \\
    r!L^{\langle r\rangle_{(m)}}=L^r, \\ 
    \text{$L^{\left<{r}\right>_{(m)}}$ vanishes in $\Spec (\gr D_{V,k}^{(m)})_{\m{red}}$ for $0<r<p^m$}, \\ 
    \text{$L^{\left<{p^m}\right>_{(m)}}$ can be identified with the composite \ref{composite}}.\end{array}\right.\right\}.\]
    Let's show this set contains $\partial_i:=\partial/\partial y_i$ and is an $R[V]$-submodule, and hence $\c{S}= \Gamma(V, \c{T}_V)$. 

    When $L=\partial_i$, we have a well-defined differential operator $\partial_i^{\langle r\rangle_{(m)}}$ of order $r$, and it defines an element in $\gr D_{V}^{(m)}$, still denoted by $\partial_i^{\langle r\rangle_{(m)}}$. For any $0\leq j<m$, we have $(\partial_i^{\langle p^j\rangle_{(m)}})^p=0$ since $p|\frac{p^ {j+1}!}{(p^j!)^p}$. Hence $\partial_i^{\left<{r}\right>_{(m)}}$ vanishes in $\Spec (\gr D_{V,k}^{(m)})_{\m{red}}$ for $0<r<p^m$.
    Recall the identification of $\Spec (\gr D_{V,k}^{(m)})_{\m{red}}$ with $T^*V_{k}^{(m)}\times_{V_{k}^{(m)}} V_{k}$ as $V_{k}$-schemes is given by $$\partial_i^{\langle p^m\rangle_{(m)}}\mapsto\partial_i^{(m)}.$$
    We conclude that $\partial_i^{\langle r\rangle_{(m)}}$ satisfies the desired properties. 

    If $L,L'\in \c{S}$ and $f\in R[V]$, we can define $$(fL)^{\langle r\rangle_{(m)}}=f^r L^{\langle r\rangle_{(m)}},
    \quad (L+L')^{\langle r\rangle_{(m)}}=\sum_{s=0}^r L^{\langle r-s\rangle_{(m)}}L'^{\langle s\rangle_{(m)}}.$$ 
    They clearly satisfy the desired properties. Hence $fL, L+L'\in \c{S}$.
\end{proof}

For any integer $r$, let $r=ap^m+b$, where $0\leq b<p^m$. By the notation and proof in \cite[1.1.3]{Berthelot1}, we have 
\[\frac{r!}{(p^m!)^aa!b!}=\left<\begin{array}{c}2p^m \\ p^m\end{array}\right>_{(m)}\cdots\left<\begin{array}{c}ap^m \\ (a-1)p^m\end{array}\right>_{(m)}\left<\begin{array}{c}ap^m+b \\ ap^m\end{array}\right>_{(m)}\in \b{Z}_{(p)}^\times\subseteq R^\times.\]
Then we can define $L^{\langle r\rangle_{(m)}}:=\frac{(p^m!)^aa!b!}{r!}(L^{\langle p^m\rangle_{(m)}})^aL^{\langle b\rangle_{(m)}}\in \gr^r D_{V}^{(m)}$.

\begin{defn}
    Choose a local coordinate system $(U, t_i\in \Gamma(U, \c O_H))$ of $H$ at the unit such that $L_{\zeta_i}|_{1_H}=\partial/\partial t_i|_{1_H}$, where $U\subset H$ is an open affine subset containing the unit. 
    Let $I_U\subset R[U]$ be the global sections of the ideal sheaf of $1_H$ in $H$. 
    By \cite[Remarques 1.4.3(iv)]{Berthelot1}, both $P^n_{(m)}(I)$ and $P^n_{(m)}(I_U)$ can be identified with 
    \[P^n_{(m)}(I/I^{n+1})\cong P^n_{(m)}(I_U/I_U^{n+1}).\]
    By \cite[1.5.1]{Berthelot1}, $P^n_{(m)}(I)\cong P^n_{(m)}(I_U)$ is a free $R$-module with basis $\underline{t}^{\langle \underline{r}\rangle_{(m)}}$
    ($ |\underline{r}|\leq n$). Let $\underline{\eta}^{\langle \underline{r}\rangle_{(m)}}$ be the dual basis. It depends not only on $\zeta_i$ but also on the choice of the coordinate system $(U, t_i)$.
    By \Cref{bij-of-inv}, the element $\eta_i^{\langle r\rangle_{(m)}}\in U^{(m)}(\f{L}(H))$ defines a right invariant differential operator $\nabla_{i}^{\langle r\rangle_{(m)}}\in \Gamma_{\m{inv}}(H, \s D^{(m)}_H)$.  %{\color{red}  $\partial/\partial t_i$ is only defined on $U$. How to make $\underline{t}^{\langle \underline{r}\rangle_{(m)}}$ a basis for the global $P^n_{(m)}(I)$}
\end{defn}

\begin{pro}\label{heighest-symbol-of-nabla}
    In $\gr^r D_{H}^{(m)}$, we have $\nabla_{i}^{\langle r\rangle_{(m)}}=L_{\zeta_i}^{\langle r\rangle_{(m)}}$.
\end{pro}

\begin{proof}
    Note that $\nabla^{\langle r\rangle_{(m)}}=\frac{(p^m!)^aa!b!}{r!}(\nabla^{\langle p^m\rangle_{(m)}})^a\nabla^{\langle b\rangle_{(m)}}$ by \cite[2.2.5]{Berthelot1}. We may assume $0\leq r\leq p^m$. Then it suffices to show that $L_{\zeta_i}^r$ and $r! \nabla_{i}^{\langle r\rangle_{(m)}}$ have the same image in $\gr^r D_{H}^{(m)}$. Note that $L_{\zeta_i}^r$ is right invariant. By \Cref{bij-of-inv}, it corresponds to an element $P$ in $U^{(m)}(\f{L}(H))$. Let us calculate the value of $P$ on the basis $\underline{t}^{\langle \underline{r}\rangle_{(m)}}:=\prod_{i=1}^{2d}t_i^{\langle r_i\rangle_{(m)}}$. 

    By definition, $P(\underline{t}^{\langle \underline{r}\rangle_{(m)}})$ equals the value of $(1\otimes L_{\zeta_i}^r)(\underline{t}^{\langle \underline{r}\rangle_{(m)}}\circ A)\in \Gamma(H\times H , \c{O}_{H\times H })$ at $1_H\times 1_H$, which clearly equals $L_{\zeta_i}^r\underline{t}^{\langle \underline{r}\rangle_{(m)}}|_{t=0}$. 
    Let $\delta_{ij}$ be the Kronecker symbol. Recall that $L_{\zeta_i}t_j|_{t=0}=\delta_{ij}$. By the Leibniz rule, we have
    \[L_{\zeta_i}^r\underline{t}^{\langle \underline{r}\rangle_{(m)}}|_{t=0}=r!\prod_{s=1}^{r}\delta_{is}^{r_s}=\left\{ \begin{aligned}
        r! &\quad \text{ if } r_i=r \text{ and } \forall s\ne i, r_s=0 \\
        0  &\quad \text{ otherwise,}
    \end{aligned}\right.\] where $|\underline{r}|=r$.
    Therefore $P-r!\eta_i^{\langle r\rangle_{(m)}}\in F^{r-1}U^{(m)}(\f{L}(H))$. Hence $L_{\zeta_i}^r$ and $r! \nabla_{i}^{\langle r\rangle_{(m)}}$ differ by a right invariant differential operators of order $\leq r-1$.
\end{proof}

\begin{cor}\label{heighest-symbol-of-alpha}
    In $\gr^r D_{V}^{(m)}$ we have $\alpha(\nabla_{i}^{\langle r\rangle_{(m)}})=L_{\zeta_i}^{\langle r\rangle_{(m)}}$.
\end{cor}

\begin{proof}
    Let $r=ap^m+b$, where $0\leq b<p^m$. By definition, we have $L^{r}_{\zeta_i}=(r!/a!)L_{\zeta_i}^{\langle r\rangle_{(m)}}$. By \Cref{heighest-symbol-of-nabla}, we have $(r!/a!)\nabla_{i}^{\langle r\rangle_{(m)}}-L^{r}_{\zeta_i}\in F^{r-1}D_{H}^{(m)}$, where $D_{H}^{(m)}=\Gamma(H,\s{D}_{H}^{(m)})$. Hence $(r!/a!)\alpha(\nabla_{i}^{\langle r\rangle_{(m)}})-\alpha(L^{r}_{\zeta_i})\in F^{r-1}D_{V}^{(m)}$. 
    By \Cref{discription-alpha}, $\alpha(L^{r}_{\zeta_i})=L^{r}_{\zeta_i}$. 
\end{proof}

\section{Calculation on characteristic cycles}\label{pcase}

%We want to imitate the methods in characteristic $0$ case.

\subsection{}\label{affintmodel}
Choose a basis $x=\{x_j\}_{j=1}^N$ on $\b{V}\cong \b{V}^*$. Let $\partial=\{\partial_j\}_{j=1}^N$ be the corresponding derivation of $x$.
    %We fix a Frobenius lifting $\BBmm\to \BBmm, x\mapsto x^p$ of $\hat{\b{V}}$. Assume $k=\b{F}_{p^s}$.
Recall $\pi^{p-1}=-p$. We define 
     \[A^{(m)}=R[\pi(\partial_j/\pi)^{p^i}, 0\leq i\leq m].\]
It can be regarded as an $R$-submodule of $D_{\b{V}}^{(m)}:=\Gamma(\b{V},\s{D}_{\b V}^{(m)})$ since $\partial_j^{p^i}=p^i!\partial_j^{\langle p^i\rangle_{(m)}}$ and the $\pi$-adic valuation of $p^i!$ is $p^i-1$. So the left $(H\times \b{G}_m)$-action on $\b{V}$ can induce a left $(H\times \b{G}_m)$-action on $A^{(m)}$. %We can also define a $\b{G}_m$-action on $A^{(m)}$ via $t\cdot\pi(\partial_j/\pi)^{p^i}\mapsto\pi(t\partial_j/\pi)^{p^i}$??. In summary, $A^{(m)}$ is equipped with an $H\times \b{G}_m$-action. 

Clearly $A^{(m)}_{\Q}= K[\partial]$. So we can identify $\Spec A^{(m)}_{\Q}$ with $\b{V}_K$ via $x_j\mapsto -\partial_j/\pi$. 
Let $Z$ be as in \Cref{setting}. Recall in \Cref{ass} we have a smooth $R$-scheme $\tilde{Z}$ and a birational morphism $p: \tilde{Z}\to Z$. We define $\omega_Z=p_*\omega_{\tilde{Z}}$. Regard $\omega_Z$ as an $\c{O}_{\b{P}\times\b{P}^1}$-module via the functor $\morZ'_*$. We can identify the coherent sheaf $\omega_{Z^{\m{aff}}_K}=\omega_{Z}|_{\b{V}_K\times \b{A}^1_K}$ with a finite $(A^{(m)}_{\Q}\otimes_K K[t])$-module $S_{\Q}^{(m)}$. 
By \cite[1.5 Proposition 2]{Serre},
there exists a finite $(A^{(m)}\otimes_R R[t])$-submodule $S^{(m)}\subset S^{(m)}_{\Q}$ such that $S^{(m)}\otimes_{\b{Z}}{\Q}=S_{\Q}^{(m)}$ and $S^{(m)}$ is a sub-$(R[H]\otimes_R R[t,t^{-1}])$-comodule of $S_{\Q}^{(m)}$. %As a corollary, for each $\xi\in \f{L}(H)_{R}$ and $k\geq 0$, $\frac{1}{k!}L_{\xi}^kS\subseteq S$.
We define $M^{(m)}:=S^{(m)}\otimes_{R[t],t\mapsto 1} R$, which is a finite $A^{(m)}$-module with $H$-action. %where the map $\alpha: R\langle \pi t^{p^m}\rangle\to R$ maps $\pi t^{p^m}$ to $1$. This is a finite $\hat{A}^{(m)}$-module with $H$-action. %and a compatible good filtration modulo $\unif$.
Then we have $M^{(m)}_\Q\cong \omega_{Y^{\m{aff}}_K}$ by \Cref{YfiberofZ}.

Recall $\unif$ is the uniformizer of $R$. $S^{(m)}_k:=S^{(m)}/\unif S^{(m)}$ is a finite $({A}^{(m)}_k\otimes_k k[t])$-module with $(H\times \Spec k[t,t^{-1}])$-action. It may not be flat over $k[t]$. Nevertheless, it is flat over an open dense subset of $\m{Spec}(k[t])$. The fact that $S_k^{(m)}$ has a
    $ \Spec k[t,t^{-1}]$-action implies that it is flat over $ \Spec k[t,t^{-1}]\subset  \Spec k[t]$. Let $\bar S^{(m)}_k$ be the maximal torsion free quotient of $S^{(m)}_k$, the natural map $S^{(m)}_k\twoheadrightarrow \bar S^{(m)}_k$ induces an isomorphism
    \[\bar S^{(m)}_k\otimes_{k[t],t\mapsto 1} k\cong S^{(m)}_k\otimes_{k[t],t\mapsto 1} k\cong M^{(m)}_k.\]
    The $(H_k\times \Spec k[t,t^{-1}])$-action on $ S^{(m)}_k$ induces an $(H_k\times \Spec k[t,t^{-1}])$-action on $\bar S^{(m)}_k$. 
    By \Cref{filmod}, $\bar S^{(m)}_k$ defines an $H_k$-invariant good filtration on $M^{(m)}_k$. Moreover, we have $$\gr M^{(m)}_k\cong \bar S^{(m)}_{k,0}:=\bar S^{(m)}_{k} \otimes_{k[t],t\mapsto 0}k.$$ Hence $\gr M^{(m)}_k$ is a quotient of $S^{(m)}_{k,0}$.

\subsection{}\label{S0=omega}
    Moreover, we choose a specific $S^{(0)}$ to do calculation. 
    Namely, the coherent sheaf $\omega_{Z^{\m{aff}}}$ defines a finite $(R[x]\otimes_R R[t])$-module $\Gamma(\b{V}\times \b{A}^1, \omega_{Z^{\m{aff}}})$ with $(H\times\Spec R[t])$-action. We have an isomorphism
    \[R[x]\otimes_R R[t]\to R[\pi x]\otimes_R R[\pi t],\quad x_j\mapsto \pi x_j, \quad t\mapsto  \pi t.\]
    Using this isomorphism, we regard $\Gamma(\b{V}\times \b{A}^1, \omega_{Z^{\m{aff}}})$ as a finite $(R[\pi x]\otimes R[\pi t])$-module $S^{(0)}$ with $(R[H]\otimes_R R[t])$-action. 
    The multiplication by $\pi\in K^{\times}=\b{G}_{m}(K)$ induces an isomorphism $S^{(0)}\otimes_{\b{Z}}{\Q}\cong \omega_{Z^{\m{aff}}_K}$ between $K[x,t]$-modules. 
    We identify $A^{(0)}_k=k[\partial]$ with $k[\pi x]$. Then we can identify $S^{(0)}_0$ with $\omega_{Z^{\m{aff}},0}$. As a corollary, we get the induced $A^{(0)}$-module $M^{(0)}$ and the induced filtration on $M^{(0)}_k$. 
    %We will show that % this choice satisfies $S^{(0)}_{k,0}\cong \omega_{Z}$ $\gr (M^{(0)}/\unif M^{(0)})$ 
    The corresponding graded module $\gr M^{(0)}_k$ is a quotient of $\omega_{Z^{\m{aff}},0,k}$.

\subsection{}
Consider the $\unif$-adic completion of $A^{(m)}$, denoted by
    \[\hat{A}^{(m)}=R\langle\pi{(\partial_j/\pi)^{p^i}}, 0\leq i\leq m\rangle.\]%=R[T^{(m)*}_0\hat{\b{V}}]
    We have an $R$-module isomorphism
    $\hat{D}_{\hat{\b{V}}}^{(m)}:=\Gamma(\hat{\b{V}},\hat{\s D}_{\hat{\b{V}}}^{(m)})\cong R[\hat{\b{V}}]\hat{\otimes}_{R} \hat{A}^{(m)}$.
    Via the naive Fourier transform map $x_i\mapsto -\partial_i/\pi$, 
    we may identify $\Spm \hat{A}^{(m)}_{\Q}$ with %$X(\pi x^{p^m})\subset(\b{V}_K)^\an$, 
    the %affinoid open subset
    closed disc of radius $|\pi|^{-\frac{1}{p^m}}$. %denoted by $B_m$. 

    For each integer $m\geq 0$, we define $\hat{S}^{(m)}$ as the $\unif$-adic completion of $S^{(m)}$. It is a $\hat{A}^{(m)}\hat\otimes R\langle t\rangle$-module with $\f{H}\times_\c{V} \hat{\b{G}}_m$-action, where $\f{H}\times_\c{V} \hat{\b{G}}_m$ is the $\unif$-adic completion of $H\times\b{G}_m$, $\c{V}=\m{Spf}(R)$. Similarly, let $\hat{M}^{(m)}$ be the completion of $M^{(m)}$. It is a $\hat{A}^{(m)}$-module with $\f{H}$-action.
    We may identify $\hat{M}^{(m)}_k$ with ${M}^{(m)}_k$, in particular it is equipped with a good filtration. 

    \begin{pro}\label{rel-M-omega}
        $\hat{M}^{(m)}$ defines a coherent sheaf $\hat{M}^{(m)}_{\Q}$ on $\Spm \hat{A}^{(m)}_{\Q}$.  It is isomorphic to the restriction of the coherent $\c{O}_{\b{P}_K^{\an}}$-module $\omega_{Y,\Q}^{\an}\cong \omega_{\f{Y},\Q}$ to $\Spm \hat{A}^{(m)}_{\Q}$. They have compatible $\f{H}_K$-actions. %Here we use the Fourier transform map to identify $\Spm \hat{A}^{(m)}_{\Q}$ with $B_m$. 
    \end{pro}

    \begin{proof}
    The analytification of ${M}^{(m)}_{\Q}\cong \omega_{Y,\Q}|_{\b{V}_K}$ is a coherent sheaf %$\hat{M}^{(m)}_{\Q}$ 
    on $\b{V}_K^{\an}$. We thus have
        \[\hat{M}^{(m)}_{\Q}\cong ({M}^{(m)}_{\Q})^{\an}|_{\Spm \hat{A}^{(m)}_{\Q}}\cong \omega_{Y,\Q}^{\an}|_{\Spm \hat{A}^{(m)}_{\Q}}. \qedhere\]
    \end{proof}

\subsection{Description of $\f{F}_\pi(\c{N})$}
%Recall that our $\f{F}_\pi$ has a $2-n$ shift compared with the usual case, where $n=\dim\b{V}_k$. 
We can regard $\f{F}_\pi(\c{N})$ as a coherent $\s{D}^\d_{\b{\hat P}, \Q}(\infty)$-module.

%The right $\s{D}^\d_{\b{\hat P}, \Q}(\infty)$-module $\f{F}_\pi(\c{N})$ is associated to the right ${D}^\d_{\b{P}, \Q}(\infty)$-module $F_\pi(\c{N})$, whose underlining $K$-module is
Since the closed disc $B_m$ is a strict neighborhood of the generic fiber of $\hat{\b V}$ in $\b{P}_K^\an$, 
we may identify $\omega_{\f{Y},\Q}\otimes_{\c{O}_{\hat{\b{P}}, \Q}}\s{D}^\d_{\b{\hat P}, \Q}(\infty)%\cong \omega_{Z,\Q}^{\an}\otimes_{\c{O}_{\hat{\b{P}}, \Q}}\s{D}^\d_{\b{\hat P}, \Q}(\infty)
$ with $\omega_{\f{Y},\Q}|_{B_m}\otimes_{\c{O}_{B_m}}\s{D}^\d_{\b{\hat P}, \Q}(\infty)$, hence identify
$$\c{N}=(\omega_{\f{Y},\Q}\otimes_{\c{O}_{\hat{\b{P}}, \Q}}\s{D}^\d_{\b{\hat P}, \Q}(\infty))/\sum_{\xi\in \f{L}(H)}L_{\xi}(\omega_{\f{Y},\Q}\otimes_{\c{O}_{\hat{\b{P}}, \Q}}\s{D}^\d_{\b{\hat P}, \Q}(\infty))$$ with
\[(\omega_{\f{Y},\Q}|_{B_m}\otimes_{\c{O}_{B_m}}\s{D}^\d_{\b{\hat P}, \Q}(\infty))/\sum_{\xi\in \f{L}(H)}L_{\xi}(\omega_{\f{Y},\Q}|_{B_m}\otimes_{\c{O}_{B_m}}\s{D}^\d_{\b{\hat P}, \Q}(\infty))).\]
Taking the Fourier transform and applying \Cref{rel-M-omega}, $\f{F}_\pi(\c{N})$ can be identified with the right $\s{D}^\d_{\b{\hat P}, \Q}(\infty)$-module
\[(\FsM^{(m)}_\Q\otimes_{\hat{A}^{(m)}_\b{Q}}\s{D}^\d_{\b{\hat P}, \Q}(\infty))/\sum_{\xi\in \f{L}(H)}L_{\xi}(\FsM^{(m)}_\Q\otimes_{\hat{A}^{(m)}_\b{Q}}\s{D}^\d_{\b{\hat P}, \Q}(\infty)),\]
where for any $b\in \FsM^{(m)}_\Q, Q\in \s{D}^\d_{\hat{\b{P}},\Q}(\infty), \xi\in \f{L}(H)$, we define
    \[L_{\xi}(b\otimes Q)=-L_{\xi}b\otimes Q-b\otimes F_\pi(L_{\xi})Q.\]
Here, as in \Cref{toroidalDmodiso'}, $L_{\xi}b$ is the Lie derivation of $b$ along $L_{\xi}$. We can thus identify $L_{\xi}b$ with $\beta(L_{\xi})b$ defined in \Cref{lie-alg-act}.
    
%={^tL_{\xi}}b\otimes t-b\otimes F_\pi(L_{\xi})t
    
%    Here $L_{\xi}b$ is the Lie derivation of $b$ along $L_{\xi}$ given by the $H$-action on $\omega_{\f Y,\Q}$, $D_{\hat{\b{V}},\Q}^\d= \Gamma(\hat{\b{V}},\s{D}_{\hat{\b{V}},\Q}^\d)$.
In particular, we can define a right $\hat{D}_{\hat{\b{V}},\Q}^{(m)}$-module 
\[F_\pi(\Ninf)^{(m)}_\Q:=(\FsM^{(m)}_\Q\otimes_{\hat{A}^{(m)}_\Q}\hat{D}_{\hat{\b{V}},\Q}^{(m)})/\sum_{\xi\in \f{L}(H)}L_{\xi}(\FsM^{(m)}_\Q\otimes_{\hat{A}^{(m)}_\Q}\hat{D}_{\hat{\b{V}},\Q}^{(m)})\] associated with the right $\s{D}^\d_{\b{\hat V}, \Q}$-module $\f{F}_\pi(\c{N})|_{\b{\hat V}}$,
where $\hat{D}_{\hat{\b{V}},\Q}^{(m)}= \Gamma(\hat{\b{V}},\hat{\s{D}}_{\hat{\b{V}},\Q}^{(m)})$, the right $\hat{D}_{\hat{\b{V}},\Q}^{(m)}$-module structure on $F_\pi(\Ninf)^{(m)}_\Q$ is given by the multiplication on the second factor. 
%composite of ${D}^\d_{\b{P}, \Q}(\infty)\xrightarrow{F_\pi}{D}^\d_{\b{P}, \Q}(\infty)$ and the right %${D}^\d_{\b{P}, \Q}(\infty)$-module structure on $\c{N}$.  
    
    Recall that ${A}^{(m)}_\Q\cong {A}^{(m')}_\Q$, ${M}^{(m)}_\Q\cong {M}^{(m')}_\Q$.
    We have $$\hat{M}^{(m)}_\Q\otimes_{\hat{A}^{(m)}_\Q}\hat{A}^{(m')}_\Q\cong {M}^{(m)}_\Q\otimes_{{A}^{(m)}_\Q}\hat{A}^{(m)}_\Q\otimes_{\hat{A}^{(m)}_\Q}\hat{A}^{(m')}_\Q\cong {M}^{(m')}_\Q \otimes_{{A}^{(m')}_\Q}\hat{A}^{(m')}_\Q\cong \hat{M}^{(m')}_\Q.$$
    %Since M^{(m)} are coherent for any $m$, its completion are equal to $-\otimes_{A^{(m)}}\hat{A}^{(m)}$
    By the flatness of ${D}_{\hat{\b{V}},\Q}^{\d}$ over $\hat{D}_{\hat{\b{V}}}^{(m)}$ (see \cite[3.5.4]{Berthelot1}), it is easy to see that $F_\pi(\Ninf)^{(m)}_\Q\otimes _{\hat{D}_{\hat{\b{V}},\Q}^{(m)}}{D}_{\hat{\b{V}},\Q}^{\d}$ is isomorphic to $ F_\pi(\Ninf)|_{\hat{\b{V}}}$, and  $F_\pi(\Ninf)^{(m)}_\Q\otimes _{\hat{D}_{\hat{\b{V}},\Q}^{(m)}}\hat{D}_{\hat{\b{V}},\Q}^{(m')}$ is isomorphic to $F_\pi(\Ninf)^{(m')}_\Q$. %We also have $\varinjlim_mF_\pi(\Ninf)^{(m)}\cong F_\pi(\Ninf)|_{\hat{\b{V}}}$, but for any $m'\geq m$, there exists only a surjective morphism $F_\pi(\Ninf)^{(m)}\otimes_{\hat{D}_{\hat{\b{V}}}^{(m)}}\hat{D}_{\hat{\b{V}}}^{(m')}\to F_\pi(\Ninf)^{(m')}$.
    
    \subsection{} Choose a nonzero left-invariant top form 
     $\omega_0$ of the invertible sheaf $\omega_{H}$.
     By \cite[1.3.4]{Berthelot2}, we can use this section to define the ``isomorphismes de transposition'': 
      $${^t\cdot}: \hat{\s D}_{{H}}^{(m)}\to \hat{\s D}_{H}^{(m),\m{op}}.$$
    By \cite[1.2.2]{Berthelot2}, as in the characteristic $0$ case we still have
     $$(f\omega_0) P= {^t P}(f) \omega_0$$ for any open set $U\subseteq H$, function $f\in \Gamma(U, \c{O}_H)$ and differential operator $P\in \Gamma(U, \hat{\s D}_{{H}}^{(m)})$.
     Since $L_{\xi}\omega_0=0$, we have ${^tL_{\xi}}=-L_{\xi}$.

    %Let $\chi$ be the character $H\to \prod \m{GL}(V_j)\xrightarrow{\prod \det_j}\b{G}_m$, %$\m{d}\chi$ the tangent map of $\chi$. 
    Consider the character \[\chi:H=G\times G\to\b{G}_m,\quad h=(g_1,g_2)\mapsto \prod_{j=1}^N\det(\rho_j(g_1g_2^{-1})).\]
    Let $ \theta$ be the automorphism $$\theta: H\to H,\quad (g_1,g_2)\mapsto(g_2,g_1).$$ %defined by switching the two factors of $H=G\times G$. 
    For any $h=(g_1,g_2)\in H$ and $a=(a_j),b=(b_j)\in\b{V}$, the pairing $\langle,\rangle: \b{V}\times \b{V}\to \b{A}^1$ satisfies
    \[\langle h\cdot a,b \rangle=\sum_{j=1}^N\m{Tr}(\rho_j(g_1)a_j\rho_j(g_2^{-1})b_j)=\sum_{j=1}^N\m{Tr}(a_j\rho_j(g_2^{-1})b_j\rho_j(g_1))=\langle a, \theta(h^{-1})\cdot b \rangle.\]
    %Therefore, for any $\xi\in\f L(H)$, the coefficient matrix of $-L_{\theta(\xi)}$ is the transpose of the coefficient matrix of $L_{\xi}$.
    Hence we have $F_\pi(L_{\xi})=L_{\theta(\xi)}-\chi(\xi)$ for any $\xi\in\f L(H)$ by direct calculations.
    
    Then we define a new action of $H$ on $M^{(m)}$ by
    \[\m{act}': H\times M^{(m)}\to M^{(m)}, \quad (h,v)\mapsto\chi(h)^{-1} hv.\]
    Repeating the constructions in \Cref{lie-alg-act} using $\m{act}'$ instead of $\m{act}$, 
    for any global section $P$ of $\s D_{H}^{(m)}$, 
    we get a morphism $\beta'(P): M^{(m)}\to M^{(m)}$. Note that
    $\beta'(L_{\xi}): M^{(m)}\to M^{(m)}$ can be identified with $L_{\xi}-\chi(\xi)$. We can thus rewrite \[L_{\xi}(b\otimes Q)
    =-L_{\xi}b\otimes Q-b\otimes F_\pi(L_\xi)Q
    =\beta'(^tL_{\xi})b\otimes Q-b\otimes   L_{\theta(\xi)}Q=\beta'(^tL_{\xi})b\otimes Q-b\otimes \alpha
    (L_{\theta(\xi)})Q\] in $\f{F}_\pi(\c{N})$ for any $b\in \FsM^{(m)}_\Q$ and $Q\in \s{D}^\d_{\hat{\b{P}},\Q}(\infty)$.

\begin{pro}\label{vanishing-higher-oper}
    For any $P\in \Gamma_{\m{inv}}(H, {\s D}_{H}^{(m)})$, $b\in \FsM^{(m)}$, $Q\in \hat{D}_{\hat{\b{V}}}^{(m)}$, we have 
    $\beta'(^tP)b\otimes Q-b\otimes \alpha(\theta P)Q$  vanishes in $F_\pi(\Ninf)^{(m)}_\Q$.
    %\[%\partial^{\left<\underline{k}\right>_{(m)}}(b\otimes Q)=    (^t\partial^{\left<\underline{k}\right>_{(m)}})'b\otimes Q-b\otimes \theta_*(\partial^{\left<\underline{k}\right>_{(m)}})Q\]
    
\end{pro}

\begin{proof}
     By \Cref{inv-generated-by-L}, we may assume $P=L_{\xi_1}\cdots L_{\xi_s}$. If $s=0$, this is trivial. Assume $s\geq 1$. Note that
     \begin{align*}%\label{vanishing-higher-oper-formula}\tag{*}
         & \beta'(^t(L_{\xi_1}\cdots L_{\xi_s}))b\otimes Q-b\otimes \alpha(\theta(L_{\xi_1}\cdots L_{\xi_s}))Q \\
         = & \beta'(^tL_{\xi_s})\cdots \beta'(^tL_{\xi_1})b\otimes Q-b\otimes \alpha(L_{\theta(\xi_1)})\cdots \alpha(L_{\theta(\xi_s)})Q \\
         = & \sum_{i=1}^{s} \Big(\beta'(^tL_{\xi_{i}})\cdots\beta'(^tL_{\xi_1})b\otimes \alpha(L_{\theta(\xi_{i+1})})\cdots \alpha(L_{\theta(\xi_s)})Q \\
         & \;\;\qquad-\beta'(^tL_{\xi_{i-1}})\cdots\beta'(^tL_{\xi_1})b\otimes \alpha(L_{\theta(\xi_{i})})\cdots \alpha(L_{\theta(\xi_s)})Q\Big)\\
         = & \sum_{i=1}^{s} L_{\xi_i}\Big(\beta'(^tL_{\xi_{i-1}})\cdots\beta'(^tL_{\xi_1})b\otimes \alpha(L_{\theta(\xi_{i+1})})\cdots \alpha(L_{\theta(\xi_s)})Q\Big).
     \end{align*}
     We conclude that $\beta'(^t(L_{\xi_1}\cdots L_{\xi_s}))b\otimes Q-b\otimes \alpha(\theta(L_{\xi_1}\cdots L_{\xi_s}))Q$ vanishes in \[F_\pi(\Ninf)^{(m)}_\Q=(\FsM^{(m)}_\Q\otimes_{\hat{A}^{(m)}_\Q}\hat{D}_{\hat{\b{V}},\Q}^{(m)})/\sum_{\xi\in \f{L}(H)}L_{\xi}(\FsM^{(m)}_\Q\otimes_{\hat{A}^{(m)}_\Q}\hat{D}_{\hat{\b{V}},\Q}^{(m)}).\qedhere\]

\end{proof}
    
\begin{defn}
    Let $m\geq0$, $P\in \Gamma(H, {\s D}_{H}^{(m)})$. We define
    \[P: \FsM^{(m)}\hat{\otimes}_{R}\hat{D}_{\hat{\b{V}}}^{(m)}\to \FsM^{(m)}\hat{\otimes}_{\hat{A}^{(m)}}\hat{D}_{\hat{\b{V}}}^{(m)}, 
    \quad b\otimes Q\mapsto\beta'(^tP)b\otimes Q-b\otimes \alpha(\theta P)Q,\]
    which is a homomorphism of right $\hat{D}_{\hat{\b{V}}}^{(m)}$-modules. 
    Note that in the source we only take tensor product over $R$.
    We then define    
    %\[F_\pi(\Ninf)^{(m)}\cong (\FsM^{(m)}\otimes_{R}\BBmm)/\sum_{\xi\in \f{L}(H), 1\leq k\leq p^m}L_{\xi}^{\left<k\right>_{(m)}}(\FsM^{(m)}\otimes_{R}\BBmm).\]
   
    %\[F_\pi(\Ninf)^{(m)}= (\FsM^{(m)}\hat{\otimes}_{\hat{A}^{(m)}}\hat{D}_{\hat{\b{V}}}^{(m)})/\sum_{\xi\in \f{L}(H), {k}\leq p^m}\nabla_\xi^{\left<{k}\right>_{(m)}}(\FsM^{(m)}\hat{\otimes}_{R}\hat{D}_{\hat{\b{V}}}^{(m)}).\]
    \[\important{F_\pi(\Ninf)^{(m)}= (\FsM^{(m)}\hat{\otimes}_{\hat{A}^{(m)}}\hat{D}_{\hat{\b{V}}}^{(m)})/\sum_{1\leq i\leq 2d, {r}\leq p^m}\nabla_{i}^{\left<{r}\right>_{(m)}}(\FsM^{(m)}\hat{\otimes}_{R}\hat{D}_{\hat{\b{V}}}^{(m)}).}\]
    This is a coherent right $\hat{D}_{\hat{\b{V}}}^{(m)}$-module. 
    \begin{comment}
    Here for $b\in \FsM^{(m)}$, $f\in \hat{D}_{\hat{\b{V}}}^{(m)}$ we define
    \[\partial^{\left<\underline{k}\right>_{(m)}}(b\otimes f):=(^t\partial^{\left<\underline{k}\right>_{(m)}})'b\otimes f-b\otimes \theta_*(\partial^{\left<\underline{k}\right>_{(m)}})f,\]
    where $L_{\xi}^{\left<i\right>_{(m)}}b$ is defined in \Cref{lie-alg-act}. 
    \end{comment}
\end{defn}

    By \Cref{vanishing-higher-oper}, we may identify 
$F_{\pi}(\Ninf)^{(m)}\otimes_{\b Z} {\Q}$ with %the quotient
\[F_\pi(\Ninf)^{(m)}_\Q=%(\FsM^{(m)}_\Q\otimes_{R}\BBmm)/\sum_{\xi\in \f{L}(H)}L_{\xi}(\FsM^{(m)}_\Q\otimes_{R}\BBmm)
(\FsM^{(m)}_\Q\otimes_{\hat{A}^{(m)}_\Q}\hat{D}_{\hat{\b{V}},\Q}^{(m)})/\sum_{\xi\in \f{L}(H)}L_{\xi}(\FsM^{(m)}_\Q\otimes_{\hat{A}^{(m)}_\Q}\hat{D}_{\hat{\b{V}},\Q}^{(m)}).\]
Hence $F_\pi(\Ninf)^{(m)}$ is indeed an integral model for $F_\pi(\Ninf)^{(m)}_\Q$.

\begin{pro}\label{quo-gr}
    Let $F_\pi(\Ninf)^{(m)}_k:=F_\pi(\Ninf)^{(m)}/\unif F_\pi(\Ninf)^{(m)}$. The good filtration on $\FsM^{(m)}_k=M^{(m)}_k$ defined in \Cref{affintmodel} gives a good filtration on $F_{\pi}(\Ninf)^{(m)}_k$, and $\gr F_{\pi}(\Ninf)^{(m)}_k$ is a quotient of
    \[\gr M^{(m)}_k\otimes_{\gr A_k^{(m)}}\Big(\gr D_{\b{V}_k}^{(m)}/\sum_{\xi\in \f{L}(H),1\leq r\leq p^m}L_{\xi}^{\left<{r}\right>_{(m)}}\gr D_{\b{V}_k}^{(m)}\Big),\]
    hence a quotient of
    \[S^{(m)}_{k,0}\otimes_{\gr A_k^{(m)}}\Big(\gr D_{\b{V}_k}^{(m)}/\sum_{\xi\in \f{L}(H),1\leq r\leq p^m}L_{\xi}^{\left<{r}\right>_{(m)}}\gr D_{\b{V}_k}^{(m)}\Big).\]
    Here $L_{\xi}^{\left<{r}\right>_{(m)}}\in \gr^r D_{\b{V}_R}^{(m)}$ is defined in \Cref{def-L^<k>}.
\end{pro}
\begin{proof}
    We identify $\FsM^{(m)}\hat{\otimes}_{\hat{A}^{(m)}}\hat{D}_{\hat{\b{V}}}^{(m)}$ with $\FsM^{(m)}\hat{\otimes}_{R}R[{\hat{\b{V}}}]$ as $R$-modules. This identification gives a right $\hat{D}_{\hat{\b{V}}}^{(m)}$-module structure on $\FsM^{(m)}\hat{\otimes}_{R}R[{\hat{\b{V}}}]$. 
    Repeating the arguments in \Cref{filonFourier}, one can show that the filtration $F^i(M^{(m)}_k\otimes_{k}k[\b{V}]):=F^iM^{(m)}_k\otimes_{k}k[\b{V}]$ 
    defines a good filtration on $M^{(m)}_k\otimes_{k}k[\b{V}]$
    with respect to the usual order filtration on $D_{\b{V}_k}^{(m)}$, and we have an isomorphism
    \[\gr(M^{(m)}_k)\otimes_{k}k[\b{V}]\cong \gr^F(M^{(m)}_k\otimes_{k}k[\b{V}]).\]
    Since $F_\pi(\Ninf)^{(m)}_k$ is a quotient of $M^{(m)}_k\otimes_{k}k[\b{V}]$, 
    we can endow it with the induced filtration. Then there exists a surjection
    \[\gr M^{(m)}_k\otimes_{\gr A_k^{(m)}}\gr D_{\b{V}_k}^{(m)}\cong \gr M^{(m)}_k\otimes_{k}k[\b{V}]\twoheadrightarrow \m{gr}\, F_\pi(\Ninf)^{(m)}_k.\]

    To verify the first claim, we need to check that, 
    for any 
    $\xi\in \f{L}(H), b\in F^iM^{(m)}_k, f\in F^j D_{\b{V}_k}^{(m)}$ and $1\leq r\leq p^m$, we have
    \[b\otimes L_{\xi}^{\left<{r}\right>_{(m)}}f
    \in  F^{i+j+r-1} F_\pi(\Ninf)^{(m)}_k.\]

    By the proof of \Cref{def-L^<k>}, each $L_{\xi}^{\left<{r}\right>_{(m)}}\in \gr^r D_{\b{V}}^{(m)}$ is contained in the left $k[\b V]$-module generated by $L_{\theta(\zeta_s)}^{\left<{r'}\right>_{(m)}}$, where $1\leq s\leq 2d$ and $1\leq r'\leq p^m$ are arbitrary. 
    We may assume $\xi=\theta(\zeta_s)$. 
    Note that the highest symbol of $\alpha(\nabla_{\theta(\zeta_s)}^{\left<{r}\right>_{(m)}})$ is $L_{\theta(\zeta_s)}^{\left<{r}\right>_{(m)}}$ by \Cref{heighest-symbol-of-alpha}. 
    It suffices to show that, for any integers $1\leq s\leq 2d$ and $1\leq r\leq p^m$, and any $b\in F^i M^{(m)}_k$ and $f\in F^j D{\b{V}_k}^{(m)}$, we have
    \[b\otimes \alpha(\theta(\nabla_{\zeta_s}^{\left<{r})\right>_{(m)}})f
    \in  F^{i+j+r-1} F_\pi(\Ninf)^{(m)}_k.\]
    
    Recall that in $F_\pi(\Ninf)^{(m)}$, we have the relation
    \[0=\nabla_{\zeta_s}^{\left<r\right>_{(m)}}(b\otimes f)=\beta'(^t\nabla_{\zeta_s}^{\left<r\right>_{(m)}})b\otimes f-b\otimes \alpha(\theta\nabla_{\zeta_s}^{\left<{r}\right>_{(m)}})f.\]
    Since the filtration on $M^{(m)}$ is compatible with the $H$-action by \Cref{affintmodel}, $\beta'(^t\nabla_{\zeta_s}^{\left<r\right>_{(m)}})b$ lies in $F^iM^{(m)}$. So $\beta'(^t\nabla_{\zeta_s}^{\left<r\right>_{(m)}})b\otimes f\in F^{i+j} F_\pi(\Ninf)^{(m)}_k$. 
    We conclude that
    $$b\otimes \alpha(\theta\nabla_{\zeta_s}^{\left<{r}\right>_{(m)}})f\in F^{i+j} F_\pi(\Ninf)^{(m)}_k\subseteq F^{i+j+r-1} F_\pi(\Ninf)^{(m)}_k.$$

    The second claim follows from \Cref{affintmodel}.
\end{proof}

\begin{pro}\label{M-support}
    Notation as in \Cref{setting}. Let $\m{Frob}: \b{V}_k\to \b{V}_k$ be the relative Frobenius.
    Then the (set-theoretic) support of $S^{(m)}_{k}$ in $\Spec (A_k^{(m)}\otimes k[t])_{\m{red}}\cong \b{V}_k\times \b{A}_k^1$ is contained in $\m{Frob}^m(Z_{0,k}'^{\m{aff}})\times \b{A}_k^1$, where $Z_{0,k}'^{\m{aff}}$ is the fiber of $Z'^{\m{aff}}_k$ at $0\in \b{A}_k^1$.
\end{pro}
\begin{proof}
    We regard $\hat{S}^{(m)}$ as a coherent sheaf on the formal scheme $\m{Spf}\, \hat{A}^{(m)}\hat\otimes R\langle t \rangle$, hence regard ${S}^{(m)}_k$ as a coherent sheaf on special fiber $\Spec A_k^{(m)}\otimes k[t]$. 
    It suffices to show $S^{(m)}_{k}$ is supported in $\m{Frob}^m(Z_k'^{\m{aff}})$.
    Recall that 
    $$S^{(m)}_{\Q}\cong \Gamma(\m{Spm}\, \hat{A}^{(m)}_{\Q}\hat{\otimes} K\langle t \rangle, \omega_{\f{Z}^{\m{aff}},\Q}),$$
    where $\f{Z}^{\m{aff}}$ is the $\unif$-adic completion of $Z^{\m{aff}}$, $\omega_{\f{Z}^{\m{aff}},\Q}$ is a coherent sheaf on the rigid space $\b{P}^{\an} \times \b{P}^{1, \an}$ supported on $(Z'^{\m{aff}}_K)^{\an}$, and $\m{Spm}\, \hat{A}^{(m)}_{\Q}\hat{\otimes} K\langle t \rangle$ is an affinoid open subset of $\b{P}_K^{\an} \times \b{P}_K^{1, \an}$. 
    Let $$\m{sp}:\Spm \hat{A}^{(m)}_{\Q}\hat{\otimes} K\langle t \rangle\to \Spec (A_k^{(m)}\otimes k[t])$$ be the canonical specialization map, which is only a morphism between topological spaces. We claim that $\m{sp}^{-1}(\m{Frob}^m(Z_{k}'^{\m{aff}}))$ contains $(Z'^{\m{aff}}_K)^{\an}\cap \Spm \hat{A}^{(m)}_{\Q}\hat{\otimes} K\langle t\rangle$. 
    Note that the generic fiber of the formal scheme $\m{Spf}\, \hat{A}^{(m)}\hat\otimes R\langle t \rangle-\m{Frob}^m(Z_k'^{\m{aff}})$ is $\m{Spm}\, \hat{A}^{(m)}_{\Q}\hat{\otimes} K\langle t \rangle-\m{sp}^{-1}(\m{Frob}^m(Z_{k}'^{\m{aff}}))$.  
    If the claim is true, then the generic fiber of $S^{(m)}|_{\m{Spf}\, \hat{A}^{(m)}\hat\otimes R\langle t \rangle-\m{Frob}^m(Z_k'^{\m{aff}})}$, which can be identified with $\omega_{\f{Z},\Q}|_{\m{Spm}\, \hat{A}^{(m)}_{\Q}\hat{\otimes} K\langle t \rangle-\m{sp}^{-1}(\m{Frob}^m(Z_{k}'^{\m{aff}}))}$, must be trivial. 
    Since $S^{(m)}\subset S^{(m)}_{\Q}$ is flat over $R$, this implies that $S^{(m)}$, hence $S^{(m)}_{k}$, is supported in $\m{Frob}^m(Z_k'^{\m{aff}})$.
    
    To prove the claim, we may replace $R$ by a finite extension. Hence we may assume it contains a $p^m$-th root $\sigma$ of $\pi$. Recall that 
    \[\hat{A}^{(m)}=R\langle\pi{x^{p^i}}, 0\leq i\leq m\rangle.\]
    Consider the homomorphism 
    \[g: R\langle y \rangle\to R\langle\pi{x^{p^i}},0\leq i\leq m\rangle, \quad y_j\mapsto \pi x_j^{p^m}.\]
    They induce the following commutative diagram of topological spaces:
    \[\begin{tikzcd}
    {\Spm K\langle \pi x, \pi x^{p^m},t\rangle} \arrow[r, "\m{sp}"] \arrow[d, "g_{\Q}"] & \Spec ( A_k^{(m)}\otimes k[t]) \arrow[d, "g_{k}"] \\
    \Spm K\langle y,t\rangle \arrow[r, "\m{sp}"]  & \Spec k[y,t].
    \end{tikzcd}\]
    We use $g_{k}$  to identify $\Spec ( A_k^{(m)}\otimes k[t])_{\m{red}}$ with $\b{V}_k\times \b{A}^1_k$.
    Assume $Z'^{\m{aff}}$ is defined by the homogeneous equations $f_j(x,t)=0$ in $\b{A}\times\b{A}^1$. %Let $y=(y_i,t_0)\in \Spm K\langle \langle \pi x, \pi x^{p^m},t\rangle\cap (Z'^{\m{aff}}_K)^{\an}\subset \b{V}_K^{\an}$,where $y_i,t_0$ are the coordinates of $y$ in $\b{V}_K^{\an}\times \b{A}^1_K$ under the chosen basis. 
    Then we may identify $g_{\Q}((Z'^{\m{aff}}_K)^{\an}\cap \Spm \hat{A}^{(m)}_{\Q}\hat{\otimes} K\langle t\rangle)$ with
    \[{\{(y^{p^m},t)\in \Spm K\langle y,t\rangle: f_j(\sigma^{-1}y,t)=0\}=\{(y^{p^m},t)\in \Spm K\langle y,t\rangle: f_j(y,\sigma t)=0\}.}\]
    The image of $g_{\Q}((Z'^{\m{aff}}_K)^{\an}\cap \Spm \hat{A}^{(m)}_{\Q}\hat{\otimes} K\langle t\rangle)$ under $\m{sp}$ is contianed in
    \[\{(y^{p^m},t)\in \Spec k[y,t] : f_j(y,0)=0\}=\m{Frob}^m(Z_{0,k}'^{\m{aff}})\times \b{A}_k^1.\]
    The claim follows. 
\end{proof}

\begin{cor}\label{nil-gr}
    Fix an algebraic closure $\algclo$ of $k$. The coherent $\c{O}_{T^{(m)*}\b{V}_{\algclo}}$-module
    \[S^{(m)}_{\algclo,0}\otimes_{A_{\algclo}^{(m)}}\Big(\gr D_{\b{V},\algclo}^{(m)}/\sum_{\xi\in \f{L}(H),1\leq r\leq p^m}L_{\xi}^{\left<{r}\right>_{(m)}}\gr D_{\b{V},\algclo}^{(m)}\Big)\Big |_{\b{V}^{\gen}_{\algclo}}\]
    is supported in the zero section and hence the ideal $\gr^{\geq 1}A_k^{(m)}$ acts nilpotently on it.
\end{cor}

\begin{proof} 
    By the same argument of \Cref{maincor}, we have $(Z'^{\m{aff}}_{0,\algclo}\times \b{V}^{\gen}_{\algclo})\cap Z(L_{\xi})=\{0\}\times \b{V}^{\gen}_{\algclo}$, where $Z(L_\xi)$ means the closed subset of $\b{V}^*_{\algclo}\times_{\algclo}\b{V}_{\algclo}$ defined by $\{L_\xi\}_{\xi\in \f{L}(H)}$.

    We can identify $\Spec (\gr D_{\b{V},\algclo}^{(m)})_{\m{red}}$ with $T^*\b{V}_{\algclo}^{(m)}\times_{\b{V}_{\algclo}^{(m)}, \m{Frob}^m} \b{V}_{\algclo}\cong \b{V}_{\algclo}^{*(m)}\times \b{V}_{\algclo}$. %where $V_k^{(m)}$ as in \Cref{def-L^<k>}.
    By \Cref{def-L^<k>}, the function $L_{\xi}^{\left<{r}\right>_{(m)}}$ vanishes in $\Spec (\gr D_{\b{V},\algclo}^{(m)})_{\m{red}}$ for $0<r<p^m$, and $L_{\xi}^{\left<{p^m}\right>_{(m)}}$ can be identified with the composite $\b{V}_{\algclo}^{*(m)}\times \b{V}_{\algclo}
    \xrightarrow{(\m{id}, \m{Frob}^m)}
    \b{V}_{\algclo}^{*(m)}\times \b{V}_{\algclo}^{(m)}\cong T^*\b{V}_{\algclo}^{(m)}\xrightarrow{L_{\xi}^{(m)}}\b{A}_{\algclo}^1$. 
    Here $L_{\xi}^{(m)}$ is the base change of 
    $$L_{\xi}: T^*\b{V}_{\algclo}\to\b{A}_{\algclo}^1$$
    along the Frobenius morphism $\m{Frob}^m: \Spec k\to \Spec k$. We can identify $L_{\xi}^{\left<{p^m}\right>_{(m)}}\circ (\m{Frob}^m, \m{id})$ with the composition of  $L_{\xi}^{(m)}$ with the absolute Frobenius morphism
    \[\m{Frob}^m_{\m{abs}}: \b{V}_{\algclo}^{*}\times \b{V}_{\algclo}\to \b{V}_{\algclo}^{*}\times \b{V}_{\algclo}.\]
    By \Cref{M-support}, the support of $S^{(m)}_{\algclo,0}$ is contained in the image of $Z'^{\m{aff}}_{0,\algclo}$ under $\b{V}_{\algclo}^{*}\xrightarrow{\m{Frob}^m} \b{V}_{\algclo}^{*(m)}$. 
    Recall that the relative Frobenius are homeomorphisms. 
    This implies
    \begin{align*}
        &\m{Supp}\, S^{(m)}_{\algclo,0}\otimes_{A_{\algclo}^{(m)}}(\gr D_{\b{V},\algclo}^{(m)}/\sum_{\xi\in \f{L}(H),1\leq r\leq p^m}L_{\xi}^{\left<{r}\right>_{(m)}}\gr D_{\b{V},\algclo}^{(m)})\Big |_{\b{V}^{\gen}_{\algclo}}\\
        =& (\m{Supp}\, S^{(m)}_{\algclo,0} \times \b{V}^{\gen}_{\algclo} )\cap Z(L_{\xi}^{\left<{p^m}\right>_{(m)}}, \xi\in \f{L}(H))\\
        =&(\m{Frob}^m, \m{id})\Big((Z'^{\m{aff}}_{0,\algclo}\times \b{V}^{\gen}_{\algclo}))\cap Z(L_{\xi}^{\left<{p^m}\right>_{(m)}}\circ (\m{Frob}^m, \m{id})) \Big)\\
        =& (\m{Frob}^m, \m{id})\Big((Z'^{\m{aff}}_{0,\algclo}\times \b{V}^{\gen}_{\algclo}))\cap Z(L_{\xi}) \Big)\\
        =&(\m{Frob}^m, \m{id})(\{0\}\times \b{V}^{\gen}_{\algclo})=\{0\}\times \b{V}^{\gen}_{\algclo}.\qedhere
    \end{align*}
\end{proof}

\begin{cor}\label{coherent}
    Let $\f F_\pi(\Ninf)^{(m)}$ be the corresponding $\hat{\s D}_{\hat{\b{V}}}^{(m)}$-module of $F_\pi(\Ninf)^{(m)}$. Then $\f{F}_{\pi}(\Ninf)^{(m)} |_{\b{V}^{\gen}_k}$ is coherent as an $\c{O}_{\hat{\b{V}}}$-module.
\end{cor}

\begin{proof}
    By \Cref{nil-gr} and \Cref{quo-gr}, $\gr^{\geq 1}A_k^{(m)}$ acts nilpotently on $\gr F_{\pi}(\Ninf)^{(m)} |_{\b{V}^{\gen}_k}$. Note that the filtration on $F_{\pi}(\Ninf)_k^{(m)}$ defined in \Cref{quo-gr} is a good filtration, implying that $\gr F_{\pi}(\Ninf)_k^{(m)}$ is a finitely generated $\gr A_k^{(m)}$-module. We have $$F^i(F_{\pi}(\Ninf)_k^{(m)}) |_{\b{V}^{\gen}_k}=F_{\pi}(\Ninf)^{(m)}_k |_{\b{V}^{\gen}_k}$$ for $i\gg 0$. Therefore, $F_\pi(\Ninf)^{(m)}_k |_{\b{V}^{\gen}_k}$ is a coherent $\c{O}_{\b{V}_k}$-module. $F_\pi(\Ninf)^{(m)}$ is complete with respect to the $\unif$-adic topology since it is a coherent $\hat{D}_{\hat{\b{V}}}^{(m)}$-module and $\hat{D}_{\hat{\b{V}}}^{(m)}$ is $\unif$-adically complete. By \cite[3.2.2]{Berthelot1}, this implies that $\f{F}_{\pi}(\Ninf)^{(m)} |_{\b{V}^{\gen}_k}$ is coherent as a $\c{O}_{\hat{\b{V}}}$-module.
\end{proof}

\begin{lem}\label{fiber-at-0}
    Let $Z'^{\m{aff}}:=Z'\cap(\b V\times\b A^1)$, and let $Z'^{\m{aff}}_{0},Z'_0$ be the fiber of $Z'^{\m{aff}},Z'$ at $0\in\b{A}^1$, respectively. We may identify $Z'_{0}\subset \b{P}$ with $\overline{Z_0'^{\m{aff}}}:=Z_0'^{\m{aff}}\cup \b{P}(Z_0'^{\m{aff}})$. 
\end{lem}

\begin{proof}
    Assume $Z'^{\m{aff}} \subset \b{V}\times \b{A}^1$ is defined by equations $f_i(v,t)=0$. 
    Assume $f_i=\sum_s t^s g_{is}(v)$. Then $g_{is}$ is a homogeneous polynomial of degree $n_i-s$. %since $Z'$ has $G_m$-action.
    So $Z_0'^{\m{aff}}$ is defined by equations $g_{i0}(v)=0$. 
    Then $Z'\cap (\mathbb P\times \mathbb A^1)$ 
    is contained in $$\{([v:x],t):\forall i, \sum_s t^s x^s g_{is}(v)=0\}.$$     It follows that $Z'_0$ is contained in \[\{([v:x],0):\forall i, g_{i0}(v)=0\}\subseteq Z_0'^{\m{aff}}\cup \b{P}(Z_0'^{\m{aff}}).\qedhere\]
\end{proof}

\begin{lem}\label{rankoverVE}
    Let $E$ be a subspace of $\f{L}(H)_{\algclo}$ of dimension equal to $d=\dim G_k$. Consider the closed subspace
    %\[S^E=\{(A, B)\in \b{V}_{k}\times (Z_0-\infty)_{k}: \frac{d}{dt}\Big|_{t=0}	f_A(e^{t\xi_1} Be^{-t\xi_2})=0, \forall (\xi_1,\xi_2)\in E\subset \f{L}(H)\}.\]
    \[S^E=\{(A, B)\in \b{V}_{\algclo}\times \b{P}(Z_{0,\algclo}'^{\m{aff}})%Z_{0,k} L_{\xi,A}\in \Gamma(\b{P, \c{O}}(1))\text{ vanishes at }B, 
    :\forall \xi\in E\subset \f{L}(H)_{\algclo}, L_{\xi,A}(B)=0\}\]
    of $\b{V}_{\algclo}\times \b{P}(Z'^{\m{aff}}_{0,\algclo})$, where $L_{\xi,A}(B)=0$ means that $\left<L_\xi(A),v\right>=0$ for any $v\in Z_{0,\algclo}'^{\m{aff}}\subseteq\b{V}_{\algclo}^*$ such that $B=[v:0]$. 
    Let $\m{pr}_1:S^E\to \b{V}_{\algclo}$ be the canonical projection morphism, $\b{V}^E_{\algclo}$ the %subspace of $A\in \b{V}_{k}$ such that the fiber $S_A^E$ has $\dim 0$, 
    complement of $\m{pr}_1(S^E)$. 
    %which is equivalent to require $S_A=\{0\}$.  %by \cite[II Exer. 3.22]{Hartshorne}, 
    %We have $\b{V}^E_{\algclo}\subset \b{V}_{\algclo}$ is open since $\m{pr}_1$ is proper.
    Then $$(\omega_{Z,0,\algclo}\otimes_{\algclo}\algclo[\b{V}])\Big/\sum_{\xi\in E}L_{\xi}(\omega_{Z,0,\algclo}\otimes_{\algclo}\algclo[\b{V}])\Big |_{\b{V}^E_{\algclo}}$$
    is supported in the zero section and is locally free of rank \[d!\int_{\Delta_{\infty}\cap\mathfrak C}
	\prod_{\alpha \in R^+} \frac{(\lambda, \alpha)^2}{(\rho, \alpha)^2}\mathrm d\lambda\]as an $\c{O}_{\b{V}_{\algclo}}$-module. 
\end{lem}
\begin{proof}
    Note that $(\omega_{Z,0,\algclo}\otimes_{\algclo}\algclo[\b{V}])/\sum_{\xi\in E}L_{\xi}(\omega_{Z,0,\algclo}\otimes_{\algclo}\algclo[\b{V}])$ is supported in $S^E$, the first assertion follows. For the second,
    by \cite[II Exer. 5.8]{Hartshorne} it suffices to check that for any closed point $A\in \b{V}^E_{\algclo}$, the fiber of $(\omega_{Z,0,\algclo}\otimes_{\algclo}\algclo[\b{V}])/\sum_{\xi\in E}L_{\xi}(\omega_{Z,0,\algclo}\otimes_{\algclo}\algclo[\b{V}])$ at $A$ has rank $\RANK$.

    Let $\bar L_{\algclo}$ be the linear subspace of $\b{P}_{\algclo}$ defined by equations $L_{\xi,A}\in \Gamma(\b{P}_{\algclo}, \c O(1))$ for $\xi\in E$. Let ${L}_{\algclo}=\bar L_{\algclo}\cap \b{V}_{\algclo}$. 
    %Embed $\bar{L}_{\xi,A}$ in to $\b{V}_{k}\times \b{A}^1_{k}$.
    By assumptions ${L}_{\algclo}\cap \b{P}(Z_{0,\algclo}'^{\m{aff}})=\emptyset$. %Both $\bar{L}_{\algclo}$ and $(Z_0-\infty)_{\algclo}$ are homogeneous, this implies
    By \Cref{fiber-at-0}, this implies $\bar{L}_{\algclo}\cap Z'_{0,\algclo}=\{0\}$. 
    In particular, $1+d=\dim Z'_{\algclo}\geq \m{codim}\,\bar{L}_{\algclo}$. As $\bar{L}_{\algclo}$ is a linear subspace of $\b{P}_{\algclo}\times \b{P}^1_{\algclo}$ defined by $1+d$ equations, this implies $d= \m{codim}\,\bar{L}_{\algclo}$.
    We identify $\b{P}_{\algclo}$ with the fiber of $\b{P}_{\algclo}\times \b{P}_{\algclo}^1\to\b{P}_{\algclo}^1$ at zero.
    The fiber of $(\omega_{Z,0,\algclo}\otimes_{\algclo}\algclo[\b{V}])/\sum_{\xi\in E}L_{\xi}(\omega_{Z,0,\algclo}\otimes_{\algclo}\algclo[\b{V}])$ at $A$ can be rewritten as $\omega_{Z,0,{\algclo}}\otimes_{\c{O}_{\b{P}_{\algclo}}}\c{O}_{\bar{L}_{\algclo}}$, which is supported in the origin $\{0\}\subset \b{V}_{\algclo}$. By \Cref{cor-G-embed}, %\Cref{ass}, 
    the sheaf $\omega_{Z,{\algclo}}$ is Cohen-Macaulay. We conclude that the derived product 
    $\omega_{Z,0,{\algclo}}\otimes^L_{\c{O}_{\b{P}_{\algclo}}} \c{O}_{\bar{L}_{\algclo}}$ is concentrated in degree $0$. Hence we have
    \[\omega_{Z,0,{\algclo}}\otimes_{\c{O}_{\b{P}_{\algclo}}}\c{O}_{\bar{L}_{\algclo}}\cong
    R\Gamma(\b{P}_{\algclo}\times \b{P}_{\algclo}^1, \omega_{Z,{\algclo}}\otimes^L_{\c{O}_{\b{P}_{\algclo}\times \b{P}^1_{\algclo}}}\c{O}_{\bar{L}_{\algclo}}).\tag{1}\label{rankoverVE-1}\]

    Let $\bar{R}$ be the maximal unramified extension of $R$.
    Lift $A$ to an element $A_{\bar{R}}\in \b{V}(\bar{R})$.
    Let $\bar{\xi}_1,\cdots, \bar{\xi}_{d}$ be a basis of $E\subset \f{L}(H)_{\algclo}$. Choose liftings $\xi_1,\cdots, \xi_{d}\in \f{L}(H)_{R}$.
    They generate a sub-$R$-module $E_{R}\subset \f{L}(H)_{R}$, which is free of rank $d$ and we have $E_R/\unif E_R\cong E$. %Let $L_{R}$ be the linear subspace of $\b{V}_{R}$ defined by equations $L_{\xi,A}$ for $\xi\in E_{R}$.
    Let $\bar{L}$ be the closed subset of $\b{P}_{\bar{R}}$ defined by $L_{\xi,A_{\bar{R}}}\in \Gamma(\b{P}_{R}, \c{O}(1))$ for $\xi\in E_{R}$.
    Since $L_{\bar{\xi}_i,A}\in \Gamma(\b{P}_{\algclo}, \c{O}(1)), 1\leq i\leq d$ are linearly independent, we may extend $L_{\xi_i,A_{\bar{R}}}\in \Gamma(\b{P}_{\bar R}, \c{O}(1))$ to a basis of $\Gamma(\b{P}_{\bar R}, \c{O}(1))$. Thus $\bar{L}\cong \b{P}_{\bar{R}}^{n-d}$. In particular, $\bar{L}$ is flat over $\Spec \bar{R}$. 
    The proper morphism $\b{P}(\bar{L}\cap Z'_{\bar{R}}\cap \b{V}_{\bar{R}})\to \Spec {\bar{R}}$ has %finite 
    empty special fiber. So $\b{P}(\bar{L}\cap Z'_{\bar{R}}\cap \b{V}_{\bar{R}})=\emptyset$. % is finite. In particular, $\bar{L}_{\bar{K}}\cap Z_{\bar{K}}$ is finite. $\bar{L}_{\bar{K}}\cap Z_{\bar{K}}=\overline{L_{\bar{K}}\cap (Z_{\bar{K}}-\infty)}$, $L_{\bar{K}}\cap (Z_{\bar{K}}-\infty)$ is homogeneous 
    In other words, $\bar{L}_{\bar{K}}\cap Z'_{\bar{K}}\cap \b{V}_{\bar K}$ equals $\{0\}$. 
    By \Cref{fiber-at-0}, this implies $\bar{L}\cap Z'_{\bar{R}}=\{0\}$. 
    %In particular, $1+\dim G=\dim Z\geq \m{codim}\,\bar{L}_{\bar{K}}$. As $L\subset \b{V}\times\b{A}^1$ is a linear subspace defined by $1+\dim G$ equations, this implies $1+\dim G= \m{codim}\,\bar{L}_{\bar{K}}$.
    %\WARN{By the arguments in \cite{l-adic} this implies that}
    Hence, the complex $\omega_{Z,\bar{K}}\otimes^L_{\c{O}_{\b{P}_{\bar K}\times \b{P}^1_{\bar K}}}\c{O}_{\bar{L}_{\bar{K}}}$ is supported in $\{0\}$. 
    We thus have identifications 
    \begin{align*}R\Gamma(\b{P}_{\bar{K}}\times \b{P}_{\bar{K}}^1, \omega_{Z,\bar{K}}\otimes^L_{\c{O}_{\b{P}_{\bar{K}}\times \b{P}_{\bar{K}}^1}}\c{O}_{\bar{L}_{\bar{K}}})
    & \cong \m{Tor}^{\c{O}_{\b{P}_{\bar{K}}\times \b{P}_{\bar{K}}^1}}_{\bullet}(\omega_{Z,\bar{K}}, \c{O}_{\bar{L}_{\bar{K}}})
    \\ & \cong \m{Tor}^{\c{O}_{\b{V}_{\bar{K}}\times \b{A}_{\bar{K}}^1}}_{\bullet}(\omega_{Z,\bar{K}}|_{\b{V}_{\bar{K}}\times \b{A}_{\bar{K}}^1}, \c{O}_{\bar{L}_{\bar{K}}}|_{\b{V}_{\bar{K}}\times \b{A}_{\bar{K}}^1}).\end{align*}
    By the arguments in \cite[3.3]{l-adic}, the latter is concentrated in degree $0$ and 
    %$$R\Gamma(\b{P}\times \b{P}^1, \omega_{Z,\bar{K}}\otimes^L_{\c{O}_{\b{P}\times \b{P}^1}}\c{O}_{\bar{L}_{\bar{K}}})   \cong R^0\Gamma(\b{P}\times \b{P}^1, \omega_{Z,\bar{K}}\otimes^L_{\c{O}_{\b{P}\times \b{P}^1}}\c{O}_{\bar{L}_{\bar{K}}})$$ 
    has dimension $\RANK$ as a $\bar{K}$-vector space.
    \begin{comment}
    %Choose lifting $L\subset \b{V}_{R}$. Let $\bar L\subset $
    %Let $S=$
    Note
    %Since $\omega_{Z, \algclo}$ is Cohen-Macaulay, %the dimension of the fiber equals to    $\chi(\omega_{Z_{\algclo}},\c{O}_{L_{k}})$. Note
    \begin{align*}
        \Gamma(\b{P}\times \b{P}^1, \omega_{Z,{\bar{k}(\f{p})}}\otimes_{\c{O}_{\b{P}\times \b{P}^1}}\c{O}_{\bar L_{\bar{k}(\f{p})}})
        \cong&\Gamma(\b{P}\times \b{P}^1, \omega_{Z,{\bar{k}(\f{p})}}\otimes^L_{\c{O}_{\b{P}\times \b{P}^1}}\c{O}_{\bar L_{\bar{k}(\f{p})}})\\
        \cong&%都flat over k(p)
        \Gamma(\b{P}\times \b{P}^1, \omega_{Z,R}\otimes^L_{\c{O}_{\b{P}\times \b{P}^1}}\c{O}_{\bar L})\otimes^L_{R} \bar{k}(\f{p})\\
        \cong&%\omega_{Z}的定义
        \Gamma(\b{P}\times \b{P}^1, (\iota p)_*\omega_{\tilde{Z}, R}\otimes^L_{\c{O}_{\b{P}\times \b{P}^1}}\c{O}_{\bar L})\otimes^L_{R} \bar{k}(\f{p})\\
        \cong& \Gamma(\tilde Z_{R}, \omega_{\tilde{Z}, R}\otimes^L L(\iota p)^*\c{O}_{\bar L})\otimes^L_{R} \bar{k}(\f{p})\\
        {\cong}&%\omega_{\tilde{Z}} locally free
        \Gamma((\iota p)^{-1}L, \omega_{\tilde{Z}, R}|_{(\iota p)^{-1}L})\otimes^L_{R} \bar{k}(\f{p}).
    \end{align*}
    Same arguments shows $\Gamma(\b{P}\times \b{P}^1, \omega_{Z,{k}}\otimes_{\c{O}_{\b{P}\times \b{P}^1}}\c{O}_{\bar L_{k}})
        \cong \Gamma((\iota p)^{-1}L, \omega_{\tilde{Z}, R}|_{(\iota p)^{-1}L})\otimes^L_{R} \bar{K}$.
    \end{comment}

    Choose a finite locally free resolution $F^\bullet$ of $\c{O}_{\bar{L}_{\bar{R}}}$ as a coherent $\c{O}_{\b{P}_{\bar{R}}\times \b{P}^1_{\bar{R}}}$-module.
    Note that both $\b{P}_{\bar{R}}\times \b{P}^1_{\bar{R}}$, $\omega_{Z,\bar{R}}$ and ${\bar{L}_{\bar{R}}}$ are flat over $\bar{R}$.
    The complex $F^\bullet\otimes_{\bar{R}}\algclo$ calculates $\c{O}_{\bar{L}_{\bar{R}}}\otimes^L_{\bar{R}}\algclo\cong \c{O}_{\bar{L}_{\algclo}}$. So it is a finite locally free resolution of $\c{O}_{\bar{L}_{\algclo}}$ as $\c{O}_{\b{P}_{\algclo}\times \b{P}^1_{\algclo}}$-module. 
    We may thus identify $\omega_{Z,{\algclo}}\otimes_{\c{O}_{\b{P}_{\algclo}\times \b{P}_{\algclo}^1}}\c{O}_{\bar{L}_{\algclo}}$ with the complex $(\omega_{Z,\bar{R}}\otimes_{\b{P}_{\bar{R}}\times \b{P}_{\bar{R}}^1} F^\bullet)\otimes_{\bar{R}}\algclo$.
    %Since $\tilde Z_{R}$ is proper over $R$, %$\omega_{\tilde{Z}, R}$ is locally free% so flat over $R$
    Then we conclude
    \begin{align*}\tag{2}\label{rankoverVE-2}
    &\chi(R\Gamma(\b{P}_{\algclo}\times \b{P}_{\algclo}^1, \omega_{Z,{\algclo}}\otimes^L_{\c{O}_{\b{P}_{\algclo}\times \b{P}_{\algclo}^1}}\c{O}_{\bar{L}_{\algclo}}))\\
    =&\sum_i (-1)^i \chi(R\Gamma(\b{P}_{\algclo}\times \b{P}_{\algclo}^1,(\omega_{Z,\bar{R}}\otimes_{\b{P}_{\bar{R}}\times \b{P}_{\bar{R}}^1} F^i)\otimes_{\bar{R}}\algclo))\\
    =&\sum_i (-1)^i \chi(R\Gamma(\b{P}_{\bar{K}}\times \b{P}_{\bar{K}}^1,(\omega_{Z,\bar{R}}\otimes_{\b{P}_{\bar{R}}\times \b{P}_{\bar{R}}^1} F^i)\otimes_{\bar{R}}\bar{K}))\\
    =&\chi(R\Gamma(\b{P}_{\bar{K}}\times \b{P}_{\bar{K}}^1, \omega_{Z,\bar{K}}\otimes^L_{\c{O}_{\b{P}_{\bar{K}}\times \b{P}_{\bar{K}}^1}}\c{O}_{\bar{L}_{\bar{K}}}))\\
    =&\RANK.
        %\chi(R\Gamma((\iota p)^{-1}L, \omega_{\tilde{Z}, R}|_{(\iota p)^{-1}L})\otimes^L_{R} \algclo)
        %=&\chi(R\Gamma((\iota p)^{-1}L, \omega_{\tilde{Z}, R}|_{(\iota p)^{-1}L})\otimes^L_{R} K)\\
        %=&\chi(\omega_{Z,K}, \c{O}_{L_{K}})=\RANK.
    \end{align*}
    Here the second equality follows from the fact that $\omega_{Z,\bar{R}}\otimes_{\b{P}_{\bar{R}}\times \b{P}_{\bar{R}}^1} F^i$ is a coherent sheaf flat over $\bar R$. The assertion then follows from (\ref{rankoverVE-1}) and (\ref{rankoverVE-2}). 
    \end{proof}

\begin{lem}\label{union-V^E}
    The union of $\b{V}^E_{\algclo}$ for all linear subspaces $E\subset \f{L}(H)_{\algclo}$ of dimension $d$ is $\b{V}^\gen_{\algclo}$.
\end{lem}

\begin{proof}
    By \Cref{orbitstr}, the open subset $\b{V}^\gen_{\algclo}\subset \b{V}_{\algclo}$ %consists of elements $A$ in $\b{V}_{\algclo}$ such that the fiber of 
    can be identified with the complement of the projection of 
    \[S=\{(A, B)\in \b{V}_{\algclo}\times \b{P}(Z_{0,\algclo}'^{\m{aff}})%Z_{0,\algclo} L_{\xi,A}\in \Gamma(\b{P, \c{O}}(1))\text{ vanishes at }B, 
    :\forall \xi\in \f{L}(H)_{\algclo}, L_{\xi,A}(B)=0\}\]
    %\forall[v:0]\in \b{P}(Z_{0,\algclo}'^{\m{aff}})
    %at $A$ has dimension $0$. 
    to the first factor $\b{V}_{\algclo}$. 
    Let $A\in \b{V}^\gen_{\algclo}$. % It suffice to show there exists a linear subspace $F\subset \f{L}(H)_{\algclo}$ of dimension $d$ such that $Z(L_{\xi,A}, \xi\in E)\subseteq Z'^{\m{aff}}_{0,\algclo}$
    For any integer $0\leq r\leq d$, we claim that there exists a linear subspace $F\subset \f{L}(H)_{\algclo}$ of dimension $r$ such that $Z'^{\m{aff}}_{0,\algclo}\cap Z(L_{\xi,A}, \xi\in F)$ has dimension $\leq d-r$. Induction on $r$. For $r=0$, this is trivial. When $r\geq 1$, assume $F$ is a linear subspace of dimension $(r-1)$ such that $Z'^{\m{aff}}_{0,\algclo}\cap Z(L_{\xi,A}, \xi\in F)$ has dimension $\leq d-r+1$. Let $\eta_1,\cdots, \eta_s$ be the generic points of $Z'^{\m{aff}}_{0,\algclo}\cap Z(L_{\xi,A}, \xi\in F)$ of dimension $(d-r+1)$. 
    In particular, each $\eta_i$ is not a closed point, so $\eta_i\notin \{0\} = \ Z'^{\m{aff}}_{0,\algclo}\cap Z(L_{\xi,A}, \xi\in \f{L}(H))$. Therefore, 
    \[U_i:=\{ \xi\in \f{L}(H): \eta_i\notin Z(L_{\xi,A})\}\]
    is a nonempty open subset of $\f{L}(H)$. Note that $\f{L}(H)$ is irreducible. We have $\cap_i U_i$ is also nonempty. Choose $\xi_0\in \cap_i U_i$ and define $F'=\m{span}\{F, \xi_0\}$. Clearly $F'$ is a linear subspace of dimension $r$ such that $Z'^{\m{aff}}_{0,\algclo}\cap Z(L_{\xi,A}, \xi\in F')$ has dimension $\leq d-r$. 
\end{proof}

\begin{cor}\label{rank}
    $\f{F}_{\pi}(\Ninf)^{(0)}|_{\b{V}^{\gen}_k}$ is locally a quotient of a free $\c{O}_{\hat{\b{V}}}$-module of rank \[\RANK.\]
\end{cor}
\begin{proof}
    By \Cref{coherent}, $\f{F}_{\pi}(\Ninf)^{(0)}|_{\b{V}^{\gen}_k}$ is a coherent $\c{O}_{\hat{\b{V}}}$-module. By Nakayama's lemma, 
    it suffices to show that for any closed point $A\in \b{V}^{\gen}_k$, the fiber of $\f{F}_{\pi}(\Ninf)^{(0)}_{\algclo}$ at $A$ has dimension $\leq \RANK$. 

    By \Cref{union-V^E}, we may find a linear subspace $E\subset \f{L}(H)_{\algclo}$ of dimension $d$ such that $A\in \b{V}^E_{{\algclo}}$. 
    %this on each $\b{V}^E_{{\algclo}}$.
    By \Cref{rankoverVE}, each $$\gr^i(\omega_{Z,0,\algclo}\otimes_{{\algclo}}{\algclo}[\b{V}])\Big/\sum_{\xi\in E}L_{\xi}(\omega_{Z,0,\algclo}\otimes_{{\algclo}}{\algclo}[\b{V}])\Big |_{\b{V}^E_{{\algclo}}}$$
    is a locally free $\c{O}_{{\b{V}}_{\algclo}}$-module. The total rank for $i\in \b{Z}$ is $\RANK$. By \Cref{S0=omega} and \Cref{quo-gr}, we have surjections 
    \[\gr^i(\omega_{Z,0,\algclo}\otimes_{{\algclo}}{\algclo}[\b{V}])/\sum_{\xi\in E}L_{\xi}(\omega_{Z,0,\algclo}\otimes_{{\algclo}}{\algclo}[\b{V}])\Big |_{\b{V}^E_{{\algclo}}}\twoheadrightarrow \gr^i\f F_{\pi}(\Ninf)^{(0)}_{\algclo}\Big |_{\b{V}^E_{{\algclo}}}.\]
    Therefore, the fiber of $\f{F}_{\pi}(\Ninf)^{(0)}_{\algclo}$ at $A$ has dimension $\leq \RANK$.
    %\[\leq \sum_i \m{rank}\, \gr^i(\omega_{Z,0}\otimes_{{\algclo}}{\algclo}[\b{V}])/\sum_{\xi\in E}L_{\xi}(\omega_{Z,0}\otimes_{{\algclo}}{\algclo}[\b{V}])\Big |_{\b{V}^E_{{\algclo}}}=\RANK.\] 
    %Thus locally we have a surjection
    %\[\c{O}_{\b{V},{\algclo}}^{\oplus \RANK}\twoheadrightarrow \f{F}_{\pi}(\Ninf)^{(0)}_{\algclo}.\]
    %Apply $$, we gets locally a surjection
%    \[\c{O}_{\b{V},{\algclo}}^{\oplus \RANK}\twoheadrightarrow F_{\pi}(\Ninf)^{(s)}_{\algclo}.\]
    %, locally we may find a surjection
    %\[\c{O}_{\hat{\b{V}}}^{\oplus \RANK}\twoheadrightarrow \f{F}_{\pi}(\Ninf)^{(0)}.\]
\end{proof}

\begin{thm}\label{2ndmainthm}
    $\f{F}_{\pi}(\c{N})$ restricted to $\mathbb V_{k}^{\mathrm{gen}}$ is $\c{O}_{\hat{\mathbb V}, \Q}$-coherent. Moreover, for any divisor $T$ of $\b{V}_k$ containing the complement of $\mathbb V_{k}^{\mathrm{gen}}$, 
    the associated isocrystal $\m{sp}^*(\f{F}_{\pi}(\c{N})|_{\b{V}_k-T})$ on $\mathbb V_{k}-T$ overconvergent along $T$ 
    has rank less or equal to 
    \[d!\int_{\Delta_{\infty}\cap\mathfrak C}
	\prod_{\alpha \in R^+} \frac{(\lambda, \alpha)^2}{(\rho, \alpha)^2}\mathrm d\lambda.\]
\end{thm}
\begin{proof}
    %By \cite[2.2.13]{Caro-L}, it suffices to show

    By \Cref{coherent}, for each $m\geq 0$,
    $\f{F}_{\pi}(\Ninf)^{(m)}|_{\mathbb V_{k}^{\mathrm{gen}}}$ is coherent as $\c{O}_{\b{V}}$-module. We claim that $\f{F}_{\pi}(\Ninf)_\b{Q}$ is $\c{O}_{\hat{\mathbb V}, \Q}$-coherent and the morphism
    \[\f{F}_{\pi}(\Ninf)^{(0)}_{\Q}\to \f{F}_{\pi}(\Ninf)_\Q\]
    is surjective. If the claim is true, it implies that $\f{F}_{\pi}(\Ninf)_{\Q}$ has rank $\leq \RANK$ by \Cref{rank}. %The natural map $\hat{\s D}_{\hat{\b{V}},\Q}^{\dagger}$

    To prove the claim above, we refer to the proof of \cite[2.2.14]{Caro-L}. 
    %We need to show that for any affine open subset $U\subset \mathbb V_{k}^{\mathrm{gen}}$, the map. 
    Let $U\subset \mathbb V_{k}^{\mathrm{gen}}$ be an affine open subset. 
    Note $U\hookrightarrow \hat{\b{V}}$ defines a local coordinate on $U$. 
    For any $m\geq 0$, by \cite[2.2.9]{Caro-L}, the canonical map $\Gamma(U, \f{F}_\pi(\Ninf)^{(0)}_{\Q})\to \Gamma(U, \f{F}_\pi(\Ninf)^{(m)}_{\Q})$ %\otimes_{\hat{D}_{\hat{\b{V}}}^{(0)}}\hat{D}_{\hat{\b{V}}}^{(m)}$ 
    has dense image. Here the topology on $\Gamma(U, \f{F}_\pi(\Ninf)^{(m)}_{\Q})$ is induced by its $\hat{\s D}_{\hat{\b{V}},\Q}^{(m)}$-module structure (see \cite[4.1.1]{Berthelot1}). By \cite[4.1.2]{Berthelot1}, it is equivalence to the topology induced by its $\c{O}_{\hat{\b{V}},\b{Q}}$-module structure since $\Gamma(U, \f{F}_\pi(\Ninf)^{(m)}_{\Q})$ is a coherent $\c{O}_{\hat{\b{V}},\b{Q}}$-module by \Cref{coherent}. %Then the surjective morphism $F_\pi(\Ninf)^{(0)}\otimes_{\hat{D}_{\hat{\b{V}}}^{(0)}}\hat{D}_{\hat{\b{V}}}^{(m)}\to F_\pi(\Ninf)^{(m)}$ implies that $F_\pi(\Ninf)^{(0)}\to F_\pi(\Ninf)^{(m)}$ has dense image. 
    Then \cite[3.7.3.1]{NonArchimedean} implies that the image of $\Gamma(U, \f{F}_\pi(\Ninf)^{(0)}_{\Q})\to \Gamma(U, \f{F}_\pi(\Ninf)^{(m)}_{\Q})$ is closed. Hence the map $\Gamma(U, \f{F}_\pi(\Ninf)^{(0)}_{\Q})\to \Gamma(U, \f{F}_\pi(\Ninf)^{(m)}_{\Q})$ is a surjection. Applying $\varinjlim_m$, we see that the map
    $\Gamma(U, \f{F}_\pi(\Ninf)^{(0)}_{\Q})\to \Gamma(U, \f{F}_\pi(\Ninf)_{\Q})$
    is surjective. By \cite[2.2.13]{Caro-L}, $\f{F}_\pi(\Ninf)_{\Q}|_{U}$ is $\c{O}_{\hat{\b{V}},\b{Q}}|_U$-coherent. 
    We conclude that $\f{F}_{\pi}(\c{N})$ restricted to $\mathbb V_{k}^{\mathrm{gen}}$ is $\c{O}_{\hat{\mathbb V}, \Q}$-coherent. 
    Then $\f{F}_{\pi}(\Ninf)_{\Q}$ is an isocrystal on $\mathbb V_{k}^{\mathrm{gen}}$ overconvergent along $T$ by \cite[2.2.12]{Caro-L}. 
\end{proof}

\begin{proof}[Proof of \Cref{mainthm}]
    Follows from \Cref{2ndmainthm} and \Cref{directsummand}.
\end{proof}

\begin{cor}\label{locallyfree}
    For any closed point $x\in \mathbb V_{k}^{\mathrm{gen}}$, $i_x^{!} \f{F}_{\pi}((\iota p)_{+}\c{O}_{G_k})[n],i_x^+ \f{F}_{\pi}((\iota p)_{+}\c{O}_{G_k})[-n]\in D^b(K)$ are concentrated in degree $0$ and have rank \[\RANK,\] where $n$ is the dimension of $\b{V}$. 
\end{cor}

\begin{proof}
    By \Cref{mainthm}, it follows from \cite[6.3.10]{Caro19padic} and \cite[5.6]{Abe2011ExplicitCO}.
\end{proof}

\section{Application: Exponential sums}\label{exponential}

Notation as in \Cref{Hyper-D}. In this section, we need to specify the divisors associated with our morphisms. Denote by $\infty$
the divisors $\b{P}_k-\b{V}_k$ in $\b{P}_k$, or $\b{P}_k\times \b{P}_k-\b{V}_k\times\b{V}_k$ in $\b{P}_k\times \b{P}_k$.
A morphism of pairs $f:(X_1,X_1')\to(X_2,X_2')$ means a morphism $f:X_1\to X_2$ such that $f(X_1-X_1')\subseteq(X_2-X_2')$, where $X_1'\subseteq X,X_2'\subseteq X_2$ are divisors.

Consider $\iota p:(\tilde{\f{Y}},\tilde D_k)\to(\hat{\b{P}},\infty)$. We have $(\iota p)_+:D_{\m{ovhol}}^b(\s{D}_{\tilde{\f{Y}},\b{Q}}^\dagger(\tilde{D}_k))\to D_{\m{ovhol}}^b(\s{D}_{\hat{\b{P}},\b{Q}}^\dagger(\infty))$.

\begin{lem}\label{Fourier}
    We have $\f{F}_\pi((\iota p)_+\c{O}_{G_k})\cong \m{Hyp}_+[2-n-d]$ as $\s{D}_{\hat{\b{P}},\b{Q}}^\dagger(\infty)$-modules, where
    $d=\dim G_k$. In particular, $\m{Hyp}_+[2-n-d]$ restricted to $\mathbb V_{k}^{\mathrm{gen}}$ is locally free as an $\c{O}_{\hat{\mathbb V}, \Q}$-module by \Cref{locallyfree}.
\end{lem}

\begin{proof}
 We use  the proper base-change theorem for overholonomic module in \cite[1.3.10]{Abe2013TheoryOW}, and the projection formula in \cite[3.3.5]{Caro19padic}.

    Let $p_1,p_2$ be the projections of $(\hat{\b{P}}\times \hat{\b{P}},\infty)$ to its factors,  and
    let $\pi_1,\pi_2$ be the canonical projections of $(\tilde{\f{Y}}\times \hat{\b{P}},(\tilde{Y}\times\infty)\cup(\tilde{D}\times\b{P}))$
    to its factors. Note that $\iota p$ is a proper morphism between smooth $R$-formal schemes. We have %and $p_1$ is smooth
    \begin{align*}
       \f{F}_\pi((\iota p)_+\c{O}_{G_k})[n-2] & \cong p_{2,+}(\c{K}_\pi\otimes^L (\iota p\times\m{id})_+\pi_1^!\c{O}_{G_k})[-2n] \\ & \cong p_{2,+}(\iota p\times\m{id})_+(L(\iota p\times\m{id})^*\c{K}_\pi\otimes^L \pi_1^!\c{O}_{G_k})[-2n] \\ & \cong \pi_{2,+}(\iota p\times\m{id})^![n-d]\c{K}_\pi[n][-2n] \\ & = \m{Hyp}_+[-d]\qedhere.
    \end{align*}
\end{proof}

To estimate the upper bound of the weights, we need to consider the hypergeometric $\s{D}_{\hat{\b{P}},\b{Q}}^\dagger(\infty)$-module with
proper support $\m{Hyp}_!=\b{D}_{\infty}(\m{Hyp}_+)$.

\begin{thm}\label{weight}
    $\mathcal H^j(\m{Hyp}_!)=0$ for $j\not=n+d-2$. Moreover, $\m{Hyp}_!$ is mixed of weight $\leq 2$ in the sense of \cite{Abe2013TheoryOW}.
\end{thm}
\begin{proof}
    By \Cref{Fourier} and \Cref{liesdeg0}, we have $\mathcal H^j(\m{Hyp}_+)=0$ for $j\not=2-d-n$.
    The first assertion follows by duality.
    Recall that as in \cite[A.2]{Abe2011ExplicitCO}, we have $\c{K}_\pi=h_+\lambda^!(\c{L}_\pi[-1])$, where $h:\f{X}\to\hat{\b{P}}\times\hat{\b{P}}$
is a proper birational morphism, $\lambda:\f{X}\to\hat{\mathbb{P}}^1$ is a morphism, and $\c{L}_\pi$ is the $\s{D}^\dagger(\infty)$-module
associated  with the Dwork isocrystal. Since the eigenvalues of the Frobenius on stalks of $\c{L}_\pi$ have norm $1$, $\c{L}_\pi$
is pure of weight $-1$ in the sense of \cite{Abe2013TheoryOW}. By the main theorem of \cite{Abe2013TheoryOW}, \[\m{Hyp}_+=\pi_{2,+}((\iota p)\times\m{id})^!h_+\lambda^!(\c{L}_\pi[-1])\] is mixed of weight $\geq -2$ and $\m{Hyp}_!=\b{D}_{\infty}(\m{Hyp}_+)$ is mixed of weight $\leq 2$.
\end{proof}

\subsection{}

Let \[\theta(z)=\exp \left(\pi z-\pi z^p\right), \quad \theta_m(z)=\exp \left(\pi z-\pi z^{p^m}\right)=\prod_{i=0}^{m-1} \theta\left(z^{{p^i}}\right).\]
Then $\theta_m(z)$ converges in a disc of radius $>1$, and the value $\theta(1)=\left.\theta(z)\right|_{z=1}$ of the power series $\theta(z)$ at $z=1$ is a primitive $p$-th root of unity in $K$ (\cite[4.1, 4.3]{Monsky1970PadicAA}). Let $u \in \mathbb{F}_{p^m}$ and let $u' \in \overline{\mathbb{Q}}_p$ be its Techmüller lifting so that $u'^{p^m}=u'$. Then we have (\cite[4.4]{Monsky1970PadicAA})
\[\left.\theta_m(z)\right|_{z=u'}= \theta(1)^{\operatorname{Tr}_{\mathbb{F}_{p^m} / \mathbb{F}_p}(u)}.\]

Let $q$ be a power of $p$ such that there exists an isomorphism $k\cong\b{F}_q$, where $k$ is the residue field of $R$. Let $\psi_m: \mathbb{F}_{q^m} \rightarrow K^*$ be the additive character defined by
\[\psi_m(u)=(\left.\exp\left(-\pi z-(-\pi) z^{q^m}\right)\right|_{z=1})^{\operatorname{Tr}_{\mathbb{F}_{q^m} / \mathbb{F}_p}(u)}.\] Then we have $\psi_m(u)=\left.\exp \left(-\pi z-(-\pi) z^{q^m}\right)\right|_{z=u'}$. Denote $\psi_1$ by $\psi$. We have $\psi_m=\psi \circ \operatorname{Tr}_{\mathbb{F}_{q^m} / \mathbb{F}_q}$.

\subsection{}\label{calculation}
For any closed point $A\in\b{V}_k$, let $k(A)$ be the residue field at $A$. By the base change theorem \cite[5.2.5]{Caro19padic},
we have \[i_A^+\m{Hyp}_!\cong l_!((\iota p)\times i_A)^+\b{D}_\infty(\c{K}_\pi),\]
where $i_A:\m{Spec}(k(A))\to \b{V}_k$ is the canonical closed immersion, and $l$ is the structure morphism of $\tilde{Y}_k$. We denote $\c{Q}:=((\iota p)\times i_A)^+\b{D}_\infty(\c{K}_\pi)$ for short.

By results in \cite{Carosurhol}, $\c{Q}$ is an overholonomic $F$-$\s{D}_{\tilde{\f{Y}},\b{Q}}^\dagger(\tilde{D}_k)$-module since it is obtained by applying six functors to the Dwork isocrystal. In our case $\tilde{Y}_k$ is the special fiber of $\tilde{\f{Y}}$. It implies that every $F$-$\s{D}_{\tilde{\f{Y}},\b{Q}}^\dagger(\tilde{D}_k)$-module is arithmetic, i.e., it is contained in the category defined in \cite[4.1]{Carosurhol}.
By \cite[6.5]{Carosurhol} we have \begin{equation}\label{Lfunction}\prod_{j\in\b{Z}}\m{det}_K(1-t\m{Fr}\,|\,\c{H}^j(l_!\c{Q}))^{(-1)^{j+1}}=\prod_{g\in |G_k|}\prod_{j\in\b{Z}}\m{det}_K(1-t\m{Fr}\,|\,\c{H}^j(l_{g,+}i_g^+\c{Q}))^{(-1)^{j+1}},\end{equation}
where $l_g:\m{Spec}(k(g))\to\m{Spec}(k)$ and $i_g:\m{Spec}(k(g))\to\tilde{Y}_k$ are the canonical morphisms,
and $\bar{k}$ is an algebraic closure of $k$.
Note that \[\m{det}_K(1-t\m{Fr}\,|\,l_{g,+}\c{E})=\m{det}_K(1-t^{\m{deg}(g)}\m{Fr}^{\m{deg}(g)}\,|\,\c{E}_g)^{\frac1{\m{deg}(g)}}\] by \cite[3.3.5]{surcoherent}, where $\c{E}$ is a coherent $\s{D}^\dagger$-module coming from an overconvergent isocrystal. %It follows that the formal Taylor expansion of $\m{det}_K(1-t\m{Fr}\,|\,l_{g,+}\c{E})$ does not have term $t$ if $\m{deg}(g)>1$.
In particular, we have \begin{equation*}\m{Tr}_K(\m{Fr}\,|\,l_!\c{Q}):=\sum_{j\in\b{Z}}(-1)^{j}\m{Tr}_K(\m{Fr}\,|\,\c{H}^jl_!\c{Q})= \sum_{g\in G_k(k)}\m{Tr}_K(\m{Fr}\,|\,i_g^+\c{Q})\end{equation*} since they are both the coefficient of $t$ in \Cref{Lfunction}.
Consider the morphism \[F:\f{G}\times\hat{\b{V}}\to\hat{\mathbb{A}}^1,\quad(g,(A_1,\cdots,A_N))\mapsto\sum_{j=1}^N\m{Tr}(\rho_j(g)A_j).\] By the definition of $\c{K}_{\pi}$, we have\[i_g^+\c{Q}=((i_g\circ\iota p)\times i_A)^+\b{D}_\infty(\c{K}_\pi)|_{\f{G}\times\hat{\b{V}}}\cong ((i_g\circ\iota p)\times i_A)^+F^+(\c{L}_{\pi^{-1}}|_{\hat{\mathbb{A}}^1})(-1)[1].\] Hence we can rewrite the trace formula above as
\begin{equation}\label{traceformula}\m{Tr}_K(\m{Fr}\,|\,l_!\c{Q})= \sum_{g\in G_k(k)}\m{Tr}_K(\m{Fr}\,|\,F_{g,A}^+(\c{L}_{-\pi}|_{\hat{\mathbb{A}}^1})(-1)[1]),\end{equation} where $F_{g,A}:=F\circ((i_g\circ \iota p)\times i_A)$.

Recall that the Frobenius structure on $\c{L}_\pi$ is just $\exp(\pi x-\pi x^q)$ if we regard $\c{L}_\pi$ as a $K\langle x\rangle^\d$-module. By the purity theorem \cite[5.6]{Abe2011ExplicitCO} and the flatness of $\c{L}_{-\pi}|_{\hat{\mathbb{A}}^1}$ over $\c{O}_{\hat{\mathbb{A}}^1}$, we have \[F_{g,A}^+(\c{L}_{-\pi}|_{\hat{\mathbb{A}}^1})(-1)[1]\cong F_{g,A}^*(\c{L}_{-\pi}|_{\hat{\mathbb{A}}^1})[2].\] Therefore, it is concentrated in degree $-2$ and its eigenvalue of the Frobenius action is given by $\exp(-\pi x-(-\pi) x^q)$. %\cite[4.3]{p-adic-Weil-II}
It follows that \begin{equation}\label{traceofdwork}
\begin{aligned}
    \m{Tr}_K(\m{Fr}\,|\,F_{g,A}^+(\c{L}_{-\pi}|_{\hat{\mathbb{A}}^1})(-1)[1])=(-1)^2\exp(-\pi z-(-\pi) z^q)|_{z=y'}=\psi(\sum_{j=1}^N\m{Tr}(A_j\rho_j(g)),
\end{aligned}\end{equation} where $y'$ is the Techm\"uller lifting of $y:=\sum_{j=1}^NA_j\rho_j(g))$.

In summary, we have the following corollary:

\begin{cor}
    Notation as above, for any closed point $A\in \b{V}^{\m{gen}}_k$, we have\[\left|\sum_{g\in G(k)}\psi\Big(\sum_{j=1}^N\m{Tr}(\bar{A}_j\bar{\rho}_j(g))\Big)\right|\leq q^{\frac{d}{2}}d!\int_{\Delta_{\infty}\cap\mathfrak C}
	\prod_{\alpha \in R^+} \frac{(\lambda, \alpha)^2}{(\rho, \alpha)^2}\mathrm d\lambda.\]
\end{cor}

\begin{proof}
    By \cite[5.6]{Abe2011ExplicitCO} again, \Cref{weight} implies that $\c{H}^j(i_A^+\m{Hyp}_!)=0$ for $j\neq d-2$. Moreover, $i_A^+\m{Hyp}_!$ is mixed of weight $\leq 2$ by the main theorem of \cite{Abe2013TheoryOW}. Then by Equations (\ref{traceformula}), (\ref{traceofdwork}) above, we have \[\m{LHS}=|\m{Tr}_K(\m{Fr}\,|\,i_A^+\m{Hyp}_!)|\leq q^{\frac {d}{2}}\cdot\m{rank}(\m{Hyp}_!|_{\b{V}_k^{\m{gen}}}).\] %by \Cref{weight} and the main theorem of \cite{Abe2013TheoryOW}.
    It follows that \[\m{LHS}\leq q^{\frac {d}2}\cdot\m{rank}(\f{F}_\pi((\iota p)_+\c{O}_{G_k})|_{\b{V}_k^{\m{gen}}})\] by \Cref{Fourier} and the
    fact that $\m{Hyp}_!=\b{D}_{\infty}\m{Hyp}_+$. We then apply \Cref{mainthm} for the upper bound of the rank.
\end{proof}

\begin{appendix}
\section{Spherical varieties over Dedekind domains}%Proof of \Cref{ass-true-on-dense}}
\label{Appendix}
    In this appendix, when we use the notation in the introduction, we omit the subscript $R$ of $R$-schemes. Notation as in \ref{setting} and \ref{ass}. We study compactifiactions of $G$ or more precisely good resolutions of $Y$ and $Z$ defined over the Dedekind domain $R$. 

    For a survey of spherical variety, see \cite{Timashev}. First, %show \Cref{ass} is satisfied in 
    consider the characteristic $0$ case. 
    We may choose a complete normal toroidal $H_K\times \b{G}_{m,K}$-spherical variety $\tilde{Z}'_K$ containing $G_K\times \b{G}_{m,K}$ as an open $H_K\times \b{G}_{m,K}$-orbit. Let $Z''_K$ be the Zariski closure of $G_K$ in $Z_K\times \tilde{Z}'_K$, and let $\tilde{Z}''_K$ be the normalization of $Z''_K$. Then $\tilde{Z}''_K$ is a toroidal $H_K\times \b{G}_{m,K}$-spherical variety, which is proper over $Z''_K$. Refining the fan of $\tilde{Z}''_K$, we get a smooth toroidal $H_K\times \b{G}_{m,K}$-spherical variety $\tilde{Z}_K$ together with an $H_K\times \b{G}_{m,K}$-equivariant birational proper morphism $\tilde{Z}_K\to Z_K$. 

    Let $D_1, \cdots, D_s$ be all the $B_K$-stable but not $H_K$-stable prime divisors of $\tilde{Y}_K$. By the Bruhat decomposition, each $D_i$ is of the form $\overline{B^{-}w_iB}$ where $w_i$ is the element in the Weyl group corresponding to a positive root. Let $\delta=\sum D_i$, $\tilde{Y}_{0,K}=\tilde{Y}_K-\delta$. We have $\tilde{Y}_{0,K}\cap G_K=B^-_KB_K$. 

\begin{lem}
	$\{g\in H_K:g\delta=\delta\}$ is the parabolic subgroup $P_K=B_K\times B^{-}_K$.
\end{lem}
\begin{proof}
	%Let $(g,h)\in P$, i.e., $h^{-1}B^{-}Bg=B^{-}B$. We need to show $(g,h)\in B\times B^{-}$. Again by the Bruhat decomposition, we may assume $g, h\in W$. Since $h^{-1}g\in B^{-}B$, $g=h$.
	First, let us show that if $g\in G_K$ satisfies $gB^{-}_KB_K=B^{-}_KB_K$, then $g\in B^{-}_K$.
	Let $P_0$ be the subgroup of $G_K$ consisting of all such elements. Then $P_0$ contains $B^{-}_K$ and $P_0B_K=B^{-}_KB_K$. Taking intersection with $P_0$, we get $P_0=B_K^{-}(B_K\cap P_0)$. So $B_K\cap P_0$ acts transitively on the flag variety $P_0/B^{-}$. By the Borel fixed point theorem \cite[18.4]{milne}, the connected component of $B\cap P_0$ has a fixed point on $P_0/B^{-}$. So $P_0/B^{-}$ is finite. By \cite[18.54]{milne}, $P_0$ is connected. So $P_0=B^{-}_K$.
	
	Let $(g,h)\in H_K$ such that $h^{-1}B^{-}_KB_Kg=B^{-}_KB_K$. Then for any $b\in B^{-}_K$, $$h^{-1}bhB^{-}_KB_K=h^{-1}bhh^{-1}B^{-}_KB_Kg=h^{-1}B^{-}_KB_Kg=B^{-}_KB_K.$$
	By the above claim, this implies $h^{-1}bh\in B_K$. So $h$ lies in the normalizer of $B_K$, which is just $B_K$ by \cite[18.53]{milne}. Similarly, one can show $g\in B_K$.
\end{proof}

By \cite[Proposition 1 in 2.4]{Brion} or \cite[Theorem 29.1]{Timashev},  there exists a closed $T_K\times T_K$-stable subvariety $S$ 
of $\tilde{Y}_{0,K}$ such that we have a $P_K$-equivariant isomorphism
\[R_{u}(P_K)\times S\rightarrow \tilde{Y}_{0,K},\quad(g,y)\mapsto gy,\]
where $R_u(P_K)$ is the unipotent radical of $P_K$. 
Moreover, $S$ is a toric variety for a quotient torus of $T_K\times T_K$ with 
the same fan as $\tilde{Y}_K$. $\tilde{Y}_{0,K}$ is smooth implies $S$ is smooth. 
We may assume that $S$ contains $1_G$. Then the open orbit of $S$ is $(T_K\times T_K)\cdot1_G=T_K$. Thus $S$ is a toric variety for the torus $T_K$.

    \begin{proof}[Proof of \Cref{ass-true-on-dense}.]
    By the standard passing to limit argument we may assume that the resolution $\tilde{Z}\to Z$ is defined over $R$ and satisfies \Cref{ass-3}(2)
    after replacing $\Spec R$ by an open dense subset. 
    By \cite[\href{https://stacks.math.columbia.edu/tag/0559}{Tag 0559}]{stacksproject}, the geometric irreducibility of the generic fiber implies the geometric irreducibility of the fibers over an open dense subset of $\m{Spec}(R)$.
    Thus the desired conditions are satisfied after deleting finitely many primes. 
    \end{proof}

\begin{thm}[{\cite[6.2.8]{Frob-split}}]\label{app-thm}
    Fix an algebraic closure $\bar{k}$ of the residue field $k$ of $R$. Let $X$ be a $G_{\bar k}$-embedding over $\bar k$ in sense of \cite[6.2.1]{Frob-split}. 
    That is, $X$ is a normal variety equipped with an action of $G_{\bar k} \times G_{\bar k}$ containing $G_{\bar k}$ as an open orbit. 
    Let $f: \tilde{X}\to X$ be a smooth toroidal resolution in the sense of \cite[6.2.5]{Frob-split}. Then $Rf_*\mathcal{\c{O}}_{\tilde{X}}\cong \mathcal{\c{O}}_{X}$, $R^if_*\omega_{\tilde{X}}=0$ for $i>0$. 
\end{thm}

%In the following we assume $Z/R$ is a variety satisfying \Cref{ass}. In particular, 

\begin{cor}\label{cor-G-embed}
    %Same assumptions as in \Cref{app-thm}. 
    Keep the assumption and notation of \Cref{app-thm}. Then $Rf_*\omega_{\tilde{X}}$ is a Cohen-Macaulay $\mathcal{\c{O}}_{X}$-module. 
    Suppose $X$ is smooth, then we have $Rf_*\omega_{\tilde{X}}\cong \omega_{X}$. %In general  % pure of %with $\dim \m{Supp}=\m{depth}=\dim \tilde{X}$. . 
\end{cor}
\begin{proof}
    First assume $X$ is smooth. 
    By the relative duality theorem \cite[\href{https://stacks.math.columbia.edu/tag/0A9Q}{Tag 0A9Q}]{stacksproject}, we have
    \[Rf_*\omega_{\tilde{X}}\cong R\s{H}\m{om}_{\mathcal{\c{O}}_{X}}(Rf_*\mathcal{\c{O}}_{\tilde{X}}, \omega_{X})\cong \omega_{X}.\]

    The second claim follows from \cite[I. \S 3. p.49. Theorem]{toroidal-embed}. More precisely, %this is a local problem so we may assume $$
    for any affine open subset $U\subseteq X$, choose an embedding $g:U\to S$ into a smooth $\bar{k}$-scheme. By the relative duality theorem \cite[\href{https://stacks.math.columbia.edu/tag/0A9Q}{Tag 0A9Q}]{stacksproject} again, we have an isomorphism
    \[g_*f_{U,*}\mathcal{\c{O}}_{f^{-1}(U)}%\cong Rg_*Rf_*\mathcal{\c{O}}_{\tilde{X}}\cong R\m{Hom}_{\mathcal{\c{O}}_{S}}(Rg_*Rf_*\omega_{\tilde{X}}, \omega_{S})[\dim S-\dim \tilde{X}].
    \cong R\s{H}\m{om}_{\mathcal{\c{O}}_{S}}(g_*f_{U,*}\mathcal{\c{O}}_{f^{-1}(U)}, \omega_{S})[\dim S-\dim \tilde{X}],\] where $f_U:f^{-1}(U)\to U$ is the restriction of $f$ on $U$.
    In particular, $R^i\s{H}\m{om}_{\mathcal{\c{O}}_{S}}(g_*f_{U,*}\mathcal{\c{O}}_{f^{-1}(U)}, \omega_{S})=0$ for $i\ne \dim S-\dim \tilde{X}$. By \cite[\href{https://stacks.math.columbia.edu/tag/0B5A}{Tag 0B5A}]{stacksproject}, this implies $g_*f_{U,*}\mathcal{\c{O}}_{f^{-1}(U)}$ is Cohen-Macaulay and hence $f_{U,*}\mathcal{\c{O}}_{f^{-1}(U)}$ is a Cohen-Macaulay module. %with $\dim \m{Supp}=\m{depth}=\dim \tilde{X}$. 
\end{proof}

\begin{cor}\label{ass-1_0}\label{ass-2_0} 
    Under \Cref{ass}, we have a smooth resolution $p_0: \tilde{Z}_k\to Z_k$ of $G_k\times \b{G}_{m,k}$-embeddings. Then $Rp_{0,*}\omega_{\tilde{Z}_k}\cong p_{0,*}\omega_{\tilde{Z}_k}$ is a Cohen-Macaulay $\mathcal{\c{O}}_{Z_k}$-module. 
\end{cor}
\begin{proof}
    By \cite[\href{https://stacks.math.columbia.edu/tag/045P}{Tag 045P}]{stacksproject} we may base change to $\bar k$. 
    Let $X_{\bar k}$ be the normarlization of $Z_{\bar k}$. 
    $p_0$ factors through a morphism $p_0': \tilde{Z}_{\bar k}\to X_{\bar k}$.
    Note that by \Cref{ass}(\ref{ass-0}), $X_{\bar k}$ is a $G_{\bar k}$-embedding over $\bar k$ in the sense of \cite[6.2.1]{Frob-split}. 
    By \cite[6.2.5]{Frob-split}, we can choose a smooth toroidal resolution $f: \tilde{X}_{\bar k}\to \tilde{Z}_{\bar k}$. 
    Then $p_{0}'\circ f: \tilde{X}_{\bar k}\to X_{\bar k}$ is also a smooth toroidal resolution. 
    %Replacing ${Z}_{\bar k}$ by its normalization, we may assume ${Z}_{\bar k}$ is normal. 
    By \Cref{cor-G-embed}, we have
    $Rp_{0,*}'\omega_{\tilde{Z}_{\bar k}}\cong R(p_{0}'\circ f)_*\omega_{\tilde{X}_{\bar k}}$ is a Cohen-Macaulay $\mathcal{\c{O}}_{X_{\bar k}}$-module. 
    Note that by \cite[IV Proposition 11]{Serre2}, the direct image of a Cohen-Macaulay module under a finite morphism is Cohen-Macaulay. 
    The claim follows. 
\end{proof}

Notation as in \Cref{ass}. Note that $\omega_{\tilde{Z}}$ is a locally free sheaf on $\tilde{Z}$ and hence flat over $R$. So  $R^0p_*\omega_{\tilde{Z}}$ is flat over $R$. Let $l:Z\to \Spec R$ be the structure morphism. By \cite[6.3.1]{PMIHES_1965__24__5_0} for $x\in Z$ we have
    \[\m{depth}_{x}\,R^0p_*\omega_{\tilde{Z}}=\m{depth}_{l(x)}R\,+\m{depth}_{x}\,(R^0p_*\omega_{\tilde{Z}})_{l(x)}.\]
Since $R$ is Cohen-Macaulay, $R^0p_*\omega_{\tilde{Z}}$ is Cohen-Macaulay at $x$ if and only if $(R^0p_*\omega_{\tilde{Z}})_{l(x)}$ is Cohen-Macaulay at $x$. 

\begin{prop}\label{ass-1}\label{ass-2} 
     $R^ip_*\omega_{\tilde{Z}}=0$ for $i>0$ and $R^0p_*\omega_{\tilde{Z}}$ is Cohen-Macaulay. The latter is equivalent to that $R^0p_*\omega_{\tilde{Z}}$  has Cohen-Macaulay fibers over $R$.
\end{prop}

\begin{proof}
   % First consider the generic fiber case. Note $Z_K$ is a normal spherical variety.  We have $R^ip_*\omega_{\tilde{Y}}=0$ for all $i\geq 1$ by Grauert-Riemenschneider theorem (\cite[2.68]{birational}).    By \cite[Theorem 15.20]{Timashev} $Z_K$ is Cohen-Macaulay and hence has rational singularity.    The sheaf $Rp_*\omega_{\tilde{Z}}\cong R^0p_*\omega_{\tilde{Z}}$ coincides with the dualizing sheaf for $Z$ by \cite[5.70]{birational}, hence Cohen-Macaulay. 
        %follows from the proof of \cite[3.3]{l-adic} or the special fiber case and the fact that Cohen-Macaulay is an open property. 

    %Now consider the special fibers. 
    We may assume $R$ is a complete discrete valuation ring with uniformizer $\unif$. 
    For each integer $n\geq 0$, let $R_n=R/\unif^{n+1} R$, and let $p_n: \tilde{Z}_n\to  Z_n$ be the base change of $p$. Let's prove $R^ip_{n,*}\omega_{\tilde{Z}_n}=0$ for $i>0$ and $R^0p_{n,*}\omega_{\tilde{Z}_{n-1}}\cong R^0p_{n,*}\omega_{\tilde{Z}_{n}}\otimes_{R_n} R_{n-1}$ by induction on $n$.

    When $n=0$, this follows from \Cref{ass-1_0}. Assume $n>0$, we have a short exact sequence 
    \[0\to \omega_{\tilde{Z}_{0}}\xrightarrow{\unif^{n}}\omega_{\tilde{Z}_{n}}\xrightarrow{} \omega_{\tilde{Z}_{n-1}}\to 0\]
    of coherent $\c{O}_{\tilde{Z}_{n}}$-modules. Applying $Rp_*$ to this sequence, we get a long exact sequence 
    \[0\to R^0p_{n,*}\omega_{\tilde{Z}_{0}}\xrightarrow{\unif^{n}} R^0p_{n,*}\omega_{\tilde{Z}_{n}}\to R^0p_{n,*}\omega_{\tilde{Z}_{n-1}}\to R^1p_{n,*}\omega_{\tilde{Z}_{0}}\to \cdots\]
    \[\cdots\to  R^ip_{n,*}\omega_{\tilde{Z}_{0}}\to R^ip_{n,*}\omega_{\tilde{Z}_{n}}\to R^ip_{n,*}\omega_{\tilde{Z}_{n-1}}\to \cdots\]
By the inductive hypothesis $R^ip_{n,*}\omega_{\tilde{Z}_{0}}=R^ip_{n,*}\omega_{\tilde{Z}_{n-1}}=0$ for $i\geq 1$, we conclude that $R^ip_{n,*}\omega_{\tilde{Z}_{n}}=0$ and we have a short exact sequence 
    \[0\to R^0p_{n,*}\omega_{\tilde{Z}_{0}}\xrightarrow{\unif^{n}} R^0p_{n,*}\omega_{\tilde{Z}_{n}}\to R^0p_{n,*}\omega_{\tilde{Z}_{n-1}}\to  0\]
In particular, the natural morphism $R^0p_{n,*}\omega_{\tilde{Z}_{n}}\to R^0p_{n,*}\omega_{\tilde{Z}_{n-1}}\to R^0p_{n,*}\omega_{\tilde{Z}_{0}}$ is surjective. So we may identify the image of $R^0p_{n,*}\omega_{\tilde{Z}_{0}}\xrightarrow{\unif^{n}} R^0p_{n,*}\omega_{\tilde{Z}_{n}}$ with the image of $R^0p_{n,*}\omega_{\tilde{Z}_{n}}\xrightarrow{\unif^{n}} R^0p_{n,*}\omega_{\tilde{Z}_{n}}$. 
We thus have $R^0p_{n,*}\omega_{\tilde{Z}_{n-1}}\cong R^0p_{n,*}\omega_{\tilde{Z}_{n}}\otimes_{R_n} R_{n-1}$. 

By the theorem on formal functions \cite[III 11.1]{Hartshorne},  $\varprojlim_n R^ip_*\omega_{\tilde{Z}_{n}}$ is isomorphic to $(R^ip_*\omega_{\tilde{Z}})^\wedge%\cong Rp_*\omega_{\tilde{Z}}
$. We conclude that $R^ip_*\omega_{\tilde{Z}}=0$ for $i>0$ and $R^0p_*\omega_{\tilde{Z}}\otimes_R k\cong R^0p_{0,*}\omega_{\tilde{Z}_{0}}$. By \Cref{ass-2_0}, the special fiber of $R^0p_*\omega_{\tilde{Z}}$ is Cohen-Macaulay. 

The generic fiber case can be deduced from \cite[Theorem 15.20]{Timashev} and \cite[5.70]{birational} or from the special fiber case. 
Note that $R$ is an excellent ring since it is a Dedekind domain with fraction field of characteristic zero, we get that the Cohen-Macaulay locus of $R^0p_*\omega_{\tilde{Z}}$ is open by \cite[7.8.3(iv)]{PMIHES_1965__24__5_0}. This locus contains the special fiber. Since $Z\to \Spec R$ is proper, this implies that the Cohen-Macaulay locus is the whole space. 
\end{proof}
\end{appendix}

%\addcontentsline{toc}{section}{Reference}
\bibliography{ref}

@article {Popov,
    AUTHOR = {Popov, V. L. },
    TITLE = {Contraction of the actions of reductive algebraic groups},
    JOURNAL = {Math. USSR-Sb.},
    YEAR = {1987},
    NUMBER = {2},
    PAGES = {311-335},
}

@article {l-adic,
    AUTHOR = {Fu, L. and Li, X.},
    TITLE = {Hypergeometric $\ell$-adic sheaves for reductive groups},
    NOTE = {\url{https://arxiv.org/abs/2411.11215v3}},
    YEAR = {2025}
}

@inproceedings{Hotta,
    author = {Hotta, R.},
    title = {Equivariant \text{D}-modules},
    booktitle = {Proceedings of ICPAM Spring School in Wuhan},
    year = {1991},
}

@incollection{K,
  author = {Kapranov, M.},
  title = {Hypergeometric functions on reductive groups},
  booktitle = {Integrable systems and algebraic geometry (Kobe/Kyoto, 1997)},
  publisher = {World Sci. Publ.},
  address = {River Edge, NJ},
  year = {1998},
  pages = {236--281}
}

@incollection {Huyghe2004Trans,
    AUTHOR = {Noot-Huyghe, Christine},
     TITLE = {Transformation de {F}ourier des {$\s{D}$}-modules
              arithm\'{e}tiques. {I}},
 BOOKTITLE = {Geometric aspects of {D}work theory. {V}ol. {I}, {II}},
     PAGES = {857--907},
 PUBLISHER = {Walter de Gruyter, Berlin},
      YEAR = {2004},
      ISBN = {3-11-017478-2},
   MRCLASS = {14F10 (32C38)},
  MRNUMBER = {2099091},
}

@article{Carosurhol,
  author  = {Caro, Daniel},
  title   = {$\mathcal{D}$-modules arithmétiques surholonomes},
  journal = {Annales scientifiques de l'École Normale Supérieure},
  series  = {4},
  year    = {2009},
  volume  = {42},
  number  = {1},
  pages   = {141--192},
  doi     = {10.24033/asens.2092},
  url     = {https://www.numdam.org/articles/10.24033/asens.2092/}
}

@article{Berthelot,
author = {Berthelot, P.},
year = {2002},
month = {01},
pages = {1-80},
title = {Introduction \`a LA théorie arithmétique des \text{D}-modules},
journal = {Asterisque}
}

@inproceedings{Hartshorne,
  title={Algebraic geometry},
  author={Hartshorne, R.},
  booktitle={Graduate texts in mathematics},
  year={1977},
  url={https://api.semanticscholar.org/CorpusID:276502947}
}

@article{Berthelot1,
  author  = {Berthelot, P.},
  title   = {$\mathcal{D}$-modules arithmétiques. {I}. Opérateurs différentiels de niveau fini},
  journal = {Annales scientifiques de l'École Normale Supérieure},
  year    = {1996},
  volume  = {29},
  pages   = {185--272}
}

@article{Abe2013TheoryOW,
      AUTHOR = {Abe, Tomoyuki and Caro, Daniel},
     TITLE = {Theory of weights in {$p$}-adic cohomology},
   JOURNAL = {Amer. J. Math.},
  FJOURNAL = {American Journal of Mathematics},
    VOLUME = {140},
      YEAR = {2018},
    NUMBER = {4},
     PAGES = {879--975},
      ISSN = {0002-9327,1080-6377},
   MRCLASS = {14F10 (14F30)},
  MRNUMBER = {3828038},
MRREVIEWER = {Piotr\ Achinger},
       DOI = {10.1353/ajm.2018.0021},
       URL = {https://doi.org/10.1353/ajm.2018.0021},
}

@misc{stacksproject,
  author       = {The {Stacks project authors}},
  title        = {The Stacks project},
  howpublished = {\url{https://stacks.math.columbia.edu}},
  year         = {2025},
}

@article{Abe2011ExplicitCO,
  title={Explicit calculation of Frobenius isomorphisms and Poincar\'e duality in the theory of arithmetic $\mathscr{D}$-modules},
  author={Tomoyuki Abe},
  journal={arXiv: Algebraic Geometry},
  year={2011},
  url={https://api.semanticscholar.org/CorpusID:119643631}
}

@inproceedings{Monsky1970PadicAA,
  title={P-adic analysis and zeta functions},
  author={Paul Monsky},
  year={1970},
  url={https://api.semanticscholar.org/CorpusID:118071028}
}

@article{Serre,
  author  = {Serre, Jean-Pierre},
  title   = {{Groupes de Grothendieck des sch{\'e}mas en groupes r{\'e}ductifs d{\'e}ploy{\'e}s}},
  journal = {Publications Math{\'e}matiques de l'IH{\'E}S},
  year    = {1968},
  volume  = {34},
  pages   = {37--52},
  note    = {Institut des Hautes {\'E}tudes Scientifiques}
}

@book{Serre2,
  author    = {Jean-Pierre Serre},
  title     = {Alg{\`e}bre Locale --- Multiplicit{\'e}s},
  series    = {Lecture Notes in Mathematics},
  volume    = {11},
  year      = {1975},
  publisher = {Springer-Verlag},
  address   = {Berlin, Heidelberg},
  note      = {ISBN: 978-3-540-37423-9},
  url       = {https://www.springer.com/gp/book/9783540374239}
}

@article{PMIHES_1965__24__5_0,
     author = {Grothendieck, Alexander},
     title = {{\'El\'ements de g\'eom\'etrie alg\'ebrique : {IV.} {\'Etude} locale des sch\'emas et des morphismes de sch\'emas, {Seconde} partie}},
     journal = {Publications Math\'ematiques de l'IH\'ES},
     pages = {5--231},
     publisher = {Institut des Hautes \'Etudes Scientifiques},
     volume = {24},
     year = {1965},
     zbl = {0135.39701},
     language = {fr},
     url = {https://www.numdam.org/item/PMIHES_1965__24__5_0/}
}

@book{Berthelot2,
     author = {Berthelot, P.},
     title = {$\mathcal {D}$-modules arithm\'etiques. {II} : descente par {Frobenius}},
     series = {M\'emoires de la Soci\'et\'e Math\'ematique de France},
     publisher = {Soci\'et\'e math\'ematique de France},
     number = {81},
     year = {2000},
     doi = {10.24033/msmf.394},
     mrnumber = {2001k:14043},
     zbl = {0948.14017},
     language = {fr},
     url = {https://www.numdam.org/item/MSMF_2000_2_81__1_0/}
}

@article{Caro-L,
  author  = {Caro, Daniel},
  title   = {Fonctions {$L$} associ{\'e}es aux $\mathcal{D}$-modules arithm{\'e}tiques. {Cas des courbes}},
  journal = {Compositio Mathematica},
  year    = {2006},
  volume  = {142},
  pages   = {169--206}
}

@article {padictorus,
    AUTHOR = {Fu, Lei and Li, Peigen and Wan, Daqing and Zhang, Hao},
     TITLE = {{$p$}-adic {GKZ} hypergeometric complex},
   JOURNAL = {Math. Ann.},
  FJOURNAL = {Mathematische Annalen},
    VOLUME = {387},
      YEAR = {2023},
    NUMBER = {3-4},
     PAGES = {1629--1689},
      ISSN = {0025-5831,1432-1807},
   MRCLASS = {14F30 (11T23 14G15 33C70)},
  MRNUMBER = {4657433},
MRREVIEWER = {Noriko\ Yui},
       DOI = {10.1007/s00208-022-02491-9},
       URL = {https://doi.org/10.1007/s00208-022-02491-9},
}

@book{milne,
  author    = {Milne, J. S.},
  title     = {Algebraic Groups},
  publisher = {Cambridge University Press},
  year      = {2017},
  address   = {Cambridge}
}

@book {NonArchimedean,
    AUTHOR = {Bosch, S. and G\"untzer, U. and Remmert, R.},
     TITLE = {Non-{A}rchimedean analysis},
    SERIES = {Grundlehren der mathematischen Wissenschaften [Fundamental
              Principles of Mathematical Sciences]},
    VOLUME = {261},
      NOTE = {A systematic approach to rigid analytic geometry},
 PUBLISHER = {Springer-Verlag, Berlin},
      YEAR = {1984},
     PAGES = {xii+436},
      ISBN = {3-540-12546-9},
   MRCLASS = {32K10 (30G05 46P05)},
  MRNUMBER = {746961},
MRREVIEWER = {W.\ Bartenwerfer},
       DOI = {10.1007/978-3-642-52229-1},
       URL = {https://doi.org/10.1007/978-3-642-52229-1},
}

@book{Timashev,
  author    = {Timashev, D. A.},
  title     = {Homogeneous Spaces and Equivariant Embeddings},
  series    = {Encyclopaedia of Mathematical Sciences},
  volume    = {138},
  publisher = {Springer},
  year      = {2011},
  address   = {Berlin, Heidelberg}
}

@online{Brion,
  author  = {Brion, Michel},
  title   = {Vari{\'e}t{\'e}s sph{\'e}riques},
  year    = {1997},
  url     = {http://www-fourier.ujf-grenoble.fr/~mbrion/spheriques.pdf},
  urldate = {2023-10-05}
}

@book{birational,
  author    = {Koll{\'a}r, J{\'a}nos and Mori, Shigefumi},
  title     = {Birational Geometry of Algebraic Varieties},
  publisher = {Cambridge University Press},
  year      = {1998},
  address   = {Cambridge},
  series    = {Cambridge Tracts in Mathematics},
  volume    = {134},
  isbn      = {978-0-521-63277-5},
  doi       = {10.1017/CBO9780511662560}
}

@article{surcoherent,
     author = {Caro, Daniel},
     title = {$\mathcal {D} $-modules arithm\'etiques surcoh\'erents. {Application} aux fonctions $L$},
     journal = {Annales de l'Institut Fourier},
     pages = {1943--1996},
     publisher = {Association des Annales de l{\textquoteright}institut Fourier},
     volume = {54},
     number = {6},
     year = {2004},
     doi = {10.5802/aif.2072},
     mrnumber = {2134230},
     zbl = {02162447},
     language = {fr},
     url = {https://www.numdam.org/articles/10.5802/aif.2072/}
}

@article {Caro19padic,
    AUTHOR = {Caro, Daniel},
    TITLE = {Arithmetic $\s{D}$-modules over algebraic varieties of characteristic $p>0$},
    year = {2019},
    NOTE = {\url{https://arxiv.org/abs/1909.08432v3}},
}

@article{Ddaggeraffine,
     author = {Noot-Huyghe, Christine},
     title = {{$D^\dagger (\infty )$}-affinit\'e des sch\'emas projectifs},
     journal = {Annales de l'Institut Fourier},
     pages = {913--956},
     publisher = {Association des Annales de l{\textquoteright}institut Fourier},
     volume = {48},
     number = {4},
     year = {1998},
     doi = {10.5802/aif.1643},
     mrnumber = {2000a:14019},
     zbl = {0910.14005},
     language = {fr},
     url = {https://www.numdam.org/articles/10.5802/aif.1643/}
}

@Inbook{Berthelot1990,
author="Berthelot, P.",
editor="Baldassarri, Francesco
and Bosch, Siegfried
and Dwork, Bernard",
title="Cohomologie rigide et th{\'e}orie des $\s{D}$-modules",
bookTitle="p-adic Analysis: Proceedings of the International Conference held in Trento, Italy, May 29--June 2, 1989",
year="1990",
publisher="Springer Berlin Heidelberg",
address="Berlin, Heidelberg",
pages="80--124",
isbn="978-3-540-46906-3",
doi="10.1007/BFb0091135",
url="https://doi.org/10.1007/BFb0091135"
}

@book{toroidal-embed,
  author    = {Kempf, G. and Knudsen, F. and Mumford, D. and Saint-Donat, B.},
  title     = {Toroidal Embeddings {I}},
  series    = {Lecture Notes in Mathematics},
  volume    = {339},
  publisher = {Springer},
  year      = {1973},
  address   = {Berlin, Heidelberg}
}

@book{Frob-split,
  author    = {Brion, M. and Kumar, S.},
  title     = {Frobenius Splitting Methods in Geometry and Representation Theory},
  series    = {Progress in Mathematics},
  volume    = {231},
  publisher = {Birkh{\"a}user},
  address   = {Boston},
  year      = {2005}
}

@book {logtangent,
    AUTHOR = {Deligne, Pierre},
     TITLE = {Equations diff\'erentielles \`a{} points singuliers
              r\'eguliers},
    SERIES = {Lecture Notes in Mathematics},
    VOLUME = {Vol. 163},
 PUBLISHER = {Springer-Verlag, Berlin-New York},
      YEAR = {1970},
     PAGES = {iii+133},
   MRCLASS = {14D05 (14C30)},
  MRNUMBER = {417174},
MRREVIEWER = {Helmut\ Hamm},
}

@article {charporbit,
    AUTHOR = {Rittatore, A.},
     TITLE = {Algebraic monoids and group embeddings},
   JOURNAL = {Transform. Groups},
  FJOURNAL = {Transformation Groups},
    VOLUME = {3},
      YEAR = {1998},
    NUMBER = {4},
     PAGES = {375--396},
      ISSN = {1083-4362,1531-586X},
   MRCLASS = {14L30 (20M99)},
  MRNUMBER = {1657536},
MRREVIEWER = {Lex\ Renner},
       DOI = {10.1007/BF01234534},
       URL = {https://doi.org/10.1007/BF01234534},
}

@book {toricorbit,
    AUTHOR = {Gelfand, I. M. and Kapranov, M. M. and Zelevinsky, A.
              V.},
     TITLE = {Discriminants, resultants, and multidimensional determinants},
    SERIES = {Mathematics: Theory \& Applications},
 PUBLISHER = {Birkh\"auser Boston, Inc., Boston, MA},
      YEAR = {1994},
     PAGES = {x+523},
      ISBN = {0-8176-3660-9},
   MRCLASS = {14N05 (13D25 14M25 15A69 33C70 52B20)},
  MRNUMBER = {1264417},
MRREVIEWER = {I.\ Dolgachev},
       DOI = {10.1007/978-0-8176-4771-1},
       URL = {https://doi.org/10.1007/978-0-8176-4771-1},
}

\end{document}